\pdfoutput=1
\UseRawInputEncoding
\documentclass[12pt]{amsart}
\usepackage{oldgerm}
\usepackage{latexsym}
\usepackage{amsmath}
\usepackage{amscd}
\usepackage{youngtab}
\usepackage{amssymb} 
\usepackage{amsmath}
\usepackage{amsthm}
\usepackage{latexsym}
\usepackage[mathscr]{eucal}
\usepackage{amsbsy}

\makeatletter
\@namedef{subjclassname@2020}{%
  \textup{2020} Mathematics Subject Classification}
\makeatother

\newlength{\margins}
\usepackage[top=\margins,bottom=\margins,left=\margins,right=\margins]{geometry}
\usepackage{xcolor}

\begin{document}
\title [Harder's conjecture I] { Harder's conjecture I}  

\author[H. Atobe]{Hiraku ATOBE}
\address{Hiraku  Atobe, Department of Mathematics, Hokkaido University,
Kita 10, Nishi 8, Kita-Ku, Sapporo, Hokkaido, 060-0810, Japan}
\email{atobe@math.sci.hokudai.ac.jp}
\author[M. Chida]{ Masataka CHIDA}
\address{Masataka Chida, School of Science and Technology for Future Life, Tokyo Denki University,
5 Senju Asahi-cho, Adachi-ku, Tokyo 120-8551, Japan}
\email{chida@mail.dendai.ac.jp}
\author [T. Ibukiyama] {Tomoyoshi IBUKIYAMA} 
\address{Tomoyoshi Ibukiyama, Department of Mathematics, Graduate School of Science, Osaka University, 
Machikaneyama 1-1, Toyonaka, Osaka, 560-0043 Japan}
\email{ibukiyam@math.sci.osaka-u.ac.jp}
\author[H. Katsurada]{Hidenori KATSURADA}
\address{Hidenori Katsurada, Muroran Institute of Technology, Mizumoto 27-1, Muroran, 050-8585  Japan}
\email{hidenori@mmm.muroran-it.ac.jp}
\author[T. Yamauchi]{Takuya YAMAUCHI}
\address{Takuya Yamauchi, Mathematical Institute, Tohoku University,  6-3,Aoba, Aramaki, Aoba-Ku, Sendai 980-8578, JAPAN}
\email{takuya.yamauchi.c3@tohoku.ac.jp}

\thanks{Atobe was supported by JSPS KAKENHI Grant Number 19K14494. Chida was supported by JSPS KAKENHI Grant Number JP18K03202. Ibukiyama was supported by JSPS KAKENHI Grant Number JP19K03424. Katsurada was  supported by KAKENHI Grant Number 16H03919.  Yamauchi was supported by JSPS KAKENHI Grant Number .19H01778. }

\date{April 27, 2022}

\maketitle
\newcommand{\Bell}{\ensuremath{\boldsymbol\ell}}


\newcommand{\alp}{\alpha}
\newcommand{\bet}{\beta}
\newcommand{\gam}{\gamma}
\newcommand{\del}{\delta}
\newcommand{\eps}{\epsilon}
\newcommand{\zet}{\zeta}
\newcommand{\tht}{\theta}
\newcommand{\iot}{\iota}
\newcommand{\kap}{\kappa}
\newcommand{\lam}{\lambda}
\newcommand{\sig}{\sigma}
\newcommand{\ups}{\upsilon}
\newcommand{\ome}{\omega}
\newcommand{\vep}{\varepsilon}
\newcommand{\vth}{\vartheta}
\newcommand{\vpi}{\varpi}
\newcommand{\vrh}{\varrho}
\newcommand{\vsi}{\varsigma}
\newcommand{\vph}{\varphi}
\newcommand{\Gam}{\Gamma}
\newcommand{\Del}{\Delta}
\newcommand{\Tht}{\Theta}
\newcommand{\Lam}{\Lambda}
\newcommand{\Sig}{\Sigma}
\newcommand{\Ups}{\Upsilon}
\newcommand{\Ome}{\Omega}

\def\depth{{\rm depth}}
\def\bbD{{\Bbb D}}
\def\p{{\partial}}

\newcommand{\frka}{{\mathfrak a}}    \newcommand{\frkA}{{\mathfrak A}}
\newcommand{\frkb}{{\mathfrak b}}    \newcommand{\frkB}{{\mathfrak B}}
\newcommand{\frkc}{{\mathfrak c}}    \newcommand{\frkC}{{\mathfrak C}}
\newcommand{\frkd}{{\mathfrak d}}    \newcommand{\frkD}{{\mathfrak D}}
\newcommand{\frke}{{\mathfrak e}}    \newcommand{\frkE}{{\mathfrak E}}
\newcommand{\frkf}{{\mathfrak f}}    \newcommand{\frkF}{{\mathfrak F}}
\newcommand{\frkg}{{\mathfrak g}}    \newcommand{\frkG}{{\mathfrak G}}
\newcommand{\frkh}{{\mathfrak h}}    \newcommand{\frkH}{{\mathfrak H}}
\newcommand{\frki}{{\mathfrak i}}    \newcommand{\frkI}{{\mathfrak I}}
\newcommand{\frkj}{{\mathfrak j}}    \newcommand{\frkJ}{{\mathfrak J}}
\newcommand{\frkk}{{\mathfrak k}}    \newcommand{\frkK}{{\mathfrak K}}
\newcommand{\frkl}{{\mathfrak l}}    \newcommand{\frkL}{{\mathfrak L}}
\newcommand{\frkm}{{\mathfrak m}}    \newcommand{\frkM}{{\mathfrak M}}
\newcommand{\frkn}{{\mathfrak n}}    \newcommand{\frkN}{{\mathfrak N}}
\newcommand{\frko}{{\mathfrak o}}    \newcommand{\frkO}{{\mathfrak O}}
\newcommand{\frkp}{{\mathfrak p}}    \newcommand{\frkP}{{\mathfrak P}}
\newcommand{\frkq}{{\mathfrak q}}    \newcommand{\frkQ}{{\mathfrak Q}}
\newcommand{\frkr}{{\mathfrak r}}    \newcommand{\frkR}{{\mathfrak R}}
\newcommand{\frks}{{\mathfrak s}}    \newcommand{\frkS}{{\mathfrak S}}
\newcommand{\frkt}{{\mathfrak t}}    \newcommand{\frkT}{{\mathfrak T}}
\newcommand{\frku}{{\mathfrak u}}    \newcommand{\frkU}{{\mathfrak U}}
\newcommand{\frkv}{{\mathfrak v}}    \newcommand{\frkV}{{\mathfrak V}}
\newcommand{\frkw}{{\mathfrak w}}    \newcommand{\frkW}{{\mathfrak W}}
\newcommand{\frkx}{{\mathfrak x}}    \newcommand{\frkX}{{\mathfrak X}}
\newcommand{\frky}{{\mathfrak y}}    \newcommand{\frkY}{{\mathfrak Y}}
\newcommand{\frkz}{{\mathfrak z}}    \newcommand{\frkZ}{{\mathfrak Z}}


\newcommand{\bfa}{{\mathbf a}}    \newcommand{\bfA}{{\mathbf A}}
\newcommand{\bfb}{{\mathbf b}}    \newcommand{\bfB}{{\mathbf B}}
\newcommand{\bfc}{{\mathbf c}}    \newcommand{\bfC}{{\mathbf C}}
\newcommand{\bfd}{{\mathbf d}}    \newcommand{\bfD}{{\mathbf D}}
\newcommand{\bfe}{{\mathbf e}}    \newcommand{\bfE}{{\mathbf E}}
\newcommand{\bff}{{\mathbf f}}    \newcommand{\bfF}{{\mathbf F}}
\newcommand{\bfg}{{\mathbf g}}    \newcommand{\bfG}{{\mathbf G}}
\newcommand{\bfh}{{\mathbf h}}    \newcommand{\bfH}{{\mathbf H}}
\newcommand{\bfi}{{\mathbf i}}    \newcommand{\bfI}{{\mathbf I}}
\newcommand{\bfj}{{\mathbf j}}    \newcommand{\bfJ}{{\mathbf J}}
\newcommand{\bfk}{{\mathbf k}}    \newcommand{\bfK}{{\mathbf K}}
\newcommand{\bfl}{{\mathbf l}}    \newcommand{\bfL}{{\mathbf L}}
\newcommand{\bfm}{{\mathbf m}}    \newcommand{\bfM}{{\mathbf M}}
\newcommand{\bfn}{{\mathbf n}}    \newcommand{\bfN}{{\mathbf N}}
\newcommand{\bfo}{{\mathbf o}}    \newcommand{\bfO}{{\mathbf O}}
\newcommand{\bfp}{{\mathbf p}}    \newcommand{\bfP}{{\mathbf P}}
\newcommand{\bfq}{{\mathbf q}}    \newcommand{\bfQ}{{\mathbf Q}}
\newcommand{\bfr}{{\mathbf r}}    \newcommand{\bfR}{{\mathbf R}}
\newcommand{\bfs}{{\mathbf s}}    \newcommand{\bfS}{{\mathbf S}}
\newcommand{\bft}{{\mathbf t}}    \newcommand{\bfT}{{\mathbf T}}
\newcommand{\bfu}{{\mathbf u}}    \newcommand{\bfU}{{\mathbf U}}
\newcommand{\bfv}{{\mathbf v}}    \newcommand{\bfV}{{\mathbf V}}
\newcommand{\bfw}{{\mathbf w}}    \newcommand{\bfW}{{\mathbf W}}
\newcommand{\bfx}{{\mathbf x}}    \newcommand{\bfX}{{\mathbf X}}
\newcommand{\bfy}{{\mathbf y}}    \newcommand{\bfY}{{\mathbf Y}}
\newcommand{\bfz}{{\mathbf z}}    \newcommand{\bfZ}{{\mathbf Z}}


\newcommand{\cala}{{\mathcal A}}
\newcommand{\calb}{{\mathcal B}}
\newcommand{\calc}{{\mathcal C}}
\newcommand{\cald}{{\mathcal D}}
\newcommand{\cale}{{\mathcal E}}
\newcommand{\calf}{{\mathcal F}}
\newcommand{\calg}{{\mathcal G}}
\newcommand{\calh}{{\mathcal H}}
\newcommand{\cali}{{\mathcal I}}
\newcommand{\calj}{{\mathcal J}}
\newcommand{\calk}{{\mathcal K}}
\newcommand{\call}{{\mathcal L}}
\newcommand{\calm}{{\mathcal M}}
\newcommand{\caln}{{\mathcal N}}
\newcommand{\calo}{{\mathcal O}}
\newcommand{\calp}{{\mathcal P}}
\newcommand{\calq}{{\mathcal Q}}
\newcommand{\calr}{{\mathcal R}}
\newcommand{\cals}{{\mathcal S}}
\newcommand{\calt}{{\mathcal T}}
\newcommand{\calu}{{\mathcal U}}
\newcommand{\calv}{{\mathcal V}}
\newcommand{\calw}{{\mathcal W}}
\newcommand{\calx}{{\mathcal X}}
\newcommand{\caly}{{\mathcal Y}}
\newcommand{\calz}{{\mathcal Z}}


\newcommand{\scra}{{\mathscr A}}
\newcommand{\scrb}{{\mathscr B}}
\newcommand{\scrc}{{\mathscr C}}
\newcommand{\scrd}{{\mathscr D}}
\newcommand{\scre}{{\mathscr E}}
\newcommand{\scrf}{{\mathscr F}}
\newcommand{\scrg}{{\mathscr G}}
\newcommand{\scrh}{{\mathscr H}}
\newcommand{\scri}{{\mathscr I}}
\newcommand{\scrj}{{\mathscr J}}
\newcommand{\scrk}{{\mathscr K}}
\newcommand{\scrl}{{\mathscr L}}
\newcommand{\scrm}{{\mathscr M}}
\newcommand{\scrn}{{\mathscr N}}
\newcommand{\scro}{{\mathscr O}}
\newcommand{\scrp}{{\mathscr P}}
\newcommand{\scrq}{{\mathscr Q}}
\newcommand{\scrr}{{\mathscr R}}
\newcommand{\scrs}{{\mathscr S}}
\newcommand{\scrt}{{\mathscr T}}
\newcommand{\scru}{{\mathscr U}}
\newcommand{\scrv}{{\mathscr V}}
\newcommand{\scrw}{{\mathscr W}}
\newcommand{\scrx}{{\mathscr X}}
\newcommand{\scry}{{\mathscr Y}}
\newcommand{\scrz}{{\mathscr Z}}


\newcommand{\AAA}{{\mathbb A}} 
\newcommand{\BB}{{\mathbb B}}
\newcommand{\CC}{{\mathbb C}}
\newcommand{\DD}{{\mathbb D}}
\newcommand{\EE}{{\mathbb E}}
\newcommand{\FF}{{\mathbb F}}
\newcommand{\GG}{{\mathbb G}}
\newcommand{\HH}{{\mathbb H}}
\newcommand{\II}{{\mathbb I}}
\newcommand{\JJ}{{\mathbb J}}
\newcommand{\KK}{{\mathbb K}}
\newcommand{\LL}{{\mathbb L}}
\newcommand{\MM}{{\mathbb M}}
\newcommand{\NN}{{\mathbb N}}
\newcommand{\OO}{{\mathbb O}}
\newcommand{\PP}{{\mathbb P}}
\newcommand{\QQ}{{\mathbb Q}}
\newcommand{\RR}{{\mathbb R}}
\newcommand{\SSS}{{\mathbb S}} 
\newcommand{\TT}{{\mathbb T}}
\newcommand{\UU}{{\mathbb U}}
\newcommand{\VV}{{\mathbb V}}
\newcommand{\WW}{{\mathbb W}}
\newcommand{\XX}{{\mathbb X}}
\newcommand{\YY}{{\mathbb Y}}
\newcommand{\ZZ}{{\mathbb Z}}


\newcommand{\tta}{\hbox{\tt a}}    \newcommand{\ttA}{\hbox{\tt A}}
\newcommand{\ttb}{\hbox{\tt b}}    \newcommand{\ttB}{\hbox{\tt B}}
\newcommand{\ttc}{\hbox{\tt c}}    \newcommand{\ttC}{\hbox{\tt C}}
\newcommand{\ttd}{\hbox{\tt d}}    \newcommand{\ttD}{\hbox{\tt D}}
\newcommand{\tte}{\hbox{\tt e}}    \newcommand{\ttE}{\hbox{\tt E}}
\newcommand{\ttf}{\hbox{\tt f}}    \newcommand{\ttF}{\hbox{\tt F}}
\newcommand{\ttg}{\hbox{\tt g}}    \newcommand{\ttG}{\hbox{\tt G}}
\newcommand{\tth}{\hbox{\tt h}}    \newcommand{\ttH}{\hbox{\tt H}}
\newcommand{\tti}{\hbox{\tt i}}    \newcommand{\ttI}{\hbox{\tt I}}
\newcommand{\ttj}{\hbox{\tt j}}    \newcommand{\ttJ}{\hbox{\tt J}}
\newcommand{\ttk}{\hbox{\tt k}}    \newcommand{\ttK}{\hbox{\tt K}}
\newcommand{\ttl}{\hbox{\tt l}}    \newcommand{\ttL}{\hbox{\tt L}}
\newcommand{\ttm}{\hbox{\tt m}}    \newcommand{\ttM}{\hbox{\tt M}}
\newcommand{\ttn}{\hbox{\tt n}}    \newcommand{\ttN}{\hbox{\tt N}}
\newcommand{\tto}{\hbox{\tt o}}    \newcommand{\ttO}{\hbox{\tt O}}
\newcommand{\ttp}{\hbox{\tt p}}    \newcommand{\ttP}{\hbox{\tt P}}
\newcommand{\ttq}{\hbox{\tt q}}    \newcommand{\ttQ}{\hbox{\tt Q}}
\newcommand{\ttr}{\hbox{\tt r}}    \newcommand{\ttR}{\hbox{\tt R}}
\newcommand{\tts}{\hbox{\tt s}}    \newcommand{\ttS}{\hbox{\tt S}}
\newcommand{\ttt}{\hbox{\tt t}}    \newcommand{\ttT}{\hbox{\tt T}}
\newcommand{\ttu}{\hbox{\tt u}}    \newcommand{\ttU}{\hbox{\tt U}}
\newcommand{\ttv}{\hbox{\tt v}}    \newcommand{\ttV}{\hbox{\tt V}}
\newcommand{\ttw}{\hbox{\tt w}}    \newcommand{\ttW}{\hbox{\tt W}}
\newcommand{\ttx}{\hbox{\tt x}}    \newcommand{\ttX}{\hbox{\tt X}}
\newcommand{\tty}{\hbox{\tt y}}    \newcommand{\ttY}{\hbox{\tt Y}}
\newcommand{\ttz}{\hbox{\tt z}}    \newcommand{\ttZ}{\hbox{\tt Z}}

\newcommand{\rank}{\mathrm{rank}}

\newcommand{\phm}{\phantom}
\newcommand{\ds}{\displaystyle }
\newcommand{\smallstrut}{\vphantom{\vrule height 3pt }}
\def\bdm #1#2#3#4{\left(
\begin{array} {c|c}{\ds{#1}}
 & {\ds{#2}} \\ \hline
{\ds{#3}\vphantom{\ds{#3}^1}} &  {\ds{#4}}
\end{array}
\right)}
\newcommand{\wtd}{\widetilde }
\newcommand{\bsl}{\backslash }
\newcommand{\GL}{{\mathrm{GL}}}
\newcommand{\SL}{{\mathrm{SL}}}
\newcommand{\GSp}{{\mathrm{GSp}}}
\newcommand{\PGSp}{{\mathrm{PGSp}}}
\newcommand{\SP}{{\mathrm{Sp}}}
\newcommand{\SO}{{\mathrm{SO}}}
\newcommand{\SU}{{\mathrm{SU}}}
\newcommand{\Ind}{\mathrm{Ind}}
\newcommand{\Hom}{{\mathrm{Hom}}}
\newcommand{\Ad}{{\mathrm{Ad}}}
\newcommand{\Sym}{{\mathrm{Sym}}}
\newcommand{\Mat}{\mathrm{M}}
\newcommand{\sgn}{\mathrm{sgn}}
\newcommand{\trs}{\,^t\!}
\newcommand{\iu}{\sqrt{-1}}
\newcommand{\oo}{\hbox{\bf 0}}
\newcommand{\ono}{\hbox{\bf 1}}
\newcommand{\smallcirc}{\lower .3em \hbox{\rm\char'27}\!}
\newcommand{\bAf}{\bA_{\hbox{\eightrm f}}}
\newcommand{\thalf}{{\textstyle{\frac12}}}
\newcommand{\shp}{\hbox{\rm\char'43}}
\newcommand{\Gal}{\operatorname{Gal}}
\newcommand{\ev}{\mathrm{ev}}

\newcommand{\bdel}{{\boldsymbol{\delta}}}
\newcommand{\bchi}{{\boldsymbol{\chi}}}
\newcommand{\bgam}{{\boldsymbol{\gamma}}}
\newcommand{\bome}{{\boldsymbol{\omega}}}
\newcommand{\bpsi}{{\boldsymbol{\psi}}}
\newcommand{\blam}{{\boldsymbol{\lambda}}}
\newcommand{\GK}{\mathrm{GK}}
\newcommand{\EGK}{\mathrm{EGK}}
\newcommand{\MGK}{\mathrm{MGK}}
\newcommand{\ord}{\mathrm{ord}}
\newcommand{\nd}{\mathrm{nd}}
\newcommand{\diag}{\mathrm{diag}}
\newcommand{\ua}{{\underline{a}}}
\newcommand{\ub}{{\underline{b}}}
\newcommand{\uc}{{\underline{c}}}
\newcommand{\ud}{{\underline{d}}}
\newcommand{\ue}{{\underline{e}}}
\newcommand{\uf}{{\underline{f}}}
\newcommand{\ug}{{\underline{g}}}
\newcommand{\uh}{{\underline{h}}}
\newcommand{\ui}{{\underline{i}}}
\newcommand{\uj}{{\underline{j}}}
\newcommand{\uk}{{\underline{k}}}
\newcommand{\ul}{{\underline{l}}}
\newcommand{\um}{{\underline{m}}}
\newcommand{\un}{{\underline{n}}}
\newcommand{\ZZn}{\ZZ_{\geq 0}^n}
\newcommand{\uzet}{{\underline{\zeta}}}
\newcommand{\St}{\mathrm{St}}
\newcommand{\Sp}{\mathrm{Sp}}
\newcommand{\Spin}{\mathrm{Spin}}
\newcommand{\alg}{\mathrm{alg}}

\newtheorem{theorem}{Theorem}[section]
\newtheorem{lemma}[theorem]{Lemma}
\newtheorem{proposition}[theorem]{Proposition}
\newtheorem{corollary}[theorem]{Corollary}
\newtheorem{conjecture}[theorem]{Conjecture}
\newtheorem{definition}[theorem]{Definition}
\newtheorem{remark}[theorem]{{\bf Remark}}

\theoremstyle{plain}
\theoremstyle{definition}
\theoremstyle{remark}%
\newtheorem{condition}{Condition}
%

\begin{abstract}
Let $f$ be a primitive form with respect to $\SL_2(\ZZ)$. 
Then, we propose a conjecture on  the congruence between the Klingen-Eisenstein lift of the Duke-Imamoglu-Ikeda lift of $f$ and a certain lift of a vector valued Hecke eigenform with respect to  $\Sp_2(\ZZ)$. This conjecture implies  Harder's conjecture.
We prove the above conjecture in some cases.
\end{abstract} 

\tableofcontents
                                    
\section{Introduction}
Harder's conjecture is one of the most fascinating conjectures in the arithmetic of automorphic forms. It plays an important role in constructing nontrivial elements of the Bloch-Kato Selmer group attached to a modular form (cf. \cite{Dummigan-Ibukiyama-Katsurada11}). Harder's conjecture predicts that the Hecke eigenvalues of a primitive form for $\SL_2(\ZZ)$ are related with those of a certain Hecke eigenform for $\Sp_2(\ZZ)$ modulo some prime ideal. We explain it more precisely. For a non-increasing sequence  ${\bf k}=(k_1,\ldots,k_n)$ of non-negative integers we denote by $M_{\bf k}(\Sp_n(\ZZ))$ and $S_{\bf k}(\Sp_n(\ZZ))$ the spaces of modular forms and cusp forms of weight ${\bf k}$ (or, weight $k$, if ${\bf k}=(\overbrace{k,\ldots,k}^n)$) for $\Sp_n(\ZZ)$, respectively. (For the definition of modular forms, see Section 2). Let $f(z)=\sum_{m=1}^{\infty}a(m,f)\exp(2\pi \sqrt{-1}mz)$ be a primitive form in $S_{2k+j-2}(\SL_2(\ZZ))$, and suppose that a `big prime' $\frkp$ divides the algebraic part of $L(k+j,f)$.  Then, Harder \cite{Harder03} conjectured that there exists a Hecke eigenform $F$ in $S_{(k+j,k)}(\Sp_2(\ZZ))$ such that
\[\lambda_F(T(p)) \equiv a(p,f)+p^{k-2}+p^{j+k-1} \pmod {\frkp'} \]
for any prime number $p$, where $\lambda_F(T(p))$ is the eigenvalue of the Hecke operator $T(p)$ on $F$, and $\frkp'$ is a prime ideal of $\QQ(f) \cdot \QQ(F)$ lying above $\frkp$. One of main difficulties in treating this congruence  arises from the fact that this is not concerning the congruence between Hecke eigenvalues of two Hecke eigenforms of the same weight. Indeed, the right-hand side of the above is not the Hecke eigenvalue of a Hecke eigenform if $j>0$. Several attempts have been made to overcome this obstacle.  Ibukiyama \cite{Ibukiyama08}, \cite{Ibukiyama14} proposed a half-integral weight version of Harder's conjecture given as  
congruences of Hecke eigenforms and related it to the original Harder's conjecture 
through his conjectural Shimura type correspondence for vector valued Siegel modular forms of degree two
(and this Shimura type conjecture was now proved by H. Ishimoto \cite{Ishimoto20}). 
This explains the Harder conjecture for odd $k$ (\cite{ibuoberwolfach}) and
the proved example of 
congruence in \cite{Ibukiyama14} means the Harder conjecture for $(k,j)=(5,18)$. 
In \cite{Bergstrom-Dummigan16}, Bergstr\"om and Dummigan, among other things, reformulated Harder's conjecture as congruence between a certain induced representation of $\pi_f$ and 
a cuspidal automorphic representation of $\mathrm {GSp}(2)$.
In \cite{Chenevier-Lannes19}, Chenevier and Lannes gave several congruences between theta series of even unimodular lattices, and  using Arthur's endoscopic classification and Galois representation theoretic method, they, among other things, proved Harder's conjecture for $(k,j)=(10,4)$.
In this paper we consider a conjecture concerning the congruence between two liftings to higher degree 
of Hecke eigenforms (of integral weight) of degree two.  More precisely, for the $f$ above, let $\scri_n(f)$ be the Duke-Imamoglu-Ikeda lift of $f$  to the space of cusp forms of weight ${j \over 2}+k+{n \over 2}-1$ for $\Sp_n(\ZZ)$ with 
$n$ even. For a sequence 
\begin{align*}
& {\bf k}=\Bigl(\overbrace{{j \over 2}+k+{n \over 2}-1,\ldots,{ j \over 2}+k+{n \over 2}-1}^n,\overbrace{{j \over 2}+{3n \over 2}+1,\ldots,{j \over 2}+{3n \over 2}+1}^n\Bigr)
\end{align*}
with $k \ge n+2$, let  $[\scri_n(f)]^{{\bf k}}$ be the Klingen-Eisenstein lift of $\scri_n(f)$ to $M_{{\bf k}}(\Sp_{2n}(\ZZ))$. 
 Then, we propose the  following conjecture:
\begin{conjecture} (Conjecture \ref{conj.main-conjecture})
Let $k,j$ and ${\bf k}$ be as above. Let $f(z) \in S_{2k+j-2}(\SL_2(\ZZ))$ be a primitive form and  $\frkp$  a prime ideal of $\QQ(f)$. Then under  certain assumptions,
 there exists a Hecke eigenform $F$ in $S_{(k+j,k)}(\Sp_2(\ZZ))$ such that
\[\lambda_{\scra^{(I)}_{2n}(F)}(T)  \equiv \lambda_{[\scri_n(f)]^{{\bf k}}}(T) \pmod {\frkp'}\]
for any integral Hecke operator $T$. Here, $\scra_{2n}^{(I)}(F)$ is the lift of $F$ to $S_{\bf k}(\Sp_{2n}(\ZZ))$, called the lift of type $\scra^{(I)}$, which will be defined in Theorem \ref{th.atobe1}. (As for the definition of integral Hecke operators, see Section 3.)
\end{conjecture}
This conjecture implies Harder's conjecture (cf. Theorem \ref{th.enhanced-Harder}).

The advantage of this formulation is that one can compare the  Hecke eigenvalues of two Hecke eigenforms. Indeed, by using the same argument as in Katsurada-Mizumoto \cite{Katsurada-Mizumoto12}, under the above assumption, we can prove that there exists a Hecke eigenform $G \in M_{{\bf k}}(\Sp_{2n}(\ZZ))$ such that $G$ is not a constant multiple of $[\scri_n(f)]^{{\bf k}}$ and 
\[\lambda_G(T) \equiv \lambda_{[\scri_n(f)]^{{\bf k}}}(T)\ \pmod {\frkp'}\]
for any integral Hecke operator $T$.  Therefore, to prove the above conjecture, it suffices to show that $G$ is a lift of type $\scra^{(I)}$.
Here we expect that $G$ can be taken as $\scra^{(I)}$ and indeed  we will see that 
in the cases $(k,j)=(10,4),(14,4)$ and $(4,24)$ using the dimension formula due to  Taibi \verb+ https://otaibi.perso.math.cnrs.fr/dimtrace+ and the numerical tables of Hecke eigenvalues due to Poor-Ryan-Yuen \cite{Poor-Ryan-Yuen09} and  Ibukiyama-Katsurada-Poor-Yuen \cite{Ib-Kat-P-Y14}. As a result, we prove Conjecture \ref{conj.main-conjecture} and so Harder's conjecture in those cases.

This paper is organized as follows. In Section 2, we give a brief  summary of  Siegel modular forms, especially about their $\QQ$-structures or $\ZZ$-structures. In Section 3, after giving a summary  of several $L$-values, we state Harder's conjecture. In Section 4,  we introduce several lifts, and among other things define the lift of $\scra^{(I)}$-type of a vector valued modular form in $S_{(k+j,k)}(\Sp_2(\ZZ))$, and propose a conjecture and explain how this conjecture implies Harder's conjecture. 
In Section 5, we consider the pullback formula of the Siegel Eisenstein series with differential operators. In Section 6, we consider the congruence for vector valued Klingen-Eisenstein series, which is a generalization of  \cite{Katsurada-Mizumoto12}, and explain how the assumption that $\frkp$ divides the algebraic part of $L(k+j,f)$ for $f \in S_{2k+j-2}(\SL_2(\ZZ))$ gives the congruence between  $[\scri_n(f)]^{{\bf k}}$ and another Hecke eigenform in $M_{{\bf k}}(\Sp_{2n}(\ZZ))$.  In Section 7, we give a formula for the Fourier coefficients of the Klingen-Eisenstein series, from which we can confirm some assumption 
 in our main results. In Section 9, we state our main results,  which confirm our conjecture, and so  Harder's.  

In a subsequent paper, we will prove Conjecture \ref{conj.main-conjecture} and so Harder's in more general setting, that is, in the case $k$ is even and $j \equiv 0 \text{ mod } 4$, that is, we will prove these conjectures without using the dimension formula or the computation of Hecke eigenvalues of Siegel modular forms  
(cf. \cite{A-C-I-K-Y20}).

\noindent\textbf{Acknowledgments.} The authors thank S. B\"ocherer, G. Chenevier, N. Dummigan, G. Harder, T. Ikeda, N. Kozima and S. Sugiyama for valuable comments. They also thank the referee for a careful and intelligent 
reading of their paper and for the numerous helpful suggestions to improve the exposition.

\indent
{\sc Notation.}  
Let $R$ be a commutative ring. We denote by $R^{\times}$ the unit group of $R$. 
We denote by $M_{mn}(R)$ the set of
$m \times n$-matrices with entries in $R.$ In particular put $M_n(R)=M_{nn}(R).$   Put $\GL_m(R) = \{A \in M_m(R) \ | \ \det A \in R^\times \},$ where $\det
A$ denotes the determinant of a square matrix $A$. For an $m \times n$-matrix $X$ and an $m \times m$-matrix
$A$, we write $A[X] = {}^t\! X A X,$ where $^t\! X$ denotes the
transpose of $X$. Let $\Sym_n(R)$ 
denote
the set of symmetric matrices of degree $n$ with entries in
$R.$ Furthermore, if $R$ is an integral domain of characteristic different from $2,$ let  $\calh_n(R)$ denote the set of half-integral matrices of degree $n$ over $R$, that is, $\calh_n(R)$ is the subset of symmetric
matrices of degree $n$ with entries in the field of fractions of $R$ whose $(i,j)$-component belongs to
$R$ or ${1 \over 2}R$ according as $i=j$ or not.  
 We say that an element $A$ of $M_n(R)$ is non-degenerate if $\det A \not=0$. For a subset $S$ of $M_n(R)$ we denote by $S^{\nd}$ the subset of $S$
consisting of non-degenerate matrices. If $S$ is a subset of $\Sym_n(\RR)$ with $\RR$ the field of real numbers, we denote by $S_{>0}$ (resp. $S_{\ge 0}$) the subset of $S$
consisting of positive definite (resp. semi-positive definite) matrices. The group  
$\GL_n(R)$ acts on the set $\Sym_n(R)$ by 

\hskip 2cm $\GL_n(R) \times \Sym_n(R) \ni (g,A) \longmapsto A[g] \in \Sym_n(R).
$ 

\noindent 
Let $G$ be a subgroup of $\GL_n(R).$ For a $G$-stable subset ${\mathcal B}$ of $\Sym_n(R)$  we denote by $\calb/G$ the set of equivalence classes of $\calb$ under the action of  $G.$ We sometimes use the same symbol $\calb/G$ to denote a complete set of representatives of $\calb/G.$ We abbreviate $\calb/\GL_n(R)$ as $\calb/\!\!\sim$ if there is no fear of confusion. Let $R'$ be a subring of $R$. Then two symmetric matrices $A$ and $A'$ with
entries in $R$ are said to be equivalent over $R'$ with each
other and write $A \sim_{R'} A'$ if there is
an element $X$ of $\GL_n(R')$ such that $A'=A[X].$ We also write $A \sim A'$ if there is no fear of confusion. 
For square matrices $X$ and $Y$ we write $X \bot Y = \begin{pmatrix} X &O \\ O & Y \end{pmatrix}$.

For an integer $D \in \ZZ$ such that $D \equiv 0$ or $\equiv 1 \ {\rm mod} \ 4,$ let ${\textfrak d}_D$ be the discriminant of $\QQ(\sqrt{D}),$ and put ${\textfrak f}_D= \sqrt{{ D \over {\textfrak d}_D}}.$ We call an integer $D$ a fundamental discriminant if it is the discriminant of some quadratic extension of $\QQ$ or $1.$ For a fundamental discriminant $D,$ let $\left({\displaystyle
D \over \displaystyle * }\right)$ be the character corresponding to $\QQ(\sqrt{D})/\QQ.$ Here we make the convention that  $\left({\displaystyle
 D  \over \displaystyle * } \right)=1$ if $D=1.$ For an integer $D$ such that 
 $D \equiv 0$ or $\equiv 1 \ {\rm mod} \ 4,$ we define $\left({\displaystyle
 D  \over \displaystyle * } \right)=\left({\displaystyle
 \frkd_D  \over \displaystyle * } \right)$.  
We put ${\bf e}(x)=\exp(2 \pi \sqrt{-1} x)$ for $x \in \CC,$ and for a prime number $p$ we denote by ${\bf e}_p(*)$ the continuous additive character of $\QQ_p$ such that ${\bf e}_p(x)= {\bf e}(x)$ for $x \in \ZZ[p^{-1}].$

Let $K$ be an algebraic number field, and $\frkO=\frkO_K$ the ring of integers 
in $K$. For a prime ideal $\frkp$ we denote by $K_{\frkp}$ and $\frkO_{\frkp}$ the $\frkp$-adic completion of $K$ and $\frkO$, respectively, and put $\frkO_{(\frkp)}=\frkO_{\frkp} \cap K$.
For a prime ideal number $\frkp$ of $\frkO$, we denote by $\ord_{\frkp}(*)$ the additive valuation of $K_{\frkp}$ normalized so that $\ord_{\frkp}(\vpi)=1$ for a prime element $\vpi$ of $K_{\frkp}$. 
Moreover for any element $a, b \in \frkO_{\frkp}$
we write $b \equiv a \pmod {\frkp}$ if $\ord_{\frkp}(a-b) >0$.

\section{Siegel modular forms}
We denote by $\HH_{n}$ the Siegel upper half 
space of degree $n$, i.e.,
\[
\HH_n=\{Z\in M_{n}(\CC) \ | \ Z=\,^{t}Z=X+\sqrt{-1}Y,\ X,Y\in M_{n}(\RR),Y>0\}.
\]  
For any ring $R$ and any natural integer $n$, we define the
group $\mathrm{GSp}_n(R)$ by
\[
\mathrm{GSp}_n(R)=\{g\in M_{2n}(R) \ | \ gJ_{n}\,^{t}g=\nu(g) J_{n} \text{ with some } \nu(g) \in R^{\times}\},
\]
where $J_{n}=\left(\begin{smallmatrix} 0_{n} & -1_{n} \\ 1_{n} & 0_{n} 
\end{smallmatrix}\right)$.  We call $\nu(g)$ the symplectic 
similitude of $g$. 
We also define the  symplectic group of degree $n$ over $R$ by 
\[
\Sp_n(R)=\{g \in \mathrm{GSp}_n(R) \ | \  \nu(g)=1 \}.
\]
In particular, if $R$ is a subfield of $\RR$,
we define 
\[\mathrm{GSp}_n(R)^+=\{g \in \mathrm{GSp}_n(R) \ | \  \nu(g)>0 \}.
\]
 We put $\varGamma^{(n)}=\Sp_n(\ZZ)$ for the sake of simplicity. 
Now we define vector valued Siegel modular forms of $\varGamma^{(n)}$. 
Let  $(\rho,V)$ be a polynomial representation of $\GL_n(\CC)$ 
on a finite dimensional complex vector space $V$. We fix a Hermitian inner
product $\langle *,\ * \rangle$ on $V$ such that
\begin{align*} \langle \rho(g)v,w\rangle=\langle v,\rho({}^t\bar g)w \rangle \quad \text{ for } g \in \GL_n(\CC), v,w \in V.  \tag{H}
\end{align*}
For  any $V$-valued function $F$ on $\HH_{n}$, 
and for any $g=\left(\begin{smallmatrix} A & B \\ C & D \end{smallmatrix}
 \right)
\in \mathrm{GSp}_n(\RR)^+$, we put $J(g,Z)=CZ+D$ and 
\[
F|_{\rho}[g]=\rho(J(g,Z))^{-1}F(gZ).
\]
For a positive integer $N$, we define the principal congruence subgroup  $\varGamma^{(n)}(N)$  of $\varGamma^{(n)}$ of level $N$  by
\[\varGamma^{(n)}(N)=\{\left(\begin{smallmatrix}A & B \\ C & D \end{smallmatrix}\right) \in \varGamma^{(n)} \ | \ A \equiv D \equiv 1_n, B \equiv C \equiv O_n \text{ mod } N \}.\]
A subgroup $\varGamma$ of $\varGamma^{(n)}$ is said to be  a congruence subgroup if $\varGamma$ contains $\varGamma^{(n)}(N)$ with some $N$.
By definition, $\varGamma^{(n)}(N)$ is a congruence subgroup.
Another example of congruence subgroup is the group $\varGamma^{(n)}_0(N)$ defined by
\[\varGamma^{(n)}_0(N)=\{\left(\begin{smallmatrix}A & B \\ C & D \end{smallmatrix}\right) \in \varGamma^{(n)} \ | \ C \equiv O_n \text{ mod } N \}.\]
Let $\varGamma$ be a congruence subgroup of $\varGamma^{(n)}$. 
We say that $F$ is a holomorphic Siegel modular form of weight $\rho$ 
with respect to $\varGamma$ if $F$ is holomorphic on $\HH$ 
and $F|_{\rho}[\gamma]=F$ for any $\gamma\in \varGamma$ 
(with the extra condition of holomorphy at all the cusps  if $n=1$). 
We denote by $M_{\rho}(\varGamma)$ the space of 
modular forms of weight $\rho$ with respect to $\varGamma$, and  by $S_{\rho}(\varGamma)$ its subspace consisting of cusp forms.  

A modular form $F \in M_{\rho}(\varGamma)$ has the following Fourier expansion
\[F(Z)=\sum_{T \in S_n(\QQ)_{\ge 0}}a(T,F){\bf e}(\mathrm{tr}(TZ)) \quad \text{ with } a(T,F) \in V,\]
where $\mathrm{tr}(T)$ is the trace of a matrix $T$, and 
in particular if $\varGamma=\varGamma^{(n)}$, 
we have 
\[F(Z)=\sum_{T \in \calh_n(\ZZ)_{\ge 0}}a(T,F){\bf e}(\mathrm{tr}(TZ)),\]
and  we have $F \in S_{\rho}(\varGamma^{(n)})$ if and only if we have
$a(T,F)=0$ unless $T$ is positive definite. 
For $F, G \in M_{\rho}(\varGamma)$ the Petersson inner product is defined by
\begin{align*}
&(F,G)_\varGamma= \int_{\varGamma \backslash \HH_n} \langle \rho(\sqrt{Y})F(Z),\rho(\sqrt{Y})G(Z)\rangle \det (Y)^{-n-1} dZ, \tag{P}
\end{align*}
where $Y=\mathrm{\mathrm{Im}}(Z)$ and $\sqrt{Y}$ is a positive definite symmetric matrix such that $\sqrt{Y}^2=Y$.
This integral converges if either $F$ or $G$ belongs to $S_{\rho}(\varGamma)$.
We also define $(F,G)$ as
\[
(F,G)=[\Gamma^{(n)}:\varGamma]^{-1}(F,G)_\varGamma.
\]
Let $\lambda=(k_1,k_2,\ldots)$ be a finite or an  infinite sequence of non-negative integers 
such that $k_i\geq k_{i+1}$ for all $i$ and $k_m=0$ for some $m$.
We call this a dominant integral weight (or the Young diagram). We call the biggest integer $m$ such that 
$k_m \neq 0$ a depth of $\lambda$ and write it by $\depth(\lambda)$. It is well known that 
the set of dominant integral weights 
$\lambda$ with $\depth(\lambda)\leq n$ corresponds bijectively 
to the set of irreducible polynomial representations of the $\GL_n(\CC)$. We denote this representation by $(\rho_{n,\lambda},V_{n,\lambda})$. We also denote it by $(\rho_{\bf k},V_{\bf k})$ with  ${\bf k}=(k_1,\ldots,k_n)$ and call it the  irreducible  polynomial representation of $\GL_n(\CC)$ of highest weight ${\bf k}$. We then set $M_{\bf k}(\varGamma)=M_{\rho_{\bf k}}(\varGamma)$ and 
$S_{\bf k}(\varGamma)=S_{\rho_{\bf k}}(\varGamma)$. We say $F$ is a modular form of weight ${\bf k}$ if it is a modular form of weight $\rho_{\bf k}$. If ${\bf k}=(\overbrace{k,\ldots,k}^n)$, we simply write $M_k(\varGamma)=M_{{\bf k}}(\varGamma)$ and 
$S_k(\varGamma)=S_{{\bf k}}(\varGamma)$. We note that 
\[M_{(k+j,k)}(\varGamma^{(2)})=M_{\det^k \otimes Sym^j}(\varGamma^{(2)}) \quad \text {and }  S_{(k+j,k)}(\varGamma^{(2)})=S_{\det^k \otimes Sym^j}(\varGamma^{(2)}),\]
where $Sym^j$ is the $j$-th symmetric tensor representation of $\GL_2(\CC)$. 
In general, for 
the ${\bf k}=(k_1,\ldots,k_n)$ above, we write  ${\bf k}'
=(k_1-k_n,\ldots,k_{n-1}-k_n,0)$. Then, we have  $\rho_{\bf k} \cong \det^{k_n} \otimes \rho_{{\bf k}'}$ with $(\rho_{{\bf k}'},V_{ {\bf k}'})$ an irreducible polynomial representation of highest weight ${\bf k}'$. Here we understand that $(\rho_{{\bf k}'},V_{{\bf k}'})$ is the trivial representation on $\CC$ if $k_1=\cdots=k_{n-1}=k_n$. Moreover, we may regard an element $F \in M_{\bf k}(\varGamma)$ as a $V_{{\bf k}'}$-valued holomorphic  function on $\HH$  such that 
\[F|_{\det^{k_n} \otimes \rho_{{\bf k}'}}[\gamma]=F\]
 for any $\gamma\in \varGamma$ 
(with the extra condition of holomorphy at all the cusps  if $n=1$). 
For a representation $(\rho,V)$ of $\GL_n(\CC)$, we denote by $\frkF(\HH_n,V)$ the set of Fourier series $F(Z)$ on $\HH_n$ with values  in $V$ of the following form:
\[
F(Z)=\sum_{A \in \calh_n(\ZZ)_{\ge 0}} a(A,F) {\bf e}(\mathrm{tr}(AZ)) \quad (Z \in \HH_n, \ a(A,F) \in V).
\]
For $F(Z) \in \frkF(\HH_n,V)$ and a positive integer $r \le n$ we define $\Phi(F)(Z_1)=\Phi_r^n(F)(Z_1) \ (Z_1 \in \HH_r)$ as
\[\Phi(F)(Z_1)=\lim_{\lambda \longrightarrow \infty} F\Bigl(\begin{pmatrix} Z_1 & O \\ O & \sqrt{-1} \lambda 1_{n-r} \end{pmatrix}\Bigr).\]
We make the convention that $\frkF(\HH_0,V)=V$ and  $\Phi_0^n(F)=a(O_n,F)$. Then, $\Phi(F)$ belongs to $\frkF(\HH_r,V)$. For a representation $(\rho,V)$ of $\GL_n(\CC)$, we  denote by $\widetilde \frkF(\HH_n,V)=\widetilde \frkF(\HH_n,(\rho,V))$ the subset of $\frkF(\HH_n,V)$ consisting of elements $F(Z)$
such that the following condition is satisfied:
\begin{align*} a(A[g],F)=\rho(g)a(A,F) \text{ for any } g \in \GL_n(\CC). \tag{K0} \end{align*}

Now let ${\Bell}=(l_1,\ldots,l_n)$ be a dominant integral weight of length $n$ of depth $m$. Then  we realize  the representation space $V_{\Bell}$ in terms of  bideterminants (cf. \cite{Ibukiyama-Takemori19}).
Let $U=(u_{ij})$ be an $m \times n$ matrix of variables. For a positive integer $a \leq m$ let $\mathcal{SI}_{n,a}$ denote the set of strictly increasing sequences of positive integers not greater than $n$ of length $a$. For each $J=(j_1,\ldots,j_a) \in  \mathcal{SI}_{n,a}$ we define $U_J$ as
\[\begin{vmatrix} u_{1,j_1} & \ldots & u_{1,j_a} \\
\vdots & \ddots & \vdots \\
u_{a,j_1} &\ldots & u_{a,j_a}
\end{vmatrix}.\]
Then we say that a polynomial $P(U)$ in $U$ is a bideterminant of weight $ {\Bell}$ if
$P(U)$ is of the following form:
\[P(U)=\prod_{i=1}^{m} \prod_{j=1}^{l_i-l_{i+1}} U_{J_{ij}},\]
where $(J_{i1},\ldots,J_{i,l_i-l_{i+1}}) \in \mathcal{SI}_{n,i}^{l_i-l_{i+1}}$. 
Here we understand that $\prod_{j=1}^{l_i-l_{i+1}} U_{J_{ij}}=1$ if $l_i=l_{i+1}$.
\noindent
 Let $\mathcal {BD}_{{\Bell}}$ be the set of all bideterminants of weight ${\Bell}$. Here we make the convention that 
$\mathcal {BD}_{{\Bell}}=\{1 \}$ if ${\Bell}=(0,\ldots,0)$. 
For a commutative ring  $R$ and an $R$-algebra $S$ let $S[U]_{{\Bell}}$ denote the $R$-module of   all $S$-linear combinations of $P(U)$ for $P(U) \in  \mathcal {BD}_{{\Bell}}$. 
Then we can define an action of $\GL_n(\CC)$ on $\CC[U]_{\Bell}$ as
\[\GL_n(\CC) \times \CC[U]_{\Bell}\ni (g,P(U)) \mapsto P(U  g) \in \CC[U]_{\Bell},\]
and we can take the $\CC$-vector space $\CC[U]_{\Bell}$ as  a representation space  $V_{\Bell}$ of $\rho_{\Bell}$ under this action. 
Let $m \le n-1$ be a non-negative integer and $U=(u_{ij})$ be an $m \times n$ matrix of variables. Let ${\bf k}=(k_1,\ldots,k_n)$ with $k_1 \ge \cdots \ge k_m > k_{m+1}=\cdots=k_n$ and ${\bf k}'=(k_1-k_{m+1},\ldots,k_m-k_{m+1},\overbrace{0,\ldots,0}^{n-m})$. Here we make the convention that ${\bf k}=(k_1,\ldots,k_1)$ and ${\bf k}'=(0,\ldots,0)$ if $m=0$. Then under this notation and convention,
 $M_{{\bf k}}(\varGamma^{(n)})$ can be regarded as a $\CC$-sub-vector space of $\mathrm{Hol}(\HH_n)[U]_{{\bf k}'}$, where $\mathrm{Hol}(\HH_n)$ denotes the ring of holomorphic functions on $\HH_n$. 
 We sometimes write $F(Z)(U)$ for $F(Z) \in M_{{\bf k}}(\varGamma^{(n)})$.
Moreover, the Fourier expansion of $F(Z) \in M_{{\bf k}}(\varGamma^{(n)})$ can be expressed as
\[F(Z)=\sum_{A \in \calh_n(\ZZ)_{\ge 0}} a(A,F){\bf e}(\mathrm{tr}(AZ)),\]
where $a(A,F)=a(A,F)(U) \in \CC[U]_{{\bf k}'}$. 

Let $r$ be an integer such that $m \le r \le n$ and let ${\bf l}=(k_1,\ldots,k_{r-1},k_r)$ and ${\bf l}'=(k_1-k_{m+1},\ldots,k_r-k_{m+1},\overbrace{0,\ldots,0}^{r-m})$. 
For the $m \times n$ matrix $U$, let $U^{(r)}=(u_{ij})_{1 \le i \le m, 1 \le j \le r}$ and put $W'=\CC[U^{(r)}]_{{\bf l}'}$. Then  we can define a representation $(\tau',W')$ of $\GL_r(\CC)$.   
The representations $(\rho_{{\bf k}'},V_{{\bf k}'}) $ and $(\tau',W')$
satisfy the following conditions:
\begin{itemize}
\item[(K1)] $W' \subset V_{{\bf k}'}$.
\item[(K2)] $\rho_{{\bf k}'}\Bigl( \left(\begin{smallmatrix} g_1 & g_2 \\ O & g_4 \end{smallmatrix} \right)\Bigr)v=\tau'(g_1)v$ 
for $\left(\begin{smallmatrix} g_1& g_2 \\ O & g_4 \end{smallmatrix} \right) \in \GL_n(\CC)$ with $g_1 \in \GL_r(\CC)$ and $v \in W'$.
\item[(K3)] If $v \in V_{{\bf k}'}$ satisfies the condition
\[\rho_{{\bf k}'}\Bigl(\left(\begin{smallmatrix}1_r & O \\ O & h\end{smallmatrix}\right)\Bigr)v=v \text{ for any } h \in \GL_{n-r}(\CC),\]
then $v$ belongs to $W'$.
\end{itemize}
Let $F(Z)=\sum_{A \in \calh_n(\ZZ)_{\ge 0}} a(A,F){\bf e}(\mathrm{tr}(AZ)) \in \frkF(\HH_n,V_{{\bf k}'})$
Then, in a way similar to \cite[(2.3.29)]{Andrianov87},
 we have 
\begin{align*}
F(Z)=\sum_{A_1 \in \calh_r(\ZZ)_{\ge 0}} a\Big(\begin{pmatrix} A_1 & O \\ O & O \end{pmatrix},F\Big){\bf e}(\mathrm{tr}(A_1Z_1)) \ (Z_1 \in  \HH_r).
\end{align*}
Suppose that $F(Z)$ belongs to $\widetilde \frkF(\HH_n,V_{{\bf k}'})$.
Then, by (K0), 
\[\rho_{{\bf k}'}\Bigl(\left(\begin{smallmatrix}1_r & O \\ O & h\end{smallmatrix}\right)\Bigr)\Bigl(a(\left(\begin{smallmatrix} A_1 & O \\ O & O \end{smallmatrix}\right),F)\Bigr)=a(\left(\begin{smallmatrix} A_1 & O \\ O & O \end{smallmatrix}\right),F) \ \text{ for any } h \in \GL_{n-r}(\CC).\]
Hence, by (K3), $a(\left(\begin{smallmatrix} A_1 & O \\ O & O \end{smallmatrix}\right),F)$ belongs to $W'$ for any $A_1 \in \calh_r(\ZZ)_{\ge 0}$. This implies that $\Phi_r^n(F)$ belongs to $\frkF(\HH_r,W')$. We easily see that $\Phi_r^n(F)$ belongs to  $\widetilde \frkF(\HH_r,W')$, and therefore $\Phi_r^n$ sends  $\widetilde \frkF(\HH_n,V_{{\bf k}'})$  to $\widetilde \frkF(\HH_r,W')$. It is easily seen that it induces a mapping from
$M_{\rho}(\varGamma^{(n)})$ to $M_{\tau}(\varGamma^{(r)})$, where $\rho=\det^{k_n} \otimes \rho_{{\bf k}'}$ and $\tau=\det^{k_n} \otimes \tau'$.  Let ${\Delta}_{n,r}$
be the subgroup of ${\varGamma}^{(n)}$ defined by
$${\Delta}_{n,r} := \left\{ \, \begin{pmatrix} * & * \\ 
O_{(n-r,n+r)} & * \\ \end{pmatrix} \in {\varGamma}^{(n)} \, \right\} . $$
For $F \in S_{\tau}(\varGamma^{(r)})$ the Klingen-Eisenstein series $[F]_{\tau}^{\rho}(Z,s)$ of $F$ 
associated to $\rho$ is defined by
\begin{align*} [F]_{\tau}^{\rho}(Z,s):= \sum_{\gamma \in {\Delta}_{n,r}\backslash {\varGamma}^{(n)}} \Bigl({\det \mathrm{Im}(Z) \over \det \mathrm{Im} (\mathrm{pr}_r^n(Z))}\Bigr)^s F(\mathrm{pr}_r^n(Z))|_\rho \gamma .
\end{align*}
Here $\mathrm{pr}_r^n(Z)=Z_1$ for 
$Z=\begin{pmatrix} Z_1&Z_2 \\ {}^tZ_2&Z_4 \end{pmatrix} \in \HH_n$
with $Z_1 \in \HH_r, Z_4 \in \HH_{n-r}, Z_2 \in M_{r,n-r}(\CC)$.
We also write $[F]_{\tau}^{\rho}(Z,s)$ as $[F]_{\bf l}^{\bf k}(Z,s)$ or $[F]^{\bf k}(Z,s)$.

Suppose that $k_n$ is even and $2\mathrm{Re}(s) +k_n >n+r+1$. Then,  $[F]_\tau^\rho(Z,s)$
converges absolutely and
uniformly on $\HH_n$. This is proved by  \cite{Klingen67} in the scalar valued case, and can be proved similarly in general case. If  $[F]^{{\bf k}}(Z,s)$ can be continued holomorphically in the neighborhood of $0$ as a function of $s$, we  put $[F]_\tau^{\rho}(Z)=[F]_\tau^\rho(Z,0)$.
If $[F]_\tau^{\rho}(Z)$ is holomorphic as a function of $Z$, it belongs to 
 $M_{{\bf k}}({\varGamma}^{(n)})$, and we say that it is the Klingen-Eisenstein lift of $F$ to $M_{{\bf k}}({\varGamma}^{(n)})$. 
In particular, if $k_n >n+r+1$, then $[F]_\tau^{\rho}(Z,s)$ is holomorphic at $s=0$ as a function of $s$, and $[F]_\tau^{\rho}(Z,0)$ belongs to $M_{{\bf k}}({\varGamma}^{(n)})$, and $\Phi_{\tau}^{\rho} ([F]_{\tau}^{\rho})=F$.
 We note that $[F]_\tau^{\rho}(Z)$ is not necessarily a holomorphic as a function of $Z$ if $k_n \le n+r+1$.  
 
We define  $E_{n,{\bf k}}(Z,s)$ as
\[E_{n,{\bf k}}(Z,s)=\sum_{\gamma \in \Delta_{n,0} \backslash \varGamma^{(n)}} (\det \mathrm{Im}(Z))^s|_{\rho}\gamma\]
and call it the Siegel-Eisenstein series of weight ${\bf k}$ with respect to $\varGamma^{(n)}$. In particular, if ${\bf k}=(\overbrace{k,\ldots,k}^n)$  with $k$ even we write $E_{n,k}(Z,s)$ for $E_{n,{\bf k}}(Z,s)$.
If $k >0$, then $E_{n,k}(Z,s)$ can be continued meromorphically to the whole $s$-plane as a function of $s$.
Let  ${\bf k}=(\overbrace{k+l,\ldots,k+l}^m,\overbrace{k,\ldots,k}^{n-m})$ such that $k, l \ge 0$, and put $\rho=\det ^k \otimes \rho_{{\bf k}'}$ and
$\tau=\det^k \otimes \rho_{{\bf l}'}$ with ${\bf k}'=(\overbrace{l,\ldots,l}^m,0,\ldots,0)$ and ${\bf l}'=(\overbrace{l,\ldots,l}^m)$.
Then, for $F \in S_{\tau}(\varGamma^{(m)})$ 
we can define the Klingen-Eisenstein series  $[F]_\tau^{\rho_{\bf k}}(Z,s)$ of $F$ associated to $\rho_{\bf k}$ if $k$ is even and $2\mathrm{Re}(s) +k>n+m+1$. We note that 
 $\CC[U^{(m)}]_{{\bf l}'}$ is a subspace of  $\CC[U]_{{\bf k}'}$ spanned by $(\det U^{(m)})^l$,  and hence we have a natural isomorphism
\[
\iota:S_{k+l}(\varGamma^{(m)}) \ni f \mapsto \widetilde f:=f (\det U^{(m)})^l \in 
S_{\tau}(\varGamma^{(m)}). 
\]
We sometimes write $[f]^{\rho_{\bf k}}$ or $[f]^{{\bf k}}$ instead of $[\widetilde f ]_\tau^{\rho_{\bf k}}$ for $f \in S_{k+l}(\varGamma^{(m)})$. 
We state the holomorphy of the Klingen-Eisenstein series.
\begin{proposition} \label{prop.holomorphy-Klingen-Eisenstein}
Let $k$ be an even integer.
\begin{itemize}
\item [(1)] Suppose that $k \ge  (n+1)/2$ and that neither $k=(n+2)/2 \equiv 2 \text{ mod } 4$ nor  $k=(n+3)/2 \equiv 2 \text{ mod } 4$. Then $E_{n,k}(Z)$ belongs to $M_k(\varGamma^{(n)})$.
\item[(2)] Let ${\bf k}=(\overbrace{k+l,\ldots,k+l}^m,\overbrace{k,\ldots,k}^{n-m})$ such that $l \ge 0$ 
and $k > 3m/2+1$ and 
let $f$ be a Hecke eigenform in $S_{k+l}(\varGamma^{(m)})$. 
Then $[f]^{{\bf k}}(Z,s)$ can be continued meromorphically to the whole $s$-plane as a function of $s$, and holomorphic at $s=0$. Moreover suppose that  
 $k > (n+m+3)/2$.
Then $[f]^{{\bf k}}(Z)$ belongs to $M_{{\bf k}}(\varGamma^{(n)})$. 
\end{itemize}
\end{proposition}
\begin{proof} The assertion (1) follows from  \cite[Theorem 17.7]{Shimura00}. The assertion (2) has been proved in the case $l=0$ (cf. \cite{Langlands76}, \cite{Shimura00}).
The case $l>0$ will be proved  in Section 5.  
\end{proof}
Let ${\Bell}=(l_1,\ldots,l_n)$ be a dominant integral weight of length $n$ of depth $m$. 
Let $\widetilde V=\widetilde V_{\Bell}=\QQ[U]_{\Bell}$. Then, $(\rho_{\Bell}|\GL_n(\QQ),\widetilde V)$  is a representation of $\GL_n(\QQ)$, and 
$\widetilde V \otimes \CC=V_{\Bell}$.
We consider a $\ZZ$ structure of $V_{\Bell}$. To do this,  we fix a basis $\cals=\cals_{\Bell}=\{ P \}$ of  $\ZZ[U]_{\Bell}$.
We note here that the bideterminants are not linearly independent over $\ZZ$ and even over $\CC$ in general, so the set $\mathcal {BD}_{\Bell}$ is not necessarily 
a basis of $\ZZ[U]_{{\Bell}}$. Let $R$ be a subring of $\CC$. Since the set $\cals$ is also linearly independent over $\CC$, 
 an element $a$ of $R[U]_{\Bell}$ is uniquely written as
\[a=\sum_{P \in \cals}a_P P \text{ with } a_P \in R.\]
Let $K$ be a number field, and $\frkO$ the ring of integers in $K$.
 For  a prime ideal  $\frkp$ of $\frkO$ and $a=a(U)=\sum_{P \in \cals} a_P P \in K[U]_{\Bell}$
 with $a_P \in K$, define 
\[\ord_{\frkp}(a)= \min_{P \in \cals} \ord_{\frkp}(a_P).\]
We say that $\frkp$ divides $a$ if $\ord_\frkp(a) >0$. 
\begin{remark}
\begin{itemize}
\item[(1)]The definition of $\ord_{\frkp}(*)$ does not depend on the choice of a basis of $L$. We note that $\frkp$ does not divide  $a=a(U)$ if  $\frkp$ does not divide $a(U_0)$ for some 
element $U_0$ of $M_{m,n}(\frkO)$.
\item[(2)] There is no canonical choice of a basis of $\widetilde V$. But several standard choices are known. One of them is a basis associated with the semi-standard Young tableaux (cf. \cite{Green80}). We note that that it is also a basis of $L$. This can be proved by 
 a careful analysis of the proof of \cite[(4.5a)]{Green80} combined with \cite[(4.6a)]{Green80}.
\end{itemize}
\end{remark}
For a subring $R$ of $\CC$, we denote by  $M_{{\bf k}}(\varGamma^{(n)})(R)$ the $R$-submodule of $M_{{\bf k}}(\varGamma^{(n)})$  consisting of all modular forms $F$ such that $a(T,F) \in R[U]_{{\bf k}'}$  for all $T \in \calh_n(\ZZ)_{\ge 0}$.  

We consider tensor products of modular forms, which will be used on and  after Section 5.
Let $n_1$ and $n_2$ be positive integers. Let ${\bf k}_1=(k_1,\ldots,k_m,k_{m+1},\ldots,k_{n_1})$ and ${\bf k}_{2}=(k_1,\ldots,k_m,k_{m+1},\ldots,k_{n_2})$ be non-increasing sequences of integers such that $k_{m} >k_{m+1}=\cdots= k_{n_i}=l$ for $i=1,2$.
Then 
$(\rho_{{\bf k}_1} \otimes \rho_{{\bf k}_2}, V_{1} \otimes V_{2})$ is a representation of $\GL_{n_1}(\CC) \times \GL_{n_2}(\CC)$. Put ${\bf k}_1'=(k_1-l,\ldots,k_m-l,\overbrace{0,\ldots,0}^{n_1-m})$ and
${\bf k}_2'=(k_1-l,\ldots,k_m-l,\overbrace{0,\ldots,0}^{n_2-m})$.
Then, $\rho_{{\bf k}_1} \otimes \rho_{{\bf k}_2}=
(\det^l \otimes \rho_{{\bf k}_1'}) \otimes (\det^l \otimes \rho_{{\bf k}_2'})$
with $(\rho_{\bf {k}_i'},V_{i}')$ a polynomial representation of highest weight ${\bf k}_i'$ for $i=1,2$.
To make our formulation smooth, we sometimes regard 
a modular form of scalar weight $k$ for $\varGamma^{(n)}$ as a function with values in the one-dimensional vector space spanned by $\det U^l$ with a non-negative integer $l \le k$, where $U$ is an $n \times n$ matrix of variables. 
 Let $U_1$ and $U_2$ be $m \times n_1$ and $m \times n_2$ matrices, respectively, of variables and for a commutative ring $R$ and an $R$-algebra $S$ let
\begin{align*}
S[U_1,U_2]_{{\bf k}_1',{\bf k}_2'}&=\Bigl\{\sum_j P_{j}(U_1)P_{j}(U_2) 
\quad (\text{ finite sum }) \text{ with } P_{j} (U_i) \in S[U_i]_{{\bf k}_i'} \ (i=1,2) \Bigr\}.
\end{align*}
Here we make the convention that $P_{j}(U_{i}) \in \langle (\det U_{i})^{k_1-l} \rangle_{\CC}$ if $n_i=m$ and $k_1=\cdots=k_m$ as stated above. 
Then, as a representation space $W=W_{{\bf k}_1',{\bf k}_2'}$ of 
$\rho_{{\bf k}_1'} \otimes \rho_{{\bf k}_2'}$ we can take  $\CC[U_1,U_2]_{{\bf k}_1',{\bf k}_2'}$. 
Let 
\[\widetilde W=\widetilde W_{{\bf k}_1',{\bf k}_2'}=\QQ[U_1,U_2]_{{\bf k}_1',{\bf k}_2'}. \]
Then $\widetilde W \cong \widetilde V_{1}' \otimes \widetilde V_{2}'$ and  $\widetilde W \otimes_{\QQ} \CC=W$.
Let 
\[M=M_{{\bf k}_1',{\bf k}_2'}=\ZZ[U_1,U_2]_{{\bf k}_1',{\bf k}_2'}. \]
We note that 
\[M=\left\{\sum_{P_{\tau_1} \in \mathcal {S}_{{\bf k}_1'}, P_{\tau_2} \in \mathcal{S}_{{\bf k}_2'}}a_{\tau_1,\tau_2}P_{\tau_1}(U_1)P_{\tau_2}(U_2) 
\; \middle| \; a_{\tau_1, \tau_2} \in \ZZ
\right\}.\]
Here we make the convention that $P_{\tau_i}(U_2)=
(\det U_i)^{k_1-l}$ if
$n_i=m$ and $k_1=\cdots=k_m$. Therefore, $M$ is a lattice of $\widetilde W$ and $M \cong L_1 \otimes L_2$ with 
$L_i=\ZZ[U_i]_{{\bf k}_i'} \ (i=1,2).$
Thus  $(\rho_{{\bf k}_1} \otimes \rho_{{\bf k}_2}, V_1 \otimes V_2 )$ has also a $\QQ$-structure and $\ZZ$-structure and we can define $\ord_{\frkp}(a \otimes b)$ for $a \otimes b \in \widetilde  W_K$.  If $\dim_{\CC} V_1=1,$ then we identify $V_1,\widetilde V_1$ and $L_1$ with $\CC,\QQ$ and $\ZZ$, respectively, and for $a, b \in V_1$ and $w \in V_2$, we write
$a \otimes b$ and $a \otimes w$ as $ab$ and  $aw$, respectively through the identifications $V_1 \otimes V_1 \cong V_1$ and $V_1 \otimes V_2 \cong V_2 \otimes V_1 \cong V_2$. 
The tensor product $M_{{\bf k}_1}(\varGamma^{(n_1)}) \otimes M_{{\bf k}_2}(\varGamma^{(n_2)})$ is regarded as a $\CC$-subspace of
$(\mathrm{Hol}(\HH_{n_1}) \otimes \mathrm{Hol}(\HH_{n_2}))[U_1,U_2]_{{\bf k}_1',{\bf k}_2'}$.

\section{Harder's conjecture}
In this section we review several arithmetical properties of 
Hecke eigenvalues and $L$ values of modular forms, then state the original Harder's conjecture in \cite{Harder03}.
In the later section, we will treat a generalized version of the conjecture.
Let ${\bf L}_n={\bf L}(\varGamma^{(n)},\mathrm{GSp}_n^+(\QQ) \cap M_{2n}(\ZZ))$ be the Hecke algebra over $\ZZ$ associated to the Hecke pair $(\varGamma^{(n)},\mathrm{GSp}_n^+(\QQ) \cap M_{2n}(\ZZ))$ and for a subring $R$ of $\CC$ put ${\bf L}_n(R)={\bf L}_n \otimes _{\ZZ}R$. 
For an element  $T=\varGamma^{(n)} g \varGamma^{(n)} \in {\bf L}_n(\CC)$, let 
\[T=\bigsqcup_{i=1} ^r \varGamma^{(n)} g_i\]
be the coset decomposition. Then, for a modular form $F \in M_{{\bf k}}(\varGamma^{(n)})$  we define $F|T$ as
\[F|T=\nu(g)^{k_1+\cdots+k_n-n(n+1)/2}\sum_{i=1}^r F|_{\rho_{\bf k}}g_i.\]
This defines an action of the Hecke algebra ${\bf L}_n(\CC)$ on $M_{{\bf k}}$. The operator $F \mapsto F|T$ with $T \in {\bf L}_n(\CC)$ is called the Hecke operator. 
We say that $F$ is a Hecke eigenform if $F$ is a common eigenfunction of all Hecke operators $T \in {\bf L}_n(\CC)$.
Then we have 
\[F|T =\lambda_F(T) F \text{ with } \lambda_F(T) \in \CC  \text{ for any } T \in {\bf L}_n(\CC).\]
We call $\lambda_F(T)$ the Hecke eigenvalue of $T$ with respect to $F$. 
For a Hecke eigenform $F$ in $M_{{\bf k}}(\varGamma^{(n)})$, we denote by  $\QQ(F)$ the field generated over $\QQ$ by all the Hecke eigenvalues $\lambda_F(T)$ with $T \in {\bf L}_n(\QQ)$ and call it the  Hecke field of $F$. For two Hecke eigenforms $F$ and $G$ we sometimes write $\QQ(F,G)=\QQ(F)\QQ(G)$.
We say that an element $T \in {\bf L}_n(\QQ)$ is integral with respect to  $M_{{\bf k}}(\varGamma^{(n)})$ if $F|T \in M_{{\bf k}}(\varGamma^{(n)})(\ZZ)$ for any $F \in  M_{{\bf k}}(\varGamma^{(n)})(\ZZ)$. 
We denote by ${\bf L}_n^{({\bf k})}$ the subset of ${\bf L}_n(\QQ)$  consisting of all integral elements with respect to  $M_{{\bf k}}(\varGamma^{(n)})$. 
The following proposition can be proved in the same manner as
Proposition 4.2 of \cite{Katsurada08}. 
\begin{proposition} \label{prop.Hecke-stable}
We have  ${\bf L}_n \subset {\bf L}_n^{({\bf k})}$ for any ${\bf k}=(k_1,\ldots,k_n)$ with $k_n \ge n+1$.
\end{proposition}
For  a non-zero rational number $a$, we define an element 
$[a]=[a]_n$ of ${\bf L}_n$ by $[a]_n=\varGamma^{(n)}(a 1_{n})\varGamma^{(n)}.$
For each integer $m$ define an element $T(m)$ of ${\bf L}_n$
by 
$$T(m)=\sum_{d_1,\ldots,d_n,e_1,\ldots,e_n}\varGamma^{(n)}(d_1 \bot \cdots \bot d_n \bot e_1 \bot \cdots \bot e_n)\varGamma^{(n)},$$ where
$d_1,\ldots,d_n,e_1,\ldots,e_n$ run over all positive integer satisfying
$$d_i|d_{i+1}, \ e_{i+1}|e_i \ (i=1,\ldots,n-1), d_n|e_n,d_ie_i=m \
(i=1,\ldots,n).$$ Furthermore, for $i=1,\ldots,n$ and a prime number $p$
 put
$$T_i(p^2)=\varGamma^{(n)}(1_{n-i} \bot p1_i \bot p^21_{n-i} \bot p1_i)\varGamma^{(n)}.$$
As is well known, ${\bf L}_n(\QQ)$ is generated over $\QQ$ by $T(p),
T_i(p^2) \ (i=1,\ldots,n),$ and $[p^{-1}]_n$ for all $p$. We note that 
$T_n(p^2)=[p]_n$. We note that  ${\bf L}_n$ is generated over $\ZZ$ by  $T(p)$ and
$T_i(p^2) \ (i=1,\ldots,n)$ for all $p$. 
Let $F$ be a Hecke eigenform in $M_{{\bf k}}(\varGamma^{(n)})$.  As is well known, $\QQ(F)$ is a totally real algebraic number field of finite degree.  Now, first we consider the integrality of the eigenvalues of Hecke operators. For an algebraic number field $K,$ let $\frkO_K$ denote the ring of integers in $K.$ The following assertion can be proved in the same manner as in  \cite{Mizumoto91}. (See also \cite{Katsurada08}.) 
 \begin{proposition} \label{prop.Hecke-integrality} 
Let $k_n \ge n+1.$ Let $F$ be a Hecke eigenform in  $S_{{\bf k}}(\varGamma^{(n)})$. Then $\lambda_F(T)$ belongs to $\frkO_{\QQ(F)}$ for any $T \in {\bf L}_n^{({\bf k})}.$ 
\end{proposition}
   Let ${\bf L}_{n,p}={\bf L}( \varGamma^{(n)}, \mathrm{GSp}_n({\QQ})^+  \cap \GL_{2n}({\ZZ}[p^{-1}]))$ be the Hecke algebra associated with the pair $(\varGamma^{(n)},\mathrm{GSp}_n({\QQ})^+  \cap \GL_{2n}({\ZZ}[p^{-1}])).$  ${\bf
L}_{n,p}$ can be considered as a subalgebra of ${\bf L}_{n},$ and is generated over ${\QQ}$ by $T(p)$ and $T_i(p^2) \ (i=1,2,\ldots, n),$ and $[p^{-1}]_n.$ We now review the Satake $p$-parameters of ${\bf L}_{n,p};$ let ${\bf P}_n={\QQ}[X_0^{\pm},X_1^{\pm},\ldots,X_n^{\pm}]$ be the ring of Laurent polynomials in $X_0,X_1,\ldots,X_n$ over ${\QQ}.$ Let ${\bf W}_n$ be the group of ${\QQ}$-automorphisms of ${\bf P}_n$ generated by all permutations in variables $X_1,\ldots,X_n$ and by the automorphisms $\tau_1,\ldots,\tau_n$ defined by
$$\tau_i(X_0)=X_0X_i,\tau_i(X_i)=X_i^{-1},\tau_i(X_j)=X_j \ (j\not=i).$$
Moreover, a group $\tilde {\bf W}_n$ isomorphic to ${\bf W}_n$ acts on the set $T_n=({\CC}^{\times})^{n+1}$ in a way similar to the above. 
Then there exists a  ${\QQ}$-algebra isomorphism $\Phi_{n,p}$, called the Satake isomorphism, from  ${\bf L}_{n,p}$ to the ${\bf W}_n$-invariant subring ${\bf P}_n^{{\bf W}_n}$ of ${\bf P}_n.$  Then for a ${\QQ}$-algebra homomorphism $\lambda$ from ${\bf L}_{n,p}$ to ${\CC},$   there exists an element $(\alpha_0(p,\lambda),\alpha_1(p,\lambda),\ldots,\alpha_n(p,\lambda))$ of ${\bf T}_n$ satisfying
$$\lambda(\Phi_{n,p}^{-1}(F(X_0,X_1,\ldots,X_n)))=F(\alpha_0(p,\lambda),\alpha_1(p,\lambda),\ldots,\alpha_n(p,\lambda))$$
for $F \in  {\bf P}_n^{{\bf W}_n}.$ The equivalence class of 
$(\alpha_0(p,\lambda),\alpha_1(p,\lambda),\ldots,\alpha_n(p,\lambda))$ under the action of $\tilde {\bf W}_n$ is uniquely determined  by $\lambda.$ We call this the Satake parameters of ${\bf L}_{n,p}$ determined by $\lambda.$
 Now let 
 $F$ be a Hecke eigenform in  $M_{{\bf k}}(\varGamma^{(n)})$.  Then for each prime number $p,$ $F$ defines a ${\QQ}$-algebra homomorphism $\lambda_{F,p}$ from ${\bf L}_{n,p}$ to ${\CC}$ in a usual way, and we denote by $\alpha_0(p),\alpha_1(p),\ldots,\alpha_n(p)$ the Satake parameters of ${\bf L}_{n,p}$ determined by
$F.$ 
For later purpose, we consider special elements in ${\bf L}_{n,p};$ the polynomials $r_n(X_1,\ldots,X_n)=\sum_{i=1}^n (X_i+X_i^{-1})$ and $\rho_n(X_0,X_1,\ldots,X_n)=X_0^2X_1X_2 \cdots X_n r_n(X_1,\ldots,X_n)$ are elements of ${\bf P}_n^{{\bf W}_n},$ and thus we can define elements $\Phi_{n,p}^{-1}(r_n(X_1,\ldots,X_n))$ and $\Phi_{n,p}^{-1}(\rho_n(X_0,X_1,\ldots,X_n))$ of ${\bf L}_{n,p},$ which are  denoted by ${\bf r}_{n,1}(p)$ and $\widetilde \rho_{n,1}$, respectively. 
\begin{proposition}  \label{prop.special-Hecke} 
\begin{itemize}
\item[(1)]  We have
$$\widetilde \rho_{n,1}(p)=p^{n(n+1)/2}[p]_n\cdot {\bf r}_{n,1}(p),$$ and in particular
$$\lambda_F(\widetilde \rho_{n,1}(p))=p^{\sum_{i=1}^n k_i-n(n+1)/2 }\sum_{i=1}^n (\alpha_i(p)+\alpha_i(p)^{-1}).$$
\item[(2)] Let ${\bf k}=(k_1,\ldots,k_n)$ with $k_1 \ge \cdots \ge k_n \ge 0$. 
Suppose that $k_n \ge (n+1)/2$. 
Then
$\widetilde{\bf r}_{n,1}(p):=p^{k_1-1} {\bf r}_{n,1}(p)$ belongs to ${\bf L}_{n}^{({\bf k})}$.
\end{itemize}
\end{proposition}
\begin{proof}  The assertion (1) can easily be checked remarking that
$\Phi_{n,p}(p^{n(n+1)/2}[p]_n)=X_0^2X_1\cdots X_n$ (cf. \cite[Lemma 3.3.34]{Andrianov87}).
We will prove the assertion (2). 
Put 
\[\mathrm{GSp}_n(\QQ)^+_{\infty}=\biggl\{ \begin{pmatrix} A & B \\ O & D \end{pmatrix} \in \mathrm{GSp}_n(\QQ)^+ \biggr\},\] 
$\varGamma_{\infty}^{(n)}=\varGamma^{(n)} \cap \mathrm{GSp}_n(\QQ)^+_{\infty},$ and let
${\bf L}_{n,\infty}={\bf L}(\varGamma_{\infty}^{(n)},\mathrm{GSp}_n(\QQ)^+_{\infty} \cap M_{2n}(\ZZ))$ be the Hecke algebra associated to the Hecke pair 
$(\varGamma_{\infty}^{(n)},\mathrm{GSp}_n(\QQ)^+_{\infty} \cap M_{2n}(\ZZ))$. Then there is a natural injection from ${\bf L}_n$ into
${\bf L}_{n,\infty}$. 
For an element $D \in M_n(\ZZ)^{\rm nd}$ put $D^*={}^tD^{-1}$ and for non-negative  integers $a,b$ such that $a+b \le n$, put $D_{a,b}=D_{a,b}(p)=1_{n-a-b} \bot p1_a \bot p^2 1_b$, $\widetilde U(D_{a,b})=
 (p^2D_{a,b}^* \bot D_{a,b})$ and $\Pi_{a,b}=\varGamma^{(n)}_{\infty}\widetilde U(D_{a,b})\varGamma^{(n)}_{\infty}$.  For an  element $\begin{pmatrix} A & B \\ O & D \end{pmatrix} \in \varGamma^{(n)}_{\infty}$ with $D \in  M_n(\ZZ)^{\rm nd}$ and  $L \in \GL_n(\ZZ) D \GL_n(\ZZ)$, 
we define the set $B(L,M)$ as
\[B(L,M) =\left\{ N \in M_n(\ZZ) \  \middle| \ \begin{pmatrix} \nu L^* & N \\ O & L \end{pmatrix} \in \varGamma^{(n)}_{\infty} M \varGamma^{(n)}_{\infty} \right\},\]
where $\nu=\nu(M)$. Then $\#\bigl(B(L,M)/L\bigr)$ does not depend on the choice of $L$, and is uniquely determined by $M$, which will be denoted by $\alpha(M)$, and in particular put $\beta(D_{a,b})=\alpha(p^2D_{a,b}^* \bot D_{a,b})$.
Then, by \cite[p.160]{Andrianov87}, as an element of ${\bf L}_{n,\infty}$,  $\widetilde \rho_{n,1}(p)$ is expressed as have
\begin{align*}
 \widetilde \rho_{n,1}(p)&=p^{(n-1)n/2} \bigl( \beta(D_{n-1,0})^{-1} \Pi_{n-1,0} +p^{n+1} \beta( D_{n-1,1})^{-1} \Pi_{n-1,1}\bigr).
\end{align*}
Let ${\bf k'}=(k_1-k_n,\ldots,k_{n-1}-k_n,0)$ and put $m=\mathrm{depth}({\bf k}')$.
Let $F(Z)$ be an element of $M_{\bf k}(\varGamma^{(n)})(\ZZ)$. Then we have
\[F(Z)=\sum_{T \in \calh_n(\ZZ)_{\ge 0}} a(T,F)(U) {\bf e}({\rm tr}(TZ) )\] 
 with $a(T,F)(U) \in \ZZ[U]_{{\bf k}'}$. 
Then, we have
\begin{align*}
&F|\widetilde \rho_{n,1}(p)(Z)=p^{2(k_1+\cdots +k_n)-n(n+1)+(n-1)n/2}\\
& \times \sum_{T} \Bigl\{\beta(D_{n-1,0})^{-1} (\det D_{n-1,0})^{-k_n} \\
&\times \sum_{L \in \Lambda_n \backslash \Lambda_n D_{n-1,0} \Lambda_n}  a(T,F)(UL^{-1})\sum_{N \in B(L,\widetilde U(D_{n-1,0}))/L} {\bf e}\Bigl({\rm tr}\bigl( T(p^2\ {}^t \!L^{-1}Z +N)L^{-1}\bigl)\Bigr)\\
&+\beta(D_{n-1,1})^{-1} (\det D_{n-1,1})^{-k_n} \\
&\times \sum_{L \in \Lambda_n \backslash \Lambda_n D_{n-1,1} \Lambda_n}  a(T,F)(UL^{-1})\sum_{N \in B(L,\widetilde U(D_{n-1,1}))/L} {\bf e}\Bigl({\rm tr}\bigl( T(p^2\ {}^t\!L^{-1}Z +N)L^{-1}\bigl)\Bigr) \Bigl\},\\
\end{align*}
where $\Lambda_n=\GL_n(\ZZ)$.
For $i=0,1$ we have
\begin{align*}
& \sum_{N \in B(L,\widetilde U(D_{n-1,i}))/L} {\bf e}\Bigl({\rm tr}\bigl( T\ (p^2\ {}^t\! L^{-1}Z +N)L^{-1}\bigr)\Bigr)
\\ &={\bf e}({\rm tr}(p^2T[L^{-1}]Z)) \sum_{N \in B(L,\widetilde U(D_{n-1,i}))/L} {\bf e}({\rm tr}( L^{-1}TN)).
\end{align*}
We have 
\begin{align*}
&\sum_{N \in B(L,\widetilde U(D_{n-1,i}))/L} {\bf e}({\rm tr}(L^{-1}TN)) =\begin{cases} \beta(D_{n-1,i}) & \text{ if } L^{-1}T \in M_n(\ZZ), \\
0 & \text{ otherwise.} \end{cases}
\end{align*}
We note that $L^{-1}T \in M_n(\ZZ)$ if and only if $p^2T[L^{-1}] \in \calh_n(\ZZ)$. Hence we have
\begin{align*}
&F|\widetilde \rho_{n,1}(p)(Z)=p^{2(k_1+\cdots +k_n)-n(n+1)+(n-1)n/2}\\
& \times \sum_{A} {\bf e}({\rm tr}(AZ)) \Bigl\{p^{-k_n(n-1)}  \sum_{L \in \Lambda_n \backslash \Lambda_n D_{n-1,0} \Lambda_n}  a(A[L],F)(UL^{-1})\\
&+p^{-(n+1)k_n+n+1} \sum_{L \in \Lambda_n \backslash \Lambda_n D_{n-1,1} \Lambda_n}  a(A[L],F)(U  L^{-1})\Bigl\},
\end{align*}
and therefore 
\begin{align*}
F|{\bf r}_{n,1}(p)(Z) &=p^{k_1+\cdots +k_n-n} \times \sum_{A} {\bf e}({\rm tr}(AZ)) \Bigl\{p^{-k_n(n-1)}  \sum_{L \in \Lambda_n \backslash \Lambda_n D_{n-1,0} \Lambda_n}  a(A[L],F)(U L^{-1})\\
&+p^{-(n+1)k_n+n+1} \sum_{L \in \Lambda_n \backslash \Lambda_n D_{n-1,1} \Lambda_n}  a(A[L],F)(UL^{-1})\Bigl\}.
\end{align*}
We note that $a(T,F)(U)$ is expressed as a $\ZZ$-linear combination of polynomials of the following form:
\[P(U)=\prod_{i=1}^{m} \prod_{j=1}^{l_i-l_{i+1}} U_{J_{ij}}\]
with $l_i=k_i-k_n$, where $(J_{i1},\ldots,J_{i,l_i-l_{i+1}}) \in \mathcal{SI}_{n,i}^{l_i-l_{i+1}}$. Therefore, to prove the assertion (2), it suffices to show that 
\begin{align*}
p^{2k_1+k_2+\cdots+k_n-n-1}p^{-k_n(n-1+2i)+(n+1)i}P(U L^{-1}) \in \ZZ[U]_{{\bf k'}} \tag{I}
\end{align*}
 for any $L \in \Lambda_n D_{n-1,i} \Lambda_n$ with $i=0,1$.
We may suppose $L=D_{n-1,i}$ with $i=0,1$.
First write $D_{n-1,1}=p^{d_1} \bot \cdots \bot p^{d_n}$ with 
$d_1=\cdots=d_{n-1}=1$ and $d_n=2$.
Then we have
\[P(UD_{n-1,1}^{-1})=p^{-\sum_{i=1}^{m} \sum_{j=1}^{l_i-l_{i+1}}\sum_{a=1}^{i} d_{J_{a+j}}}P(U),\]
where $\{ J_{a+j} \}_{1 \le i \le m,1 \le j \le l_i-l_{i+1
}, 1 \le a \le i} $ is a set of  integers such that $1 \le J_{a+j} \le n$ and $J_{a+j} \not=J_{a'+j'}$ if $a+j \not=a'+j'$.
Then we have $\sum_{a=1}^{i} d_{J_{a+j}} \le i+1$ for any $i$. Hence we have
\begin{align*}
\sum_{i=1}^{m} \sum_{j=1}^{l_i-l_{i+1}}\sum_{a=1}^{i} d_{J_{a+j}}
& \le \sum_{i=1}^{m} (l_i-l_{i+1})(i+1) 
 =2l_1+\sum_{i=2}^{m} l_i\\
&=2k_1+\sum_{i=2}^{m}k_i -(m+1)k_n
=2k_1+\sum_{i=2}^{n}k_i -(n+1)k_{n}.
\end{align*}
Hence (I) holds 
 for any $L \in \Lambda_n D_{n-1,1} \Lambda_n$.
Similarly, we have
 \[P(UD_{n-1,0}^{-1})=p^{-\gamma_{k,n}}P(U)\]
with $\gamma_{k,n}$ an integer such that 
\[\gamma_{k,n} \le \sum_{i=1}^{n-1}k_i-(n-1)k_n.\]
By assumption, we have $k_1+k_n \ge n+1$, and hence
(I) holds for  any $L \in \Lambda_n D_{n-1,0} \Lambda_n$. This proves the assertion (2).
\end{proof}

We write  
$\Gamma_{\CC}(s)=2(2\pi)^{-s}\Gamma(s)$ and write $\Gamma_{\RR}(s)=\pi^{-s/2}\Gamma(s/2)$ as usual. 
Let 
$$f(z)=\sum_{m=1}^{\infty} a(m,f){\bf e}(mz)$$
 be a primitive form in $S_k(\SL_2(\ZZ))$, that is 
let $f$ be a Hecke eigenform whose first coefficient is $1$. 
For a prime number $p$ let $\beta_{1,p}(f)$ and $\beta_{2,p}(f)$ be complex numbers such that $\beta_{1,p}(f)+\beta_{2,p}(f)=a(p,f)$ and
$\beta_{1,p}(f)\beta_{2,p}(f)=p^{k-1}$.
Then for a Dirichlet character $\chi$ we define  Hecke's L function $L(s,f)$ twisted by $\chi$ as
$$L(s,f,\chi)=\prod_p\bigl((1-\beta_{1,p}(f)\chi(p)p^{-s})(1-\beta_{2,p}(f)\chi(p)p^{-s})\bigr)^{-1}.$$
We write $L(s,f,\chi)=L(s,f)$ if $\chi$ is the principal character.

 Let $\{f_1,\ldots,f_d \}$ be a basis of  $S_{k}(\varGamma^{(1)})$ consisting of primitive forms. Let $K$ be an algebraic number field containing ${\QQ}(f_1)\cdots {\QQ}(f_d),$ and $\frkO$ the ring of integers in $K$.
Let $f$ be a primitive form in ${S}_{k}(\SL_2(\ZZ)).$  Then Shimura \cite{Shimura77} showed that there exist two complex numbers 
$c_{\pm}(f)$, uniquely determined up to ${\QQ}(f)^{\times}$ multiple such that the following property holds:

\noindent
(AL) \quad The value $\displaystyle {\Gamma_{\CC}(l) \sqrt{-1}^l L(l,f,\chi) \over \tau(\chi)c_s(f)}$
 belongs to  ${\QQ}(f)(\chi)$  for any positive integer $l \le k-1$ and a Dirichlet character $\chi,$ where $\tau(\chi)$ is the Gauss sum of $\chi$, and $s=s(l,\chi)=+$ or $-$ according as $\chi(-1)=(-1)^l$ or $(-1)^{l-1}.$ 
 
\bigskip
We note that the above value belongs to $K(\chi)$.

For short, we write
\[{\bf L}(l,f,\chi;c_s(f))={\Gamma_{\CC}(l) \sqrt{-1}^lL(l,f,\chi) \over \tau(\chi)c_s(f)}.\] 
We sometimes write $c_{s(l,\chi)}(f)=c_{s(l)}(f)$ and ${\bf L}(l,f,\chi;c_{s(l,\chi)}(f))={\bf L}(l,f;c_{s(l)}(f))$ if $\chi$ is the principal character.
We note that the value ${\bf L}(l,f,\chi;c_s(f))$ depends on the choice of $c_{s}(f)$, but if $(\chi \eta)(-1)=(-1)^{l+m}$, then  $s:=s(l,\chi)=s(m,\eta)$ and, the ratio
$\displaystyle {{\bf L}(l,f,\chi;c_s(f)) \over {\bf L}(m,f,\eta ;c_s(f)}$ does not depend on $c_s(f)$, which will be denoted by $\displaystyle {{\bf L}(l,f,\chi) \over {\bf L}(m,f,\eta)}$.
For two positive integers $l_1,l_2 \le k-1$ and Dirichlet characters $\chi_1,\chi_2$ such that $\chi_1(-1)\chi_2(-1)=(-1)^{l_1+l_2+1},$ the value 
$${\Gamma_{\CC}(l_1)\Gamma_{\CC}(l_2)L(l_1,f,\chi_1)L(l_2,f,\chi_2) \over \sqrt{-1}^{l_1+l_2+1}\tau((\chi_1 \chi_2)_0) (f , f )}$$
belongs to $\QQ(f)(\chi_1, \chi_2)$, where
$(\chi_1\chi_2)_0$ is the primitive character associated with $\chi_1\chi_2$ 
 (cf. \cite{Shimura76}). We denote this value by 
\[{\bf L}(l_1,l_2;f;\chi_1,\chi_2).\]
 In particular, we put
\[{\bf L}(l_1,l_2;f)={\bf L}(l_1,l_2;f;\chi_1,\chi_2)\]
if $\chi_1$ and $\chi_2$ are  the principal characters. This value does not depend upon the choice of ${c_{\pm}(f)}$. 
Let $f$ be a primitive form in $S_k(\SL_2(\ZZ))$. Let $f_1,\ldots,f_d$ be a basis of $S_k(\SL_2(\ZZ))$ consisting of primitive forms with $f_1=f$ and 
let $\frkD_f$ be the ideal of $\QQ(f)$ generated by all $\prod_{i=2}^d (\lambda_{f_i}(T(m))-\lambda_f(T(m)))$'s ($m \in \ZZ_{>0}$).
  For a prime ideal $\frkp$ of an algebraic number field, let $p_{\frkp}$ be the prime number such that $(p_{\frkp})=\ZZ \cap \frkp$. 
The following proposition is due to \cite{Katsurada21}, Theorem 5.4.

\begin{proposition} \label{prop.p-integrality}
Let $f$ be a primitive form in $S_k(\SL_2(\ZZ))$. 
Let  $\chi_1$ and $\chi_2$ be primitive characters  with conductors $N_1$ and $N_2$, respectively, and let $l_1,l_2$ be  positive integers such that  $k-l_1+1 \le l_2  \le l_1-1 \le k-2$
Let $\frkp$ be a  prime ideal  of $\QQ(f)(\chi_1,\chi_2)$ with $p_{\frkp} >k$.
Suppose that   $\frkp$ divides neither $\frkD_fN_1N_2$ nor  $\zeta(1-k)$.Then ${\bf L}(l_1,l_2;f;\chi_1,\chi_2)$ is $\frkp$-integral.
\end{proposition}
For two primitive forms $f_1 \in S_{k_1}(\mathrm{SL}_2(\ZZ))$ and $f_2 \in S_{k_2}(\mathrm{SL}_2(\ZZ))$ we define the Rankin-Selberg $L$ function $L(s,f_1 \otimes f_2)$ as 
\begin{align*}
&L(s,f_1 \otimes f_2)=\prod_p\bigl(\prod_{i=1}^2 \prod_{j=1}^2(1-\beta_{i,p}(f_1)\beta_{j,p}(f_2)p^{-s})\bigr)^{-1}.
\end{align*}
 Let $F$ be  a Hecke eigenform in $M_{(k_1,\ldots,k_n)}(\mathrm  {Sp}_{n}(\ZZ))$ with respect to  ${\bf L}_n,$ and for a prime number $p$ we take the $p$-Satake parameters 
 $\alpha_0(p)$, $\alpha_1(p),\ldots,\alpha_n(p)$  of $F$ so that 
$$\alpha_0(p)^2\alpha_1(p) \cdots \alpha_n(p)=p^{k_1+\cdots +k_n-n(n+1)/2}.$$
We define  the polynomial $L_p(X,F,{\rm Sp})$ by
$$L_p(X,F,{\rm Sp})=(1-\alpha_0(p)X)\prod_{r=1}^n\prod_{1 \le i_1<\cdots<i_r}(1-\alpha_0(p)\alpha_{i_1}(p)\cdots \alpha_{i_r}(p)X)$$
and the spinor $L$ function $L(s,F,\mathrm {Sp})$ by
$$L(s,F,\mathrm {Sp})=\prod_p L_p(p^{-s},F,{\rm Sp})^{-1}.$$
We note that  $L(s,f,{\rm Sp})$ is Hecke's L function $L(s,f)$
if $f$ is a primitive form. In this case we write $L_p(s,f)$ for $L_p(s,f,{\rm Sp})$.
We also define the polynomial $L_p(X,F,{\rm St})$ by
$$(1-X)\prod_{i=1}^n (1-\alpha_i(p)X)(1-\alpha_i(p)^{-1}X)$$
and 
the standard $L$-function $L(s,F,{\rm St})$ by 
$$L(s,F,{\rm St})=\prod_p L_p(p^{-s},F,{\rm St})^{-1}.$$
For a Hecke eigenform $F \in S_k(\varGamma^{(r)})$ put
\[{\bf L}(s,F,\St)=\Gamma_{\CC}(s)\prod_{i=1}^r \Gamma_{\CC}(s+k-i){ L(s,F,\St) \over (F,\ F)}.\]
\begin{remark}\label{rem.integrality-gamma-factor}
We note that for a positive integer $m \le k-r$
\[{\bf L}(m,F,\St)=A_{r,k,m} {L(m,F,\St) \over \pi^{r(k+m)+m-r(r+1)/2} (F, \ F)}\]
with an element $A_{r,k,m} \in \ZZ[2^{-1}]$ such that 
$\ord_p(A_{r,k,m})=0$ for any prime number $p \ge 2k-r-1$.
\end{remark}
\begin{proposition} \label{prop.algebraicity-standard-L}
Let $F$ be a Hecke eigenform in $S_k(\varGamma^{(r)})$.
We define $n_0=3$ if $r\geq 5$ with $r\equiv 1\bmod 4$ and $n_0=1$ otherwise. 
Let $m$ be a positive integer $n_0  \le m \le k-r$ such that $m \equiv r \pmod 2$. 
Then, 
$a(A,F)\overline{a(B,F)}{\bf L}(m,F,\St)$ belongs to $\QQ(F)$
for any $A, B \in \calh_r(\ZZ)_{>0}$.
\end{proposition}
\begin{proof} We note that the value $a(A,F)\overline{a(B,F)}{\bf L}(m,F,\St)$ remains unchanged if we replace by $F$ by $\gamma F$ with any $\gamma \in \CC^{\times}$. By the multiplicity one theorem for Hecke eigenforms (cf. 
Theorem \ref{remarkMF} (3) and Remark \ref{remarkMF} (2)), we can take some non-zero complex number $\gamma$ such that $\gamma F \in S_k(\varGamma^{(r)})(\QQ(F))$. 
For this $\gamma$, we see
${\bf L}(m,\gamma F,\St)\in \QQ(F)$ by \cite{Mizumoto91}, Appendix A.
This proves the assertion.
\end{proof}
Let $R$ be a commutative ring, and $\frka$ an ideal of $R$. For two polynomials $P(X)=\sum_{i=1}^n a_i X^i$ and $Q(X)=\sum_{i=1}^n b_i X^i$, we write 
\[P(X) \equiv Q(X) \pmod{\frka}\]
 if $a_i \equiv b_i \pmod{\frka}$ for any $1 \le i \le m$.
Now we will state Harder's conjecture. 
\begin{conjecture}(\cite{Harder03})
\label{conj.Harder}
Let $k$ and $j$ be non-negative integers such that $j$ is even and $k \ge 3$. 
Let $f=\sum a(n,f){\bf e}(nz) \in S_{2k+j-2}(\SL_2(\ZZ))$ be a primitive   form, and  suppose that  a ``large'' prime   $\frkp$ of $\QQ(f)$ 
divides ${\bf L}(k+j,f;c_{s(k+j)})$. Then, there exists a Hecke eigenform
$F \in S_{(k+j,k)}(\varGamma^{(2)})$, and a prime ideal  $\frkp' \mid\frkp$ in (any
field containing) $\QQ(f)\QQ(F)$ such that, for all primes $p$
\begin{align*} L_p(X,F,{\rm Sp}) \equiv L_p(X,f)(1-p^{k-2}X)(1-p^{j+k-1}X)
\pmod{\frkp'}. \end{align*}
In particular, 
$$\lambda_F(T(p))\equiv p^{k-2}+p^{j+k-1}+a(p,f) \pmod{\frkp'}.$$
\end{conjecture}
\begin{remark} \label{rem.Harder}
\begin{itemize}
\item[(1)] The original version of Harder's conjecture did not mention what ^^ ^^ largeness" of $\frkp$ means.
\item[(2)] To formulate Harder's conjecture we must choose the periods $c_s(f)$ in an appropriate way. The original 
version of Harder's conjecture  did not specify them. After that, Harder suggested assuming  another type of divisibility condition  instead of the divisibility of $L(k+j,f;c_s(f))$ in his conjecture (cf. \cite{Harder11}). However, it does not seem so easy to confirm such a  condition.
\item[(3)] The original version of Harder's conjecture, which  states only the last congruence on $\lambda_F(T(p))$, is naturally included in the above Euler factor
version since we have
\begin{align*}
L_p(X,F,\Sp)=& 1-\lambda_F(T(p))X+\lambda_F(\widetilde \rho_{2,1}(p))X^2\\
&-\lambda_F(T(p))p^{2k+j-3}X^3+p^{4k+2j-6}X^4,\end{align*}
and
\[L_p(X,f)=1-a(f,p)X+p^{2k-j-3}X^2.\]
 \item[(4)] The above congruence is trivial  in the case $k$ is even and $j=0$. Indeed, for the Saito-Kurokawa lift $F$ of $f$, 
we have
\[L_p(X,F,{\rm Sp}) = L_p(X,f)(1-p^{k-2}X)(1-p^{k-1}X),
\]
so we have equality, not only congruence.
\end{itemize}
\end{remark}
To avoid the ambiguity in (1) and (2) of Remark \ref{rem.Harder}, we propose the following conjecture, which we also call Harder's conjecture.
\begin{conjecture}
\label{conj.modified-Harder}
Let $k$ and $j$ be non-negative integers such that $j$ is even and $k \ge 3, j \ge 4$. Let $f$ be as that in Conjecture \ref{conj.Harder}. Suppose that a prime ideal $\frkp$ of $\QQ(f)$ satisfies $p_\frkp >2k+j-2$ and that $\frkp$ divides ${\displaystyle {\bf L}(k+j,f) \over \displaystyle {\bf L}(k_j,f)}$, where $k_j=k+j/2$ or $k+j/2+1$ according as $j \equiv 0 \pmod 4$ or $j \equiv 2 \pmod 4$. Then the same assertion as Conjecture \ref{conj.Harder} holds.
\end{conjecture}
\begin{remark} There is no ambiguity in the assumptions of Conjecture \ref{conj.modified-Harder}. Moreover, since we can compute ${\displaystyle {\bf L}(k+j,f) \over \displaystyle {\bf L}(k_j,f)}$ rigorously, we can easily check the  assumption on $\frkp$.
\end{remark}

\section{An enhanced version of Harder's conjecture}

Conjectures \ref{conj.Harder} and \ref{conj.modified-Harder} are not concerning the congruence between the Hecke eigenvalues of two Hecke eigenforms in the same space, 
and  this is one of the reasons that it is not easy to confirm it.
  To treat the conjecture more accessibly, we reformulate it 
in the case $k$ is even.  
(For odd $k$, see \cite{Ibukiyama08}, \cite{Ibukiyama14}.)

 To do so,  first, we review several results, on the  Galois representations attached to automorphic forms, and on liftings.
Let $R$ be a locally compact topological ring, and $M$ a free $R$-module of finite rank. For a  profinite group $G$, let $\rho:G \longrightarrow \mathrm{Aut}_R(M)$ be a continuous representation of $G$.
When we fix a basis of $M$ with $\mathrm{rank}_R M=n$, we write $\rho:G\longrightarrow \GL_n(R)$.  
The following result is due to Deligne \cite{Deligne68/69} in the case $n=1$, and due to Weissauer \cite{Weissauer05} in the case $n=2$.

\begin{theorem} \label{th.galspin}
Let $F$ be a Hecke eigenform in $S_{{\bf k}_n}(\varGamma^{(n)})$  with $n \le 2$, where 
${\bf k}_n= k$ or $(k+j,k)$ according as $n=1$ or $n=2$.
Let $K$ be a number field containing $\QQ(F)$ and $\frkp$ be a 
prime ideal  of $K$. Then there exists a semi-simple Galois representation $\rho_F=\rho_{F,\frkp}:\Gal(\bar \QQ/\QQ) \longrightarrow \GL_{2^n}(K_{\frkp})$ such that 
  $\rho_{F,\frkp}$ is unramified at any prime number $p \not=p_{\frkp}$ and 
\[\det (1_{2^n}-\rho_{F,\frkp}(\mathrm{Frob}_p^{-1})X)=L_p(X,F,{\rm Sp}),\]
where $\mathrm{Frob}_p$ is the arithmetic Frobenius at $p$.
\end{theorem}

\begin{theorem} \label{th.atobe1}
\begin{itemize}
\item[(1)] Let ${\bf k}=(k_1,\ldots,k_n) \in \ZZ^n$ with $k_1 \ge \cdots \ge k_n >n$, and $G$ be a Hecke eigenform in $S_{\bf k}(\varGamma^{(n)})$.  Let $k \ge 4$ and $j > 0$ and $d>0$. Assume that
\begin{itemize}
\item[(a)] $ k \equiv n \pmod 2, \ j \equiv 0 \pmod 2$;
\item[(b)] $k>2d+1$ and $j > 2d-1$; 
\item[(c)] ${j \over 2}+d < k_i-i < {j \over 2}+k-d-1$ for $i=1,\ldots,n$.
\end{itemize}
Define ${\bf k}'=(k_1',\ldots,k_{n+4d}') \in \ZZ^{n+4d}$ so that 
$k_1' \ge \ldots \ge k_{n+4d}'$ and
\begin{align*}
&\{k_1'-1,k_2'-2,\ldots,k_{n+4d}'-n-4d\} =\{k_1-1,k_2-2,\ldots,k_n-n \}\\
&\cup \bigl\{{j \over 2}+k+d-2,{j \over 2}+k+d-3, \ldots,{j \over 2}+k-d-1 \bigr\} \\
&\cup \bigl\{{j \over 2}+d,{j \over 2}+d-1,\ldots,{j \over 2}-d+1 \bigr\}.
\end{align*}
Then, for any Hecke eigenform $F \in S_{(k+j,k)}(\varGamma^{(2)})$ there exists a Hecke eigenform
$\scra_{n,d,{\bf k}}^{(I),{\bf k}'}(F,G) \in S_{{\bf k'}}(\varGamma^{(n+4d)})$ such that 
\begin{align*}
L(s,\scra_{n,d,{\bf k}}^{(I),{\bf k'}}(F,G),\St)
&=L(s,G,\St)\prod_{i=1}^{2d}L(s+d+\frac{j}{2}+k-1-i,F,\Sp).
\end{align*}
Here we make the convention that
$L(s,G,\St)=\zeta(s)$ if $n=0$.
\item[(2)] Let $k$ and $n$ be positive even integers such that $k > n > 2$. Let $f$ be a primitive form in $S_{2k-n}(\SL_2(\ZZ))$ and $G$ be a Hecke eigenform in $S_{(k,k-n+2)}(\varGamma^{(2)})$.  Then, there exists a Hecke eigenform $\scra^{(II)}_n(f,G) \in S_k(\varGamma^{(n)})$ such that
\[L(s,\scra^{(II)}_n(f,G),\St)=L(s,G,\St)\prod_{i=1}^{n-2}L(s+k-1-i,f).\]
\end{itemize}
\end{theorem}

\begin{theorem} \label{th.atobe2}
Let $G$ be a Hecke eigenform in $S_{\bf k}(\varGamma^{(n)})$ for a fixed 
${\bf k}=(k_1,\ldots,k_n) \in \ZZ^n$ with $k_1 \ge \cdots \ge k_n >n$.
For positive integers $k$ and $d$ with $k >d$, we assume one of the following conditions:
\begin{itemize}
\item[(1)] $k-d >k_1-1$ and $k \equiv d \pmod 2$.
\item[(2)] $k+d-1<k_n-n,  \ k>d$ and $k \equiv d+n \pmod 2$.
\end{itemize}
Define ${\bf k}'=(k_1',\ldots,k_{n+2d}') \in \ZZ^{n+2d}$ so that $k_1' \ge \cdots \ge k_{n+2d}'$ and
\begin{align*}
&\{k_1'-1,k_2'-2,\ldots,k_{n+2d}'-(n+2d)\} \\
&=\{k_1-1,k_2-2,\ldots,k_n-n\} \cup \{k+d-1,k+d-2,\ldots,k-d\}.
\end{align*}
Then, for any Hecke eigenform $f \in S_{2k}(\SL_2(\ZZ))$, there exists a Hecke eigenform $\scrm_{n,d,{\bf k}}^{{\bf k}'}(f,G) \in S_{\bf k'}(\varGamma^{(n+2d)})$ such that
\[L(s,\scrm_{n,d,{\bf k}}^{{\bf k}'}(f,G),\St)=L(s,G,\St)\prod_{i=1}^{2d}L(s+k+d-i,f).\]
Here we make the convention that
$L(s,G,\St)=\zeta(s)$ if $n=0$.
\end{theorem}
Theorem \ref{th.atobe1} (1) for the case $n=0$, and Theorem \ref{th.atobe1} (2) have been proved in \cite[Proposition 5.3]{Dummigan17} and \cite[Proposition 5.2]{Dummigan17}, respectively. (These results  were proved under a certain assumption. But such an assumption was proved  by Arancibia, M{\oe}glin and Renard \cite{AMR}(cf.\ Remark \ref{remarkMF}), and they are now unconditional results.) A general case of Theorem \ref{th.atobe1} (1) and Theorem \ref{th.atobe2} may be proved similarly.
But, for readers' convenience we will give their proofs in Appendix A.
 Theorem \ref{th.atobe1} was conjectured by Ibukiyama \cite{Ibukiyama12} in  special cases with numerical examples.
We say that  the lifts in (1) and (2) are the lifts of types $\scra^{(I)}$ and $\scra^{(II)}$, respectively.
 If $n=0$ and  
\begin{align*}
& {\bf k}'=\Bigl(\overbrace{{j \over 2}+k+d-1,\ldots,{ j \over 2}+k+d-1}^{2d},\overbrace{{j \over 2}+3d+1,\ldots,{j \over 2}+3d+1}^{2d}\Bigr),
\end{align*}
we simply write  $\scra^{(I)}_{4d}(F)$ instead of $\scra^{(I),{\bf k'}}_{0,d,{\bf k}}(F,G)$ because ${\bf k}$ and ${\bf k}'$ are  determined by $F$ and $d$.
Theorem \ref{th.atobe2} was conjectured by Miyawaki \cite{Miyawaki92} with numerical examples. We also note that 
$\scrm_{n,d,{\bf k}}^{{\bf k}'}(f,G)$ was constructed by Ikeda \cite{Ikeda06} in the case (2) under the assumption $k_1=\cdots=k_n$ and the non-vanishing condition. In particular in the case $n=0$, it was constructed by Ikeda  \cite{Ikeda01}, and we write it 
 as $\scri_{2d}(f)$. The following proposition is more or less well known.
\begin{proposition} \label{prop.standard-L-Klingen}
Let ${\bf k}=(k_1,\ldots,k_m,\ldots,k_n)$  and ${\bf l}=(k_1,\ldots,k_m)$ with $k_1 \ge \cdots \ge k_m \ge \cdots \ge k_n.$
Let $F$ be a Hecke eigenform in $S_{\bf l}(\varGamma^{(m)})$, and suppose that  $[F]^{\bf k}=[F]^{\bf k}(Z,0)$ belongs to $M_{\bf k}(\varGamma^{(n)})$. Then, $[F]^{\bf k}$ is a Hecke eigenform and 
\[L(s,[F]^{\bf k},\St)=L(s,F,\St)\prod_{i=m+1}^{n}\bigl(\zeta(s+k_i-i)\zeta(s-k_i+i)\bigr).\]
\end{proposition}
\begin{proof}
The assertion is well known in the case $k_1=\cdots=k_m=\cdots=k_n$ (cf. \cite[Exercise 4.3.24]{Andrianov87}, and a general case can also be proved by the same argument.
\end{proof}

Let $F$ and $G$ be Hecke eigenforms in $M_{{\bf k}}(\varGamma^{(n)})$ and $\frkp$ a prime ideal of $\QQ(F)$. We say that $F$ is Hecke congruent to $G$ modulo $\frkp$ if there is a prime ideal $\frkp'$ of $\QQ(F)\cdot\QQ(G)$ lying above $\frkp$ such that
\[\lambda_G(T) \equiv \lambda_F(T) \pmod \frkp \text{ for any } T \in {\bf L}_n^{({\bf k})}.\]
We denote this property by
\[G \equiv_{\ev} F \pmod \frkp.\] 

\begin{conjecture}
\label{conj.main-conjecture}
Let $k,  j $ and $n$ be positive integers. Suppose  that
\begin{itemize}
\item[(a)] $n \equiv k \equiv j \equiv 0 \text{ mod } 2$ and $j/2+n/2  \equiv 1 \text{ mod } 2$.
\item[(b)] $k>n+1$ and $j > n-1$.
\end{itemize}
Put 
\begin{align*}
& {\bf k}=\Bigl(\overbrace{{j \over 2}+k+{n \over 2}-1,\ldots,{ j \over 2}+k+{n \over 2}-1}^n,\overbrace{{j \over 2}+{3n \over 2}+1,\ldots,{j \over 2}+{3n \over 2}+1}^n\Bigr).
\end{align*}
Let $f(z)=\sum a(l,f) {\bf e}(lz) \in S_{2k+j-2}(\SL_2(\ZZ))$ be a primitive form.  Let $\frkp$ be a prime ideal of $\QQ(f)$ such that $p_{\frkp} > 2k+j-2$ and suppose that $\frkp$
divides ${\displaystyle {\bf L}(k+j,f) \over \displaystyle {\bf L}(j/2+k+n/2-1,f)}$. Then, there exists a Hecke eigenform
$F \in S_{(k+j,k)}(\varGamma^{(2)})$ such that 
\[\scra^{(I)}_{2n}(F) \equiv_{\ev} [\scri_n(f)]^{{\bf k}}  \pmod{\frkp}.\]

\end{conjecture}
\begin{remark} \label{rem.holomorphy-of_Klingen-Eisenstein-series}
Since we have $j+3n/2+1 >3n/2+1$,
$[\scri_n(f)]^{\bf k}$ belongs to $M_{\bf k}(\varGamma^{(2n)})$ by Proposition \ref{prop.holomorphy-Klingen-Eisenstein} (2).
\end{remark}

Let $\frkO$ be the ring of integers in an algebraic number field $K$, and $\frkP$ a maximal ideal of $\frkO$.
Let $\cala_{\frkP}$ be  the Grothendieck group of finite dimensional  Galois representations of $\Gal(\bar \QQ/\QQ)$ with coefficients $\frkO_{\frkP}/\frkP$ unramified outside $\frkP$.
 Let $\scrs$ be the set of isomorphism classes of irreducible representations  in $\cala_{\frkP.}$
Write an element $H$ of $\cala_{\frkP}$ as
\[H=\sum_{S \in \scrs} n_S S \text{ with } n_S \in \ZZ\]
and set
\[\|H\|=\sum_{S \in \scrs} |n_S| \dim S. \]
\begin{lemma} \label{lem.Chenevier-Lannes}
Let $\frkP$ be as above.
 Let $H$ be an element of $\cala_{\frkP}$. 
Suppose that $(\bar \chi^i+1)H=0$ with $i=1,2$, where 
$\bar \chi$ is the mod $\frkP$ representation of the cyclotomic character 
$\chi:\Gal(\bar \QQ/\QQ) \longrightarrow \GL_1(K_\frkP)$. Then $\|H\|$ is divisible by $(p_{\frkP}-1)/i$.
\end{lemma}
\begin{proof}
The assertion for $i=1$ has been proved in Proposition 10.4.6 of Chenevier and Lannes \cite{Chenevier-Lannes19}, and the other assertion can also be proved by using the same argument as there.
\end{proof}

\begin{theorem}\label{th.enhanced-Harder}
Let the notation be as in Conjecture \ref{conj.main-conjecture}.
\begin{itemize} 
\item[(1)] Conjecture \ref{conj.modified-Harder} holds for the case $j \equiv 0 \pmod 4$
 if Conjecture \ref{conj.main-conjecture} holds for $n=2$.
\item[(2)] Suppose that $2k+j-2 \ge 20$. Then Conjecture \ref{conj.modified-Harder} holds for the case $j \equiv 2 \pmod 4$
 if Conjecture \ref{conj.main-conjecture} holds for $n=4$.
\end{itemize}
\end{theorem}
\begin{proof} 
Let $f$ be a primitive form in Conjecture \ref{conj.modified-Harder}, and suppose that a prime ideal $\frkp$ of $\QQ(f)$ satisfies the assumptions in  Conjecture \ref{conj.modified-Harder}. Then, by Conjecture \ref{conj.main-conjecture}, there exists a Hecke eigenform $F$ in $S_{(k+j,k)}(\varGamma^{(2)})$ satisfying the conditions in Conjecture \ref{conj.main-conjecture}.
Let $K=\QQ(f)\cdot\QQ(F)$ and $\frkO$ the ring of integers in $K$. Take a prime ideal $\frkP$ of $\frkO$ lying above $\frkp$.
Then it suffices to show that 
\begin{align}
\label{C}
&L_p(X,F,{\rm Sp})  \tag{$\mathrm {C}_p$}\equiv L_p(X,f)(1-p^{k-2}X)(1-p^{j+k-1}X) \pmod \frkP \notag
\end{align}
for any prime number $p$. By Proposition \ref{prop.standard-L-Klingen}, we have 
\begin{align*}
&L(s,[\scri_n(f)]^{{\bf k}},\mathrm{St})=L(s,\scri_n(f),\mathrm{St}) \times \prod_{i=1}^n \zeta(s+\frac{j}{2}+\frac{n}{2}+1-i)\zeta(s-\frac{j}{2}-\frac{n}{2}-1+i)
\end{align*}
and
\[L(s,\scri_n(f),\mathrm{St})=\zeta(s)\prod_{i=1}^n L(s+\frac{j}{2}+k+\frac{n}{2}-1-i,f).\]
Hence 
\begin{align*}
&L(s,[\scri_n(f)]^{{\bf k}},\mathrm{St})\\
&=\zeta(s)\prod_{i=1}^n L(s+\frac{j}{2}+k+\frac{n}{2}-1-i,f)
\times \prod_{i=1}^n \zeta(s+\frac{j}{2}+\frac{n}{2}+1-i)\zeta(s-\frac{j}{2}-\frac{n}{2}-1+i)
\\ &=\zeta(s)\prod_{i=1}^n \Bigl( L(s+\frac{j}{2}+k+\frac{n}{2}-1-i,f)\zeta(s+\frac{j}{2}+
\frac{n}{2}+1-i)\zeta(s-\frac{j}{2}+\frac{n}{2}-i)\Bigr).
\end{align*}
 Then, by (1) of Theorem \ref{th.atobe1}, we have
\begin{align*}
& L(s,\scra_{2n}^{(I)}(F), \mathrm{St})=\zeta(s)\prod_{i=1}^n L(s+j/2+k+n/2-1-i,F,\mathrm{Sp}).
\end{align*}
By \cite[(3.3.52), (3.3.53), Theorem 3.3.30, Lemma 3.3.34]{Andrianov87},  for any prime number $p$,  the $i$-th  coefficient of $L_p(X,[\scri_n(f)]^{\bf k},\St)$ and $L_p(X,\scra^{(I)}_{2n}(F),\St)$ are of the form  $p^{n_i}\lambda_{[\scri_n(f)]^{\bf k}}(T_i)$ and $p^{n_i}\lambda_{\scra^{(I)}_{2n}(F)}(T_i)$ with  $n_i \in \ZZ_{\le 0}$ and  $T_i \in {\bf L}_n^{{\bf k}}$, respectively.
 Therefore,  for any prime number $p \not=p_{\frkp}$,  they belong to $\frkO_{\frkP}$, and by the assumption, we have
\[L_p(X,\scra^{(I)}_{2n}(F),\St) \equiv L_p(X,[\scri_n(f)]^{\bf k},\St) \pmod{\frkP}.\]
Hence we have 
\begin{align*}
\prod_{i=1}^n L_p(p^{i-1}X,F,{\rm Sp}) 
&\equiv \prod_{i=1}^n  L_p(p^{i-1}X,f)(1-p^{i-1}p^{k-2}X)(1-p^{i-1}p^{j+k-1}X) \pmod {\frkP} \tag{$\mathrm{D}_p$}
\end{align*}
for any prime number $p \not=p_{\frkp}$.
Let $\rho_{F}:\Gal(\bar \QQ/\QQ) \longrightarrow \GL_4(K_{\frkP})$ and $\rho_f:\Gal(\bar \QQ/\QQ) \longrightarrow \GL_2(K_{\frkP})$ be the Galois representation attached to the spin $L$ functions of $F$ and $f$, respectively.  For $\rho=\rho_F,\rho_f$ let $\bar \rho$ be the mod $\frkP$ representation of $\rho$. Then, by ($\mathrm{D}_p$),
 in the Grothendieck ring $\cala_{\frkP}$,
\[(1+\bar \chi^{-1}) \bar \rho_F=(1+ \bar \chi^{-1} )(\bar\rho_f+\bar \chi^{2-k} +\bar \chi^{-j-k+1})\]
or 
\[(1+\bar \chi^{-1})(1 +\bar \chi^{-2}) \bar \rho_F=(1+ \bar \chi^{-1})(1+\bar \chi^{-2})(\bar\rho_f+\bar \chi^{2-k} +\bar \chi^{-j-k+1})\]
according as $n=2$ or $4$.
Define an element $H$ of $\cala_{\frkP}$ as
\[H=\bar \rho_F-(\bar \rho_f+\bar \chi^{2-k} +\bar \chi^{-j-k+1}).\]
Then we have  $||H|| \le 8$. Let $n=2$. Then, we have
\[(1+\bar \chi^{-1})H=0.\]
Since we have  $2k+j-2 \equiv 2 \text{ mod } 4$, by assumption we have  $p_\frkP=p_\frkp \ge 2k+j-2 \ge 18$. Hence, by Lemma \ref{lem.Chenevier-Lannes}, we have $H=0$. Let $n=4$. Then,
\[(1+\bar \chi^{-1})(1 +\bar \chi^{-2})H=0.\]
Since we have $p_\frkP=p_\frkp >2k+j-2 \ge 20$, using Lemma \ref{lem.Chenevier-Lannes} repeatedly, we also have $H=0$. Hence, in $\cala_{\frkP}$  we have
\[\bar \rho_F=\bar \rho_f+\bar \chi^{2-k} +\bar \chi^{-j-k+1}.\]
This implies that the congruence relation (\ref{C}) holds for any prime number $p \not=p_\frkp$.

 Let $p=p_\frkp$. Then we have
\begin{align*}
\lambda_{\scra_{2n}^{(I)}(F)}({\bf r}_{2n,1}(p)) =p^{-j/2-k-n/2+1}\lambda_F(T(p))\sum_{i=1}^n p^i,
\end{align*}
and
\begin{align*}
\lambda_{[\scri_{n}(f)]^{\bf k}}({\bf r}_{2n,1}(p))=p^{-j/2-k-n/2+1}a(p;f)\sum_{i=1}^n p^i+p^{-j/2-n/2-1}\sum_{i=1}^n p^i+p^{j+n/2+1}\sum_{i=1}^n p^{-i}.
\end{align*}
Since the Hecke operator $\widetilde {\bf r}_{2n,1}(p)=p^{k+j/2+n/2-2}{\bf r}_{2n,1}(p)$ belongs to ${\bf L}_{2n}^{{\bf k}}$ by Proposition \ref{prop.special-Hecke},  we have
\begin{align*}
&\lambda_F(T(p)) \equiv \lambda_{\scra_{2n}^{(I)}(F)}(\widetilde {\bf r}_{2n,1}(p)) \equiv \lambda_{[\scri_{n}(f)]^{\bf k}}(\widetilde {\bf r}_{2n,1}(p))
\equiv a(p;f) \pmod{\frkP}.
\end{align*}
Moreover, all the coefficients of $X^m$ with $m \ge 2$ of the both polynomials $L_p(X,F,{\rm Sp})$ and $L(X,f)(1-p^{k-2}X)(1-p^{j+k-1}X)$ are congruent to $0$ mod $\frkP$. This proves the assertion.
\end{proof}
\begin{remark} \label{rem.commnet-on-main-conjecture}
\begin{itemize}
\item[(1)] Our conjecture is stronger than Conjecture \ref{conj.Harder} in the case $k$ is even.
\item[(2)] The above  conjecture tells nothing about the case $k$ is odd. However, we can propose a similar conjecture in the case $k$ is odd. 
\end{itemize}
\end{remark}

\section{Pullback formula}
\subsection{Differential operators with automorphic property}
\allowdisplaybreaks[4]
In this section, we explain some explicit differential operators that are used in 
the pullback formula. 

\subsubsection{Setting}
 
Now for an integer $n\geq 2$, fix a partition $(n_1,n_2)$ with $n=n_1+n_2$ with $n_i\geq 1$.
Let $\lambda$ be a dominant integral weight with $\depth(\lambda)\leq \min(n_1,n_2)$. For $i=1,2$, let $(\rho_{n_i,\lambda},V_{n_i,\lambda})$ be the representation of $\GL_{n_i}(\CC)$ defined in Section 1.
Put $V_{\lambda,n_1,n_2}=V_{n_1,\lambda}\otimes V_{n_2,\lambda}$. 
We regard $\HH_{n_1}\times \HH_{n_2}$ as a subset of $\HH_n$ by the diagonal embedding. 
We consider $V_{\lambda,n_1,n_2}$-valued differential operators $\DD$  
on scalar valued functions of $\HH_{n}$, satisfying  Condition \ref{condition} below 
on automorphy: 
We fix $\lambda$, $n_1$, $n_2$ as above. 
For variables $Z_i \in \HH_{n_i}$, irreducible representations $(\rho_i,V_i)$ of $\GL_{n_i}(\CC)$ for $i=1$, $2$, 
a $V_1\otimes V_2$ valued function $f(Z_1,Z_2)$ on $\HH_{n_1}\times \HH_{n_2}$, 
and $g_i=\begin{pmatrix}
A_i & B_i \\ C_i & D_i \end{pmatrix}\in \Sp_{n_i}(\RR)$, we write 
\[
(f|_{\rho_1,\rho_2}[g_1,g_2])(Z_1,Z_2)
=
\rho_1(C_1Z_1+D_1)^{-1}\otimes \rho_2(C_2Z_2+D_2)^{-1}f(g_1Z_1,g_2Z_2).
\]
We regard $\Sp_{n_1}(\RR)\times \Sp_{n_2}(\RR)$ as a subgroup of $\Sp_n(\RR)$ by 
\[
\iota(g_1,g_2)=\begin{pmatrix} A_1 & 0 & B_1 & 0 \\ 0 & A_2 & 0 & B_2 \\
C_1 & 0 & D_1 & 0 \\
0 & C_2 & 0 & D_2 \end{pmatrix} \qquad (g_i \in \Sp_{n_i}(\RR) \text{ for } i=1,2).
\]
For $Z=(z_{ij})\in \HH_n$, we denote by $\p_Z$ the following $n\times n$ symmetric matrix of partial derivations 
\[
\p_Z=\left(\frac{1+\delta_{ij}}{2}\frac{\p}{\p z_{ij}}\right)_{1\leq i,j\leq n}.
\]
For a $V_{\lambda,n_1,n_2}$ valued polynomial $P(T)$ in components of $n\times n$ symmetric matrix $T$, 
we put $\DD_P=P(\p_Z)$. 
 
\begin{condition}\label{condition}
We fix $k$ and $\lambda$. Let $\DD=P(\p_Z)$ as above.
For any holomorphic function $F$ on $\HH_n$ and any $(g_1,g_2)\in \Sp_{n_1}(\RR)\times \Sp_{n_2}(\RR)$,
the operator $\DD$ satisfies   
\[
\mathrm{Res}(\DD(F|_{k}[\iota(g_1,g_2)])=(\mathrm{Res}\ \DD(F))|_{\det^k\rho_{n_1,\lambda},\det^k\rho_{n_2,\lambda}}[g_1,g_2],
\]
where $\mathrm{Res}$ means the restriction of a function on $\HH_{n}$ to $\HH_{n_1}\times \HH_{n_2}$. 
\end{condition}
For  $Z=\begin{pmatrix} Z_1 & Z_{12} \\{}^t Z_{12} & Z_2 \end{pmatrix} \in \HH_n$ with $Z_1 \in \HH_{n_1}, Z_2 \in \HH_{n_2}, Z_{12} \in M_{n_1,n_2}(\CC)$, we sometimes write $\DD(F)\begin{pmatrix} Z_1 & O \\ O & Z_2 \end{pmatrix}$ instead of $\mathrm{Res} \ \DD(F(Z))$. This condition on $\DD$ is, roughly speaking, the condition that, 
if $F$ is a Siegel modular form of degree $n$ of weight $k$,
then $\mathrm{Res}(\DD(F))$ is a Siegel modular form of weight $\det^k \rho_{n_i,\lambda}$ for 
each variable $Z_i$ for $i=1$, $2$. Here, if $2k\geq n$, 
the condition that $\rho_1$ and $\rho_2$ correspond to the same
$\lambda$ is a necessary and sufficient condition for the existence of $\DD$
(\cite{Ibukiyama99}). A characterization for $P$ is given in \cite{Ibukiyama99}. 
We review it since we need it later.
For an $m\times 2k$ matrix $X=(x_{i\nu})$ of variables and for any $(i,j)$ with $1\leq i,j\leq m$, 
we put 
\[
\Delta_{ij}(X)=\sum_{\nu=1}^{2k}\frac{\p^2}{\p x_{i\nu}\p x_{j\nu}}.
\]
We say that a polynomial $\widetilde{P}(X)$ in $x_{i\nu}$ is pluri-harmonic if 
\[
\Delta_{ij}(X)\widetilde{P}(X)=0
\]
for any $i$, $j$ with $1\leq i,j \leq m$. 

\begin{theorem}[\cite{Ibukiyama99}]
We assume that $2k\geq n$. Notation and assumptions being as above,
the operator $\DD=P(\p_Z)$ satisfies Condition \ref{condition} if and only if 
the $V_{\lambda,n_1,n_2}$ valued polynomial $P$ satisfies the following two conditions.
\\
(1) For $i=1$, $2$, let $X_i$ be an $n_i\times 2k$ matrix of variables.
Then the polynomial
\[
\widetilde{P}(X_1,X_2):=P\begin{pmatrix} X_1\,^{t}\!X_1 & X_1\,^{t}\!X_2 \\
X_2\,^{t}\!X_1 & X_2\,^{t}\!X_2
\end{pmatrix}
\]
is pluri-harmonic for each $X_1$, $X_2$, that is,  
$\Delta_{ij}(X_1)\widetilde{P}=\Delta_{ij}(X_2)\widetilde{P}=0$, regarding 
that the variables in $X_1$ and in $X_2$ are independent. 
\\
(2) For any $A_1\in \GL_{n_1}(\CC)$ and $A_2\in \GL_{n_2}(\CC)$, we put 
\[
A=\begin{pmatrix} A_1 & 0 \\ 0 & A_2 \end{pmatrix}.
\]
Then we have 
\[
P(AT\,^{t}A)=\rho_{n_1,\lambda}(A_1)\otimes \rho_{n_2,\lambda}(A_2)P(T).
\]
Besides, for any fixed $k$ with $2k\geq n=n_1+n_2$ and $\lambda$,
the polynomial $P(T)$ satisfying (1) and (2) exists and is unique up to constant.
\end{theorem}
 
 There are a lot of results concerning explicit description of $P$, notably in 
\cite{ibuoneline}, \cite{ibuuniverse}. But still we need more explicit formula 
for our purpose and we will explain it  
in the next subsection.

\subsubsection{Explicit formula}
In this section, we consider some special type of $\lambda$.
We assume that $\lambda=(l,\ldots, l, 0, \ldots,0)$ with depth $m$. 
Put $\lambda_0=(l,\ldots,l)$. 
Then first we explain some general way to 
construct $V_{\lambda,n_1,n_2}$ polynomial $P(T)$ satisfying Condition \ref{condition}
from a scalar valued polynomial $P_0(S)$ satisfying Condition 
\ref{condition} for $\rho_{m,\lambda_0}\otimes \rho_{m,\lambda_0}$. 
Here $T$ is an $n\times n$ symmetric matrix and $S$ is an $2m\times 2m$ symmetric 
matrix.
Then for the case $m\leq 2$ and any $l$, we give a completely explicit 
description of $P(T)$. (The case $m=1$ has been already given in \cite{Ibukiyama99} and 
a new point is the case $m=2$.) 
Here we note that $\rho_{m,\lambda_0}=\det^l$ and 
$\det^k\rho_{m,\lambda_0}=\det^{k+l}$.

For any positive integers $k$, $l$, 
we consider Condition \ref{condition} for  $n=2m$, $(n_1,n_2)=(m,m)$
and from weight $k$ to weight $\det^{k+l}\otimes \det^{k+l}$.
If we denote by $P_{k,k+l}(S)$ a non-zero polynomial satisfying Condition 1 for this case, 
this is a scalar valued polynomial in components of an $2m\times 2m$ symmetric matrix $S$. 
We assume that we know $P_{k,k+l}(S)$, and then we consider how to give more general case 
starting from this $P_{k,k+l}$. 

First we review realization of representations of $\GL_{n_1}(\CC) \times \GL_{n_2}(\CC)$ by bideterminants. 
Let $U$, $V$ be $m\times n_1$ and $m\times n_2$ matrices of independent variables respectively. 
Let $\lambda=(l,\ldots,l,0,\ldots,0)$ such that 
$\depth(\lambda)=m$. For a integers $n_1$ and $n_2$ such that $n_1,n_2 \ge m$, put ${\bf k_1}'=(\overbrace{l,\ldots,l}^{m},\overbrace{0,\ldots,0}^{n_1-m})$ and ${\bf k_2}'=(\overbrace{l,\ldots,l}^{m},\overbrace{0,\ldots,0}^{n_2-m})$, and
let $\CC[U,V]_{{\bf k}_1',{\bf k}_2'}$ be the vector space defined in Section 2. Then, we can take $\CC[U,V]_{{\bf k}_1',{\bf k}_2'}$ as a representation space of $\rho_{n_1,\lambda} \otimes \rho_{n_2,\lambda}$ as explained in Section 2.
We denote by $\Bbb U$ the following $2m\times n$ matrix, where $n=n_1+n_2$:
\[
\Bbb U=\begin{pmatrix} U & 0 \\ 0 & V \end{pmatrix}.
\]
\begin{proposition}\label{prop.generalpol}
Notation being as above, consider $\lambda=(l,\ldots,l,0,\ldots,0)$ such that $\depth(\lambda)=m$. 
For a partition $n=n_1+n_2$, we assume that $m\leq \min(n_1,n_2)$. 
Let $T$ be an $n\times n$ symmetric matrix.
Then for $Q(T)=P_{k,k+l}(\Bbb U T\, ^{t}\Bbb U)$, the differential operator 
$\DD_{\lambda,n_1,n_2}=Q(\p_Z)$ for $Z\in \HH_n$ satisfies Condition \ref{condition} for 
$k$ and $\det^k\rho_{n_1,\lambda}$, $\det^k\rho_{n_2,\lambda}$. 
\end{proposition}

\begin{proof}
For $A_1\in \GL_{n_1}(\CC)$ and $A_2\in \GL_{n_2}(\CC)$, we put 
\[
A=\begin{pmatrix} A_1 & 0 \\ 0 & A_2\end{pmatrix}.
\]
The fact that $Q(T)$ is in the representation space of $\rho_{n_1,\lambda}\otimes \rho_{n_2,\lambda}$
for the action $U\rightarrow UA_1$ and $V\rightarrow VA_2$
is concretely proved by using a structure theorem on the shape of $P_{k,k+l}(S)$ in \cite[Proposition 3.1]{ibukuzumakiochiai},
but we will later give a more abstract proof in the lemma below.
So here we prove the rest.
We write 
\begin{equation}\label{Tblocks}
T=\begin{pmatrix} T_{11} & T_{12} \\
^{t}T_{12} & T_{22} \end{pmatrix},
\end{equation}
where $T_{11}$ is an $n_1\times n_1$, $T_{12}$ is an $n_1\times n_2$, and $T_{22}$ 
is an $n_2\times n_2$ matrix.
Then 
\[
Q(AT\,^{t}A)=P_{k,k+l}
\begin{pmatrix} 
UA_1 T_{11}\,^{t}(UA_1) & UA_1 T_{12} \,^{t}(VA_2) \\
(VA_2) \,^{t}T_{12} \,^{t}(UA_1) & VA_2 T_{22} \,^{t}(VA_2)
\end{pmatrix}.
\]
So surely the action of $A$ to $T$ gives the action of $A$ on $U$, $V$ given by $UA_1$ and $VA_2$. 
This means that 
\[
Q(AT\,^{t}A)=\rho_{n_1,\lambda}(A_1)\otimes \rho_{n_2,\lambda}(A_2)Q(T).
\]
Finally we see the pluri-harmonicity.
Let $X$ and $Y$ be $n_1\times 2k$ and $n_2\times 2k$ matrices,  
respectively.
We put  
\[
\widetilde{Q}(X,Y)=Q\begin{pmatrix} X\,^{t}X & X\,^{t}Y \\ Y\,^{t}X & Y\,^{t}Y \end{pmatrix}
\]
and we must show that $\widetilde{Q}$ is pluri-harmonic for each $X$ and $Y$.
As before, for $m\times 2k$ matrices $X_1$ and $X_2$, we write 
\[
\widetilde{P}_{k,k+l}(X_1,X_2)=P_{k,k+l}\begin{pmatrix} X_1\,^{t}X_1 & X_1\,^{t}X_2 \\
X_2\,^{t}X_1 & X_2\,^{t}X_2 \end{pmatrix}.
\]
Then we have $\widetilde{Q}(X,Y)=\widetilde{P}_{k,k+l}(UX, VY)$. 
We write $U=(u_{ij})$, $V=(v_{ij})$ and put 
\[
(\xi_{i\mu})_{1\leq i\leq m,1\leq \mu\leq 2k}=UX, \quad (\eta_{i\nu})_{1\leq i\leq m,1\leq \mu \leq 2k}=VY.
\]
Then we have 
\[
\xi_{i\mu}=\sum_{l=1}^{2k}u_{il}x_{l\mu}.
\]
So we have 
\begin{align*}
\frac{\p \widetilde{Q}(X,Y)}{\p x_{l\mu}} & =\sum_{i=1}^{m}u_{il}\frac{\p \widetilde{P}_{k,k+l}}{\p \xi_{i\mu}}(UX,VY), \\
\frac{\p^2 \widetilde{Q}(X,Y)}{\p x_{l\mu}\p x_{t\mu}} & = \sum_{i,j=1}^{m}u_{il}u_{jt}
\frac{\p^2 \widetilde{P}_{k,k+l}}{\p\xi_{i\mu}\p\xi_{j\mu}}
(UX,VY).
\end{align*}
So for any $l$, $t$ with $1\leq t,l\leq n_1$, we have 
\[
\sum_{\mu=1}^{2k}\frac{\p^2 \widetilde{Q}(X,Y)}{\p x_{l\mu}\p x_{t\mu}}=
\sum_{i,j=1}^{m}u_{il}u_{jt}\sum_{\mu=1}^{2k}\frac{\p^2 \widetilde{P}_{k,k+l}}{\p\xi_{i\mu}\p\xi_{j\mu}}
(UX,VY).
\]
The last expression is $0$ by the pluri-harmonicity of $P_{k,k+l}$. 
In the same way, we can show that $\widetilde{Q}(X,Y)$ is pluri-harmonic also for $Y$.
\end{proof}

\begin{lemma}\label{repspace}
Let $n_1$, $m$ be integers such that $1\leq m \leq n_1$. Let $U$ be an $m\times n_1$ matrix of variables. 
Let $Q(U)$ be a (scalar valued) polynomial in the components of $U$ such that 
$Q(BU)=\det(B)^{l}Q(U)$ for any $B\in \GL_m(\CC)$. Then 
$Q(U)$ is a linear combination of  products $\prod_{i=1}^{l}U_{I_i}$, where 
$I_i\subset \{1,\ldots,n_1\}$ with $|I_i|=m$ and $U_{I_i}$ is the $m\times m$ minor consisting of 
$p_{\nu}$-th columns 
for $p_{\nu}\in I_{i}$. 
\end{lemma}

\begin{proof}
We regard $B=(b_{ij})_{1\leq i,j\leq m}$ as a matrix of variables and define a matrix of operators by 
\[
\frac{\p}{\p B}=\left(\frac{\p}{\p b_{ij}}\right)_{1\leq i,j\leq m}. 
\]
We consider 
\[
\det\left(\frac{\p}{\p B}\right)=\sum_{\sigma\in S_m}sgn(\sigma)\frac{\p}{\p b_{1\sigma(1)}}
\cdots \frac{\p}{\p b_{m\sigma(m)}},
\]
where $S_m$ is the permutation group on $m$ letters.
By Cayley type identity (\cite{cayley}), we have 
\[
\det\left(\frac{\p}{\p B}\right)\det(B)^l=(l)_m \det(B)^{l-1},
\]
where $(x)_m=x(x+1)\cdots (x+m-1)$ is the ascending Pochhammer symbol, 
so by the assumption $Q(BU)=\det(B)^lQ(U)$, we have 
\[
\det\left(\frac{\p}{\p B}\right)Q(BU)=(l)_m\det(B)^{l-1}Q(U).
\]
Repeating this process, we have 
\[
\det\left(\frac{\p}{\p B}\right)^lQ(UB)=(l)_m(l-1)_m\cdots(1)_mQ(U).
\]
On the other hand, writing $U=(u_{i\nu})$ and $BU=(v_{i\nu})$. we have 
$v_{i\nu}=\sum_{p=1}^{m}b_{ip}u_{p\nu}$ and 
\begin{align*}
\frac{\p}{\p b_{i\sigma(i)}}(Q(BU)) & 
=\sum_{j=1}^{m}\sum_{\nu=1}^{n_1}\frac{\p v_{j\nu}}{\p b_{i\sigma(i)}}\frac{\p Q}{\p v_{j\nu}}(BU)
 = \sum_{\nu=1}^{n_1}u_{\sigma(i)\nu}\frac{\p Q}{\p v_{i\nu}}(BU).
\end{align*}
So we have 
\begin{multline*}
\det \left(\frac{\p}{\p B}\right)Q(BU)  =
\sum_{\nu_1,\ldots,\nu_m=1}^{n_1}
\left(\sum_{\sigma\in S_m}
sgn(\sigma)u_{\sigma(1)\nu_1}\cdots u_{\sigma(m)\nu_m}\right)
\frac{\p^m Q}{\p v_{1\nu_1}\cdots \p v_{m\nu_m}}(BU). 
\end{multline*}
Here if we fix $\nu_1$, \ldots, $\nu_m$, then we have 
\[
\sum_{\sigma\in S_m}sgn(\sigma)u_{\sigma(1)\nu_1}\cdots u_{\sigma(m)\nu_m}
=
\begin{vmatrix}
u_{1\nu_1} & \cdots & u_{1\nu_{m}} \\
\vdots & \ddots & \vdots \\
u_{m\nu_1} & \cdots & u_{m\nu_{m}}
\end{vmatrix}.
\]
If $\nu_i=\nu_j$ for some $i\neq j$, then of course this is $0$ and 
if the cardinality $|I|$ of $I=\{\nu_1,\ldots,\nu_m\}$ is $m$, then this is 
$U_I$ up to sign. 
By taking $B$ to be scalar, we see that $Q(U)$ is a homogeneous polynomial of 
the total degree $ml$. so the $ml$-th derivatives of $Q(U)$ is a constant. 
So we see that 
\[
\det\left(\frac{\p}{\p B}\right)^l Q(BU)
\]
is a linear combination of $l$ products of 
$m\times m$ minors of $U$. 
Since this is equal to $(l)_m(l-1)_{m-1}\cdots (1)_{m}Q(U)$, 
we see that $Q(U)$ is a linear combination of $l$ products of minors of degree $m$
of $U$.
\end{proof}
\begin{remark} \label{rem.cuspidality-of-diff-op}
By Proposition \ref{prop.generalpol}, the operator $\mathrm{Res} \ \DD_{\lambda,n_1,n_2}$ sends $M_k(\varGamma^{(n_1+n_2)})$ to $M_{\det^k \rho_{n_1,\lambda}}(\varGamma^{(n_1)})  \otimes M_{\det^k \rho_{n_2,\lambda}}(\varGamma^{(n_2)})$. Moreover, if $l>0$ and $n_1=m$,  by the property of $P_{k,k+l}(\partial_Z)$, for $A=\begin{pmatrix} A_1 & R/2 \\ {}^t R/2 & A_2 \end{pmatrix} \in \calh_{m+n_2}(\ZZ)_{\ge 0}$ with $A_1 \in \calh_{m}(\ZZ), A_2 \in \calh_{n_2}(\ZZ)$ and $R \in M_{m,n_2}(\ZZ)$, we have
$\DD_{\lambda,m,n_2}({\bf e}(\mathrm{tr}(AZ))=0$
 unless $A_1>0$. Hence we have 
$(\mathrm{Res} \ \DD_{\lambda,m,n_2})(M_k(\varGamma^{(m+n_2)})) \subset S_{\det^k \rho_{m,\lambda} }(\varGamma^{(m)}) \otimes M_{\det^k \rho_{n_2,\lambda}}(\varGamma^{(n_2)})$.
\end{remark}
If we write 
\[
P_{k,k+l}(S)=P_{k,k+l}\begin{pmatrix} S_{11} & S_{12} \\
^{t}S_{12} & S_{22} \end{pmatrix} 
\]
for $m\times m$ matrices $S_{ij}$, then by definition,  for $B_i\in \GL_m(\CC)$,  we have 
\[
P_{k,k+l}\begin{pmatrix} B_1S_{11}\,^{t}B_1 & B_1S_{12}\,^{t}B_2 \\
B_2\,^{t}S_{12}B_1 & B_2S_{22}\,^{t}B_2 \end{pmatrix}
=
\det(B_1B_2)^{l}P_{k,k+l}(S).
\]
So 
\[
P_{k,k+l}\begin{pmatrix} (B_1U)T_{11}\,^{t}(B_1U) & (B_1U)T_{12}\,^{t}(B_2V) \\
(B_2V)\,^{t}T_{12}\,^{t}(B_1U) & (B_2V)T_{22}\,^{t}(B_2V) \end{pmatrix}
=\det(B_1B_2)^l 
P_{k,k+l}\begin{pmatrix} UT_{11}\,^{t}U & UT_{12}\,^{t}V \\
V\,^{t}T_{12}\,^{t}U & VT_{22}\,^{t}V
\end{pmatrix}.
\]
So applying  Lemma \ref{repspace}, we see that $Q(T)$ is in the representation space of 
$\rho_{n_1,\lambda}\otimes \rho_{n_2,\lambda}$.

Now we apply this for concrete cases.
For $m=1$, the polynomial $P_{k,k+l}$ is essentially the (homogeneous) Gegenbauer polynomial of degree $l$.
Based on these facts and Lemma \ref{prop.generalpol}, 
this case gives differential operators for $n=n_1+n_2$
and $\lambda$ with depth $1$, that is, the case $\lambda=(l,0,\ldots,0)$. 
The corresponding representation is the symmetric tensor representation $Sym(l)$ 
of degree $l$ so $\DD$ is from weight $k$ to $\det^k Sym(l) \otimes \det^k Sym(l)$. 
An explicit generating function of such operators for general $n=n_1+n_2$ has been already given in 
\cite[p.\ 113--114]{Ibukiyama99}.

Here we give the depth $2$ case with $\lambda=(l,l,0\ldots,0)$. 
This means $m=2$ and an explicit  generating function of $P_{k,k+l}(S)$ for $l\geq 0$ is given in 
\cite[p.\ 114]{Ibukiyama99} explicitly, where $S$ is a $4\times 4$ symmetric matrix.
Polynomials for general $n=n_1+n_2$ for $\lambda=(l,l,0,\ldots,0)$  based on Proposition 
\ref{prop.generalpol}
are given as follows.
For a $4\times 4$ symmetric matrix $S=\begin{pmatrix} S_{11} & S_{12} \\ 
^{t}S_{12} & S_{22} \end{pmatrix}$ with  $2\times 2$ matrices $S_{ij}$ of variables
for $i$, $j=1$, $2$, we put 
\begin{align*}
f_1(S) & =\det(S_{12}), \\
f_2(S) & = \det(S_{11})\det(S_{22}),  \\
f_3(S) & = \det(S).
\end{align*}
For an indeterminate $t$, we put 
\begin{align*}
\Delta_0(S,t)& =1-2f_1(S)t+f_2(S)t^2, \\
R(S,t) & = (\Delta_0(S,t)+\sqrt{\Delta_0(S,t)-4f_3(S)t^2})/2.
\end{align*}
Then for each $l$, we define a polynomial $Q_l(T,U,V)=Q_{l,n_1,n_2}(T,U,V)$ by the following generating function.
\[
\frac{1}{R(\Bbb U T\,^{t}\Bbb U,t)^{k-5/2}\sqrt{\Delta_0(\Bbb U T\,^{t}\Bbb U,t)-4f_3(\Bbb U T\,^{t}\Bbb U)t^2}}
=
\sum_{l=0}^{\infty}Q_l(T,U,V)t^{l}.
\]
Here $Q_l$ is a non-zero polynomial.
For $Z \in \HH_n$, we put  
\begin{equation}\label{defdiff}
\DD_l=\DD_{l,n_1,n_2}=Q_{l,n_1,n_2}(\p_Z,U,V).
\end{equation}
Then $\DD_l$ is a differential operator satisfying Condition \ref{condition}
for $k$ and $\lambda=(l,l,0,\ldots,0)$, where the representation space is 
realized by bideterminants as we explained. When $2k \geq n$, such differential operator 
$\DD_l$ is unique up to constant.

Actually, the generating series is easily expanded by a well-known formula, and more explicitly 
we have the following formula,

\begin{lemma}\label{lem.formulaQl}
The polynomials $f_i$ being the same as above, we put 
\[
F_i(T,U,V)=f_i(\Bbb U T ^{t}\Bbb U)
\]
for $i=1$, $2$, $3$. Then we have 
\begin{align*}
& Q_l(T,U,V)= \\
& \sum_{\begin{subarray}{c} 0\leq a,b,c \\ a+2b+2c=l \end{subarray}}
\frac{(-1)^b2^a}{a!b!c!}\bigl(k+c-\frac{3}{2}\bigr)_{a+b+c}F_1(T,U,V)^a F_2(T,U,V)^b F_3(T,U,V)^c. 
\end{align*}
\end{lemma}

\subsection{Weak pullback formula}
Let $n_1$, $n_2$ be positive integers such that $n_1 \leq n_2$. 
Let $\lambda$ be a dominant integral weight such that $\depth(\lambda)\leq n_1$. 
We consider a differential operator $\DD_{\lambda}=\DD_{\lambda,n_1,n_2}$ on $\HH_{n_1+n_2}$ satisfying Condition \ref{condition}
for $k$ and $\det^k\rho_{n_1,\lambda}\otimes \det^k\rho_{n_2,\lambda}$.
For an integer $r$ such that $\depth(\lambda) \le r$, we put $\rho_r=\det^k \rho_{r,\lambda}$. 
For a Hecke eigenform $f \in S_{\rho_r}(\varGamma^{(r)})$ we define $D(s,f)$ as
\[
D(s,f)=\zeta(s)^{-1}\prod_{i=1}^{r}\zeta(2s-2i)^{-1}
L(s-r,f,\St).
\]
For any polynomial $Q(U)$ with complex coefficients,  
we denote by $\overline{Q}(U)=\overline{Q(U)}$ the polynomial obtained by changing the 
coefficients of $Q(U)$ by the complex conjugates. 
For any function $f(Z)$, we write 
$(\theta f)(Z)=\overline{f(-\overline{Z})}$. This means that if $f(z)$ is a Fourier series of the following form
\[
f(z)=\sum_{T}a(T){\bf e}(\mathrm{tr}(TZ))
\]
 with $a(T)=a(T)(U)$ a polynomial in $U$, then we have 
\[
(\theta f)(Z)=\sum_{T}\overline{a(T)}{\bf e}(\mathrm{tr}(TZ)).
\]
So if we take $a(T)$ to be real, (which is possible), we just have $\theta f=f$.

 The next theorem is (a pullback formula) essentially due to Kozima \cite{Kozima20}.
\begin{theorem}\label{th.pullback}
Let $\lambda=(l,\ldots,l,0,\ldots,0),n_1,n_2,k$ and $\DD_{\lambda,n_1,n_2}$  be 
those in Proposition \ref{prop.generalpol}.\
Besides we assume  that $k$ is even and $n_2 \ge n_1$. Let $s \in \CC$ such that  $2\mathrm{Re}(s)+k >n_1+n_2+1$. Then for any Hecke eigenform $f \in S_{\rho_{n_1}}(\varGamma^{(n_1)})$ we have 
\begin{align*}
& \Bigl( f, \DD_{\lambda,n_1,n_2}E_{n_1+n_2,k}\Bigl(\begin{pmatrix} * & O \\ O &-\bar W \end{pmatrix},\bar s\Bigr)\Bigr) 
=c(s,\rho_{n_1})D(2s+k,f)[f]_{\rho_{n_1}}^{\rho_{n_2}}(W,s),
\end{align*}
where $c(s,\rho_{n_1})$ is a function of $s$ depending on $\rho_{n_1}$ but not on $n_2$.
\end{theorem}
\begin{remark}
This type of formula has been proved in the case $k >n_1+n_2+1$ and $s=0$ in \cite{Kozima08} in more general setting, and it can also be generalized in the case $s \not=0$ using the same method as in \cite{Kozima08} (cf. \cite{Kozima20}).
Kozima \cite{Kozima20} gave an abstract pullback formula for general $\lambda$ assuming that 
$P$ in Condition 1 is realized in his special way. The existence of $P$ satisfying Condition 1 
itself has been known in \cite{Ibukiyama99}. For further development
on realization of $P$ and exact pullback formula, see \cite{Ibukiyamalaplace}.
 \end{remark}
Now we prove Proposition \ref{prop.holomorphy-Klingen-Eisenstein} (2), that is, we prove the following statement:

\bigskip
{\it Let ${\bf k}=(\overbrace{k+l,\ldots,k+l}^m,\overbrace{k,\ldots,k}^{n-m})$ such that $l \ge 0$ 
and $k > 3m/2+1$ and 
let $f$ be a Hecke eigenform in $S_{k+l}(\varGamma^{(m)})$. 
Then $[f]^{{\bf k}}(Z,s)$ can be continued meromorphically to the whole $s$-plane as a function of $s$, and holomorphic at $s=0$. Moreover suppose that  
 $k > (n+m+3)/2$.
Then $[f]^{{\bf k}}(Z)$ belongs to $M_{{\bf k}}(\varGamma^{(n)})$.}
\begin{proof} Suppose that $l >0$.
Let $\lambda=(\overbrace{l,\ldots,l}^m,0,\ldots,0)$. Then for any $n_2 \ge m$ we have
\begin{align*}
& \Bigl( f, \DD_{\lambda,m,n_2}E_{m+n_2,k}\Bigl(\begin{pmatrix} * & O \\ O &-\bar W \end{pmatrix},\bar s\Bigr)\Bigr) 
=c(s,\rho_m)D(2s+k,f)[f]_{\rho_m}^{\rho_{n_2}}(W,s).
\end{align*}
In particular, 
\begin{align*}
& \Bigl(f, \DD_{\lambda,m,m}E_{2m,k}\Bigl(\begin{pmatrix} * & O \\ O &-\bar W \end{pmatrix},\bar s\Bigr)\Bigr) 
=c(s,\rho_m)D(2s+k,f)f(W).
\end{align*}
We claim that $c(s,\rho_m)$ is a meromorphic function of $s$, and holomorphic and non-zero at $s=0$.
We note that $\DD_{\lambda,m,m}$ coincides with the differential operator $\stackrel {\!\!\!\!\!\! \circ} {\frkD_{m,k}^{l}}$
in \cite[(1.14)]{Boecherer-Schmidt00} up to constant multiple not depending on $s$. 
Moreover,  by \cite[Theorem 3.1]{Boecherer-Schmidt00}, we have 
\[ \Bigl(f, \ \stackrel {\!\!\!\!\!\! \circ} {{\frkD}_{m,k}^{l}}E_{2m,k} \Bigl(\begin{pmatrix} * & O \\ O & -\overline{W}* \end{pmatrix},\bar s \Bigr) \Bigr)=\Omega_{k+l,l}(s) D(k+2s,f)f(W),\]
where  
\[\Omega_{k+l,l}(s)=(-1)^{{m(k+l) \over 2}}2^{-m(k+l)+{mmr+3) \over 2}-2ms}\pi^{{m(m+1) \over 2}}{\Gamma_m(k+l+s-{m \over 2})\Gamma_m(k+l-{m+1 \over 2}) \over \Gamma_m(k+s)\Gamma_m(k+s-{m \over 2})}.\]
(There is a minor misprint in \cite{Boecherer-Schmidt00}. On page 1339, line 9, ``$2^{1+n(n+1)/2-2ns}$'' should be ``$2^{1-2l+n(n+3)/2-2ns}$''.)
Therefore $c(s,\rho_m)$ coincides with $\Omega_{k+l,l}(s)$ up to constant multiple. Hence, $c(s,\rho_m)$ is a meromorphic function of $s$, and holomorphic and non-zero at $s=0$.

We have 
\begin{align*}
& \Bigl(f, \DD_{\lambda,m,n}E_{m+n,k}\Bigl(\begin{pmatrix} * & O \\ O &-\bar W \end{pmatrix},\bar s\Bigr)\Bigr) 
=c(s,\rho_m)D(2s+k,f)[f]^{{\bf k}}(W,s).
\end{align*}
As stated before, $E_{m+n,k} \Bigl(\begin{pmatrix} Z & O \\ O & {W} \end{pmatrix}, s \Bigr) $ can be continued meromorphically to the whole $s$-plane, and holomorphic at $s=0$, and therefore so is $\overline{\DD_{\lambda,m,n} E_{m+n,k} \Bigl(\begin{pmatrix} * & O \\ O & -\overline{W} \end{pmatrix},\bar s \Bigr)}$. Moreover, $D(2s+k,f)$ can be continued meromorphically to the whole $s$-plane. Moreover, since we have $k >3m/2+1$, by 
\cite[Theorem 21.3]{Shimura00}, it is holomorphic and non-zero at $s=0$. This proves the first part of the assertion.
Moreover, if $k  \ge (n+m+1)/2$, then by Proposition \ref{prop.holomorphy-Klingen-Eisenstein} (1), $\overline{ \DD_{\lambda,m,n}E_{m+n,k} \begin{pmatrix} * & O \\ O & -\overline{W} \end{pmatrix} }$ belongs to $M_{\bf k}(\varGamma^{(n)})$ as a function of $W$ except the cases $k=(n+m+2) \equiv 2 \text{ mod } 4$ and $k=(n+m+3)/2 \equiv \text{ mod } 4$. Since $c(0,\rho_m)D(k,f) \not=0$, this proves the second part of the assertion.
\end{proof}

We note that the constant $c_r$ in the pullback formula 
depends on two things. One is a 
definition of  the differential operator, and the other is 
a definition of the Petersson inner product  (P) in Section 2. First we fix a definition of the inner product.
For a while, we fix a dominant integral weight $\lambda=(l_1,\ldots,l_m,0,\ldots,0)$ of depth $m$.

For an integer $r$ such that  $m \leq r$, let   $U$ be an $m \times r$ matrix of variables and ${\bf k}_r'=(l_1,\ldots,l_m,\overbrace{0,\ldots,0}^{r-m})$, and we  take $V_{r,\lambda}=\CC[U]_{{\bf k}_r'}$ as the representation space of $\rho_{r,\lambda}$ as stated before. Here we make the convention that $\CC[U]_{{\bf k}_r'}=\CC$ if $m=0$.
We fix an inner product $\langle v,w \rangle$ of $V_{r,\lambda}$ such that 
\[
\langle\rho_{r,\lambda}(g)v,w\rangle=\langle v,\rho_{r,\lambda}
(^{t}\overline{g})w\rangle
\]
as in (H) in Section 2.
This relation is valid also for the representation 
$\rho_r=\det^k\rho_{r,\lambda}$, so we often use the 
same inner product for these. 
Now we must fix an inner product $\langle *,*\rangle$ of $V_{r,\lambda}$ explicitly.
Since we have $V_{m,\lambda}=\CC[U]_{{\bf k}_m'}$, an element of $V_{r,\lambda}$ is 
a polynomial in the components of an $m\times r$ matrix $U$ where the action of $\rho_{r,\lambda}$ is 
induced by $U\rightarrow UA$ for any $A \in \GL_r(\CC)$: 
$(\rho_{r,\lambda}(A)Q)(U)=Q(UA)$. 
For any $B\in \GL_r(\CC)$, 
we obviously have 
\[
\overline{Q(U\overline{B})}=\overline{Q}(UB).
\] 
Here RHS means to substitute the argument $U$ in $\overline{Q}(U)$ by $UB$
and LHS means to replace coefficients of $Q(U\overline{B})$ by complex conjugates. 
We put 
\[
\frac{\p}{\p U}=\left(\frac{\p}{\p u_{ij}}\right)_{1\leq i \leq m, 1 \leq j \leq r}.
\]
For two homogeneous polynomials $P(U)$ and $Q(U)$  of the same degree, we define  
\[
\langle P,Q \rangle_0=P\left(\frac{\p}{\p U}\right)\overline{Q}(U).
\]
Then we have 
\begin{equation}\label{innerproduct}
\langle\rho_{r,\lambda}(A)P,Q\rangle_0=\langle P,\rho_{r,\lambda}(^{t}\overline{A})Q\rangle_0.
\end{equation}
Indeed, if we put $V=UB$, then by the chain rule we have 
\[
\frac{\p}{\p U}=\frac{\p}{\p V}\,^{t}B,
\]
so if we put $A=\,^{t}B^{-1}$, then we have 
\[
P\left(\frac{\p}{\p U}A\right)=P\left(\frac{\p}{\p V}\right).
\]
and  
\begin{align*}
P\left(\frac{\p}{\p U}A\right)\overline{Q}(U)
& =P\left(\frac{\p}{\p V}\right)\overline{Q}(V\,^{t}A)
\\ & =P\left(\frac{\p}{\p V}\right)\overline{Q(V\,^{t}\overline{A})}.
\end{align*}
So \eqref{innerproduct} is proved.
Of course such an inner product is determined only up to constant, 
and there is no canonical choice, but we must fix something.
When $\rho_{r,\lambda}$ is scalar valued representation $\det^{l}$, 
if we define an inner product by 
$\langle P,Q\rangle_0/(l)_r(l-1)_r
\cdots (1)_r$\, then by the Cayley type identity \cite{cayley}, 
this just means to take a product of scalars, so 
\[
(f,g)=\int_{\Gamma^{(r)}\backslash \HH_r}f(Z)\overline{g(Z)}\det(\mathrm{Im}(Z))^{l-r-1}dZ.
\]
Then we have a weak type of the pullback formula. Let $k$ and $l$ be  non-negative integers. 
For the dominant integral $\lambda=(l,l,0,\ldots,0)$ of depth $m$ and integers  $n_1,n_2$ such that $2 \le n_1 \le n_2$, let $\rho_{n_1}=\det^k \rho_{n_1,\lambda}$ and $\rho_{n_2}=\det^k \rho_{n_2,\lambda}$  be the representations of $\GL_{n_1}(\CC)$ and $\GL_{n_2}(\CC)$, 
respectively, as above. We note that $m=0$ or $2$ according as $l=0$ or $l>0$. 
Moreover, let $\DD_{l,n_1,n_2}$ be the differential operator corresponding to the polynomial $Q_{l,n_1,n_2}$ in Lemma \ref{lem.formulaQl}.

\begin{theorem} \label{th.weak-pullback}
 Let the notation be as above. 
We define a subspace $\widetilde M_{\rho_{n_1}}(\varGamma^{(n_1)})$ of 
$M_{\rho_{n_1}}(\varGamma^{(n_1)})$  as 
\[\widetilde M_{\rho_{n_1}}(\varGamma^{(n_1)})=\{F \in M_{\rho_{n_1}}(\varGamma^{(n_1)}) \ | \  \Phi^{{n_1}}_{2}(F) \in S_{\rho_2}(\varGamma^{(2)}) \} \]
or $M_{\rho_{n_1}}(\varGamma^{(n_1)})$ according as  $l>0$ or $l=0$.
Let $\{f_{2,j} \}_{1 \le j \le d(2)}$ be a basis of $S_{\rho_2}(\varGamma^{(2)})$ consisting of Hecke eigenforms, and take Hecke eigenforms $\{F_j \}_{d(2)+1 \le j \le d}$ so that $\{[f_{2,j}]_{\rho_2}^{\rho_{n_1}} \ (1 \le j \le d(2)), \ F_j \ (d(2)+1 \le j \le d) \}$
forms a basis of $\widetilde M_{\rho_{n_1}}(\varGamma^{(n_1)})$. 
Suppose that $k \ge \max((n_1+n_2+1)/2,6)$ and that neither $k=(n_1+n_2+2)/2 \equiv  2 \text{ mod } 4$ nor $k=(n_1+n_2+3) \equiv 2 \text{ mod } 4$.
Then 
\begin{align*}
\DD_{l,n_1,n_2} E_{n_1+n_2,k}\begin{pmatrix} Z & O \\ O & W \end{pmatrix}
& = c(0,\rho_2)\sum_{j=1}^{d(2)} { D(k,f_{2,j})  \over (f_{2,j}, \ f_{2,j})} [f_{2,j}]_{\rho_2}^{\rho_{n_1}}(Z)(U)  [ \theta f_{2,j}]_{\rho_2}^{\rho_{n_2}}(W)(V) \\
& \quad +\sum_{j=d(2)+1} ^d F_j(Z)(U)G_j(W)(V) \qquad (Z \in \HH_{n_1}, W \in \HH_{n_2}), \end{align*}
where   $G_j$ is a certain element of $M_{\rho_{n_2}}(\varGamma^{(n_2)})$.
Here, $U$ and $V$ are  $m \times n_1$ and $m \times n_2$ matrices  of variables, respectively, and  we regard $[f_{2,j}]_{\rho_2}^{\rho_{n_1}}$ and $F_j$ (resp. $[\theta f_{2,j}]_{\rho_2}^{\rho_{n_2}}$ and $G_j$)  as 
elements of  $\mathrm{Hol}[U]_{{\bf k}_{n_1}'}$ (resp. $\mathrm{Hol}[V]_{{\bf k}_{n_2}'}$).  Moreover we have 
\[c(0,\rho_2)=\frac{2^{9-2(k+2l)}(-1)^{k+l}\pi^3 (2k-3)_l(2k-1)_{2l-3}}
{l!}.
\]
 \end{theorem}
\begin{proof} First suppose that  $l >0$.
Let $d_0=\dim M_{\rho_{n_1}}(\varGamma^{(n_1)})$ and $\{F_j\}_{d+1 \le j \le d_0}$ 
be a basis of the orthogonal complement of $\widetilde M_{\rho_{n_1}}(\varGamma^{(n_1)})$
in $M_{\rho_{n_1}}(\varGamma^{(n_1)})$ with respect to the Petersson inner product. Then we have
\begin{align*}
\DD_{l,n_1,n_2}E_{n_1+n_2,k}\begin{pmatrix}Z & O \\ O & W \end{pmatrix}
&=\sum_{j=1}^{d(2)} [f_{2,j}]_{\rho_2}^{\rho_{n_1}}(Z)(U) G_j(W)(V) +\sum_{j=d(2)+1}^{d_0} F_j(Z)(U)G_j(W)(V).
\end{align*}
For  $m_1 \ge m_2$ and $l_1 \ge l_2$ and 
$H(Z,W)(U,V)=\sum_j A_j(Z)(U)B_j(W)(V) \in M_{\rho_{m_1}}(\varGamma^{(m_1)}) \otimes M_{\rho_{l_1}}(\varGamma^{(l_1)})$, we define $\Phi_{m_2}^{m_1} \otimes \Phi_{l_2}^{l_1}(H(Z,W)(U,V))$ as
\[\Phi_{m_2}^{m_1} \otimes \Phi_{l_2}^{l_1}(H(Z,W)(U,V))=\sum_j \Phi_{m_2}^{m_1}(A_j(Z)(U))\Phi_{l_2}^{l_1}(B_j(W)(V)).\]
We note that we have $\Phi^{n_1}_{2}(F_j)=0$ for $d(2)+1 \le j \le d$. Hence we have 
\begin{align} 
&(\Phi_{2}^{n_1} \otimes \Phi_{n_2}^{n_2})\Bigl(\DD_{l,n_1,n_2}E_{n_1+n_2,k}\begin{pmatrix} Z^{(2)} & O \\ O & W \end{pmatrix}\Bigr) \notag \\
&= \sum_{j=1}^{d(2)}f_{2,j}(Z^{(2)})(U^{(2)}) G_j(W)(V) +\sum_{j=d+1}^ {d_0}\Phi_{2}^{n_1}(F_j(Z)(U))\Phi_{n_2}^{n_2}(G_j(W)(V)),\notag \end{align}
where $Z^{(2)}=\mathrm{pr}^{n_1}_2(Z)$ for $Z \in \HH_{n_1}$, and $U^{(2)}=(u_{ij})_{1 \le i,j \le 2}$ for $U=(u_{ij})_{1 \le i \le 2, 1 \le j \le n_1}$.
On the other hand, let $A=\begin{pmatrix} A_1 & O & R_1/2 \\ O & O & O \\ {}^t R_1/2 & O & D_1 
\end{pmatrix} \in \calh_{n_1+n_2}(\ZZ)_{\ge 0}$ with $A_1 \in \calh_{2}(\ZZ), \ D_1 \in \calh_{n_2}(\ZZ)$ and $R_1 \in M_{2,n_2}(\ZZ)$.
Then we have
\begin{align*}
&(\Phi_{2}^{{n_1}} \otimes \Phi_{n_2}^{{n_2}})
\Bigl(\mathrm{Res} \ \DD_{l,n_1,n_2} \Bigl(
{\bf e}\Bigl({\rm tr}\Bigl(A\begin{pmatrix} Z & Z_{12} \\ {}^t Z_{12} & W \end{pmatrix}\Bigr)\Bigr) \Bigr)\Bigr)\\
&=\mathrm{Res} \ \DD_{l,2,n_2} \Big({\bf e}\Bigl(\mathrm{tr}\Big(\begin{pmatrix} A_1 & R_1/2 \\ {}^tR_1/2 & D_1 \end{pmatrix}\begin{pmatrix} Z^{(2)} & Z_{12}^{(2)} \\ {}^tZ_{12}^{(2)} & W \end{pmatrix}\Big)\Big)\Big)
\end{align*} for
$\begin{pmatrix} Z & Z_{12} \\ {}^t Z_{12} & W \end{pmatrix} \in \HH_{n_1+n_2}$ with $Z \in \HH_{n_1},\, W \in \HH_{n_2}$ and $Z_{12} \in M_{n_1,n_2}(\CC)$. Here $Z_{12}^{(2)}$ is the upper-left $2 \times n_2$ block of $Z_{12}$.
Hence we have
\begin{align} 
&(\Phi_{2}^{n_1} \otimes \Phi_{n_2}^{n_2})\Bigl(\DD_{l,n_1,n_2}E_{n_1+n_2,k}\begin{pmatrix} Z & O \\ O & W\end{pmatrix}\Bigr)
=\DD_{l,2,n_2} E_{2+n_2,k}\begin{pmatrix} Z^{(2)} & O\\ O & W \end{pmatrix}, \notag \end{align}
and therefore, by Remark \ref{rem.cuspidality-of-diff-op},
$(\Phi_{2}^{n_1} \otimes \Phi_{n_2}^{n_2})\Bigl(\DD_{l,n_1,n_2}E_{n_1+n_2,k}\begin{pmatrix} Z & O \\ O & W\end{pmatrix}\Bigr)$ belongs to 
$S_{\rho_2}(\varGamma^{(2)}) \otimes M_{\rho_{n_2}}(\varGamma^{(n_2)})$.
We note that $\Phi^{\rho_{n_1}}_{\rho_2}(F_j) \not\in S_{\rho_2}(\varGamma^{(2)})$ for $d+1 \le j \le d_0$. 
Hence we have $G_j(W)(V)=0$ for $d+1 \le j \le d_0$
and
\[\DD_{l,2,n_2} E_{2+n_2,k}\begin{pmatrix} Z^{(2)} & O\\ O & W \end{pmatrix}=\sum_{j=1}^{d(2)}f_{2,j}(Z^{(2)})(U^{(2)})  G_j(W)(V).\]
By Theorem \ref{th.pullback}, we see that
\[(f_{2,j},\ f_{2,j})\theta G_j(W)(V)=c(0,\rho_2)D(k,f_{2,j})[f_{2,j}]_{\rho_2}^{\rho_{n_2}}(W)(V).\]
We note that $D(k,f_{2,j})$ is real number, and hence 
\[G_j(W)(V)=c(0,\rho_2){D(k,f_{2,j}) \over (f_{2,j},\ f_{2,j})}[\theta f_{2,j}]_{\rho_2}^{\rho_{n_2}}(W)(V).\]
This proves the first part of the assertion. 
By using the same argument as above, we have
\[\DD_{l,2,2} E_{4,k}\begin{pmatrix} Z^{(2)} & O\\ O & W^{(2)} \end{pmatrix}=\sum_{j=1}^{d(2)}f_{2,j}(Z^{(2)})(U^{(2)})  \theta f_{2,j}(W^{(2)})(V^{(2)}),\]
where $W^{(2)}=\mathrm{pr}^{n_1}_2(Z)$ for $W \in \HH_{n_2}$, and $V^{(2)}=(v_{ij})_{1 \le i,j \le 2}$ for $V=(v_{ij})_{1 \le i \le 2, 1 \le j \le n_1}$.
Let 
\[\iota:S_{k+l}(\varGamma^{(2)}) \longrightarrow S_{\rho_2}(\varGamma^{(2)}) \]
be the isomorphism stated before. Take an element $g_{2,j} \in S_{k+l}(\varGamma^{(2)})$ such that $f_{2,j}=\iota(g_{2,j})$. Then $\{ g_{2,j}\}_j$ forms a basis of $S_{k+l}(\varGamma^{(2)}) $ and
\[D(k,f_{2,j})=D(k,g_{2,j})\]
and
\[(f_{2,j},\ f_{2,j})=(g_{2,j}, \ g_{2,j}).\]
We also note that 
\[\DD_{l,2,2}=(\det U^{(2)}\det V^{(2)})^l\DD_{l.k+l},\]
 where 
$\DD_{k,k+l}=P_{k,k+l}(\partial _{\widetilde Z})$ with $\widetilde Z \in \HH_4$.
Hence we have 
\begin{align*}
\DD_{k,k+l} E_{4,k}\begin{pmatrix} Z^{(2)} & O\\ O & W^{(2)} \end{pmatrix}=c(0,\rho_2) \sum_{j=1}^{d(2)} {D(k,g_{2,j}) \over (g_{2,j}, \ g_{2,j})}g_{2,j}(Z^{(2)}) \otimes \theta g_{2,j}(W^{(2)}). \end{align*}
Hence we have
\[\Big(g_{2,j}, \DD_{k,k+l} E_{4,k}\begin{pmatrix} * & O\\ O & -\overline{W^{(2)}} \end{pmatrix}\Bigr)=c(0,\rho_2)D(k,g_{2,j})g_{2,j}(W^{(2)}).\]
On the other hand, let 
$\stackrel {\!\!\!\!\!\! \circ} {{\frkD}_{2,k}^{l}}$ be the differential operator in \cite[(1.14)]{Boecherer-Schmidt00}. Then by 
\cite[Theorem 3.1]{Boecherer-Schmidt00} we have  
\[ \Bigl(g_{2,j}, \ \stackrel {\!\!\!\!\!\! \circ} {{\frkD}_{2,k}^{l}}E_{4,k}\begin{pmatrix} * & O \\ O & -\overline{W^{(2)}} \end{pmatrix}\Bigr)=\widetilde c_2 D(k,g_{2,j})g_{2,j}(W^{(2)})\]
with 
\[\widetilde c_2=(-1)^{k+l}2^{6-2(k+l)}\pi^3 {\Gamma(k+l-1)\Gamma(k+l-3/2)^2\Gamma(k+l-2) \over \Gamma(k)\Gamma(k-1/2)\Gamma(k-1)\Gamma(k-3/2)}.\]
By page 71 in \cite{Katsurada10}, we have 
\[\DD_{\lambda}=d_{k,l}\stackrel {\!\!\!\!\!\! \circ} {{\frkD}_{2,k}^{l}},
\]
with 
\[d_{k,l}={\begin{pmatrix} 2k+2l-5 \\ l\end{pmatrix} \over \prod_{i=1}^l (k+l-2-i/2)(k+l-3/2-i/2)}.\]
Hence we have $c(0,\rho_2)=d_{k,l}\widetilde c_2$, and by a simple computation we prove  the assertion.
Next suppose that $l=0$. Then the assertion can be proved using the same argument as above.

\end{proof}
\begin{remark}
(1) The second part of the assertion can be also proved by the fact that the differential operator is realized uniformly in 
Lemma \ref{lem.formulaQl} and its operation on the automorphy factor is essentially the same as the case 
when $U=V=1_2$. \\
(2) If $k > n_1+n_2+1$, then $[f]_{\rho_r}^{\rho_{n_1}}(Z,U)$ and $[f]_{\rho_r}^{\rho_{n_2}}(W,V)$ are holomorphic modular forms for any Hecke eigenform $f$  in $S_{\rho_r}(\varGamma^{(r)})$, and we can get an explicit pullback formula  (cf. Theorems \ref{th.explicit-pullback} and \ref{th.constant}).
 However, if $k \le n_1+n_2+1$, it does not necessarily hold. This is why we say that  the formula in the above theorem is a weak type of pullback formula. We note that it is sufficient for proving our main results in Section 8.
\end{remark}
\section{Congruence for Klingen-Eisenstein lifts}
To explain why Conjecture \ref{conj.main-conjecture} is reasonable, we  consider congruence for Klingen-Eisenstein series, which is a generalization of  \cite{Katsurada-Mizumoto12}. For $\lambda=(k-l,k-l,0,0)$ with $k \ge l$ and $2 \le m \le 4$, let $(\rho_{m,\lambda},V_{m,\lambda})$ be the representation of $\GL_m(\CC)$ defined in the previous section, and put $\rho_m=\det^l \otimes \rho_{m,\lambda}$ and ${\bf k}_m'=(k-l,k-l,\overbrace{0,\ldots,0}^{m-2})$ and 
${\bf k}_m=(k,k,\overbrace{l,\ldots,l}^{m-2})$. Let $U$ and $V$ be  $2 \times n_1$ and $2 \times n_2$ matrices  of variables, respectively. Then we recall that $V_{n_1,\lambda}=\CC[U]_{{\bf k}_{n_1}'}, V_{n_2,\lambda}=\CC[V]_{{\bf k}_{n_2}'}$ and that every element $F$ of $M_{{\rho}_{n_1}}(\varGamma^{(n_1)}) \otimes M_{{\rho}_{n_2}}(\varGamma^{(n_2)})$ is expressed as
\[F(Z_1,Z_2)=\sum_{A_1 \in \calh_{n_1}(\ZZ)_{\ge 0}, A_2 \in \calh_{n_2}(\ZZ)_{\ge 0}} c(A_1,A_2;F)(U,V) {\bf e}(\mathrm{tr}(A_1Z_1+A_2Z_2))\]
with $c(A_1,A_2;F)(U,V) \in  \CC[U,V]_{{\bf k}_{n_1}',{\bf k}_{n_2}'}$.   For a  subring $R$ of $\CC$, we denote by $(M_{{\rho}_{n_1}}(\varGamma^{(n_1)}) \otimes M_{{\rho}_{n_2}}(\varGamma^{(n_2)}))(R)$ the submodule of $M_{{\rho}_{n_1}}(\varGamma^{(n_1)}) \otimes M_{{\rho}_{n_2}}(\varGamma^{(n_2)})$ consisting of all $F$'s such that $c(A_1,A_2;F)(U,V) \in R[U,V]_{{\bf k}_{n_1}',{\bf k}_{n_2}'}$ for all $A_1 \in \calh_{n_1}(\ZZ)_{\ge 0}, A_2 \in \calh_{n_2}(\ZZ)_{\ge 0}$. 
We also note that every element $F$ of 
$M_{{\rho}_{n_1}}(\varGamma^{(n_1)}) \otimes V_{n_2,\lambda}$ is  expressed as
\[F(Z_1)=\sum_{A_1 \in \calh_{n_1}(\ZZ)_{\ge 0}} c(A_1;F)(U,V) {\bf e}(\mathrm{tr}(A_1Z_1))\]
with $c(A_1;F)(U,V) \in  \CC[U,V]_{{\bf k}_{n_1}',{\bf k}_{n_2}'}$. We then define a submodule 
$(M_{{\rho}_{n_1}}(\varGamma^{(n_1)}) \otimes V_{n_2,\lambda})(R)$ of   
$M_{{\rho}_{n_1}}(\varGamma^{(n_1)}) \otimes V_{n_2,\lambda}$ consisting of all $F$'s such that 
$c(A_1;F)(U,V) \in R[U,V]_{{\bf k}_{n_1}',{\bf k}_{n_2}'}$ for all $A_1 \in \calh_{n_1}(\ZZ)_{\ge 0}$.

For positive integers $n$ and $l$, put
\[Z(n,l)=\zeta(1-l)\prod_{j=1}^{[n/2]} \zeta(1+2j-2l).\]
We define $\widetilde E_{n,l}$ as 
\[\widetilde E_{n,l}(Z)=Z(n,l)E_{n,l}(Z)\]
and we set 
\begin{align*}
\cale(Z_1,Z_2)&=\cale_{k,l,n_1,n_2}(Z_1,Z_2)=(k-l)!(2\pi \sqrt{-1})^{-2(k-l)}\DD_{k-l,n_1,n_2} \widetilde E_{n_1+n_2,l}\begin{pmatrix} Z_1 & O \\ O & Z_2 \end{pmatrix}.
\end{align*}
Moreover, for positive integers $m,l$ and a Hecke eigenform $F \in S_{k}(\varGamma^{(2)})$ put
\[\calc_{m,l}(F)={Z(m,l) \over Z(4,l)}{\bf L}(l-2,F,\St).\]
We also use the same symbol $\calc_{m,l}(f)$ to denote 
the value $\displaystyle {Z(m,l) \over Z(4,l)}{\bf L}(l-2,f,\St)$ for a Hecke eigenform $f \in S_{\rho_2}(\varGamma^{(2)})$.
As sated before, we have the following isomorphism:
\begin{align*}
\iota:S_k(\varGamma^{(2)}) \ni F \mapsto \widetilde F:=F (\det U)^{k-l} \in S_{\rho_2}(\varGamma^{(2)}), \tag{IS}
\end{align*}
where $U$ is $2 \times 2$ matrix of variables. Then we note that
$\calc_{m,l}(\widetilde F)=\calc_{m,l}(F)$ for a Hecke eigenform $F \in S_k(\varGamma^{(2)})$.

Now, for our later purpose, we rewrite a special case of  Theorem \ref{th.weak-pullback} as follows.
\begin{proposition} \label{prop.pullback1}
Let $n_1,n_2$ be integers such that $2 \le n_1 \le n_2 \le 4$ and  let $k,l$ be  even positive integers such that $k \ge l$. Then we have 
\begin{align*}
\cale_{k,l,n_1,n_2} (Z_1,Z_2) 
&=\gamma_2 \sum_{j=1}^{d(2)} \calc_{n_1+n_2,l}(f_{2,j})[f_{2,j}]_{\rho_2}^{\rho_{n_1}}(Z_1)(U)[\theta f_{2,j}]_{\rho_2}^{{\rho}_{n_2}}(Z_2)(V)\\
& \quad +\sum_{j=d(2)+1}^d F_j(Z_1)(U)\widetilde G_j(Z_2)(V),
\end{align*}
where $\gamma_2$ is a certain rational number which is $p$-unit for any prime number $p>2k$, and $\widetilde G_j(Z_2)(V)$ is an element of $M_{{\rho}_{n_2}}(\varGamma^{(n_2)})$.
\end{proposition}
We write $\cale(Z_1,Z_2)$ as 
\begin{align*}
\label{partial}
\cale(Z_1,Z_2)=\sum_{N \in \calh_{n_2}} g_{(k,l,n_1,n_2),N}^{(n_1)}(Z_1){\bf e}(\mathrm{tr} (NZ_2)). \tag {$\ast$}
\end{align*} 
Then  $g_N^{(n_1)}=g_{(k,l,n_1,n_2),N}^{(n_1)}$ belongs to $M_{{\rho}_{n_1}}(\varGamma^{(n_1) })  \otimes V_{n_2,\lambda}$.  To consider congruence between Klingen-Eisenstein lift and another modular form of the same weight, we rewrite the above proposition as follows:
 \begin{corollary} \label{cor.pullback2} 
Under the same notation and the assumption as above,  let $N \in \calh_{n_2}(\ZZ)_{>0}$. 
Then, 
\begin{align*}
g_N^{(n_1)}(Z_1)
&=\gamma_2
 \sum_{j=1}^{d(2)} \calc_{n_1+n_2,l}(f_{2,j})[f_{2,j}]_{\rho_2}^{{\rho}_{n_1}}(Z_1)(U)\overline{a(N,[f_{2,j}]_{\rho_2}^{{\rho}_{n_2}})(V)}\\
& \quad +\sum_{j=d(2)+1}^d F_j(Z_1)(U)a(N,\widetilde G_j)(V).
\end{align*}
\end{corollary}

Observe that the first term in the right-hand side
of the above  is invariant if
we multiply $f_{2,j}$ by an element of $\CC^{\times}$.

To see the Fourier expansion of $\widetilde E_{n,k}(Z)$, we  review the polynomial $F_p(B,X)$ attached to the local Siegel series $b_p(B,s)$ for an element $B$ of $\calh_n(\ZZ_p)$ (cf.\ \cite{Katsurada99}). 
 We define $\chi_p(a)$ for $a \in {\QQ}^{\times}_p $ as follows:
 $$\chi_p(a):=  
 \left\{\begin{array}{cl} 
  +1            & {\rm if } \ {\QQ}_p(\sqrt {a})={\QQ}_p ,\\
  -1            & {\rm if}  \ {\QQ}_p(\sqrt {a})/{\QQ}_p \ {\rm is \ quadratic \ unramified},\\
  0             & {\rm if}  \ {\QQ}_p(\sqrt {a})/{\QQ}_p \ {\rm is \ quadratic \ ramified}. 
\end{array}\right.$$
For an element  $B \in \calh_n(\ZZ_p)^{\rm nd}$ with $n$ even,  we define $\xi_p(B)$ by 
$$\xi_p(B):=\chi_p((-1)^{n/2}\det B).$$ 
For a non-degenerate half-integral matrix $B$ of size $n$ over ${\ZZ}_p$ define a polynomial $\gamma_p(B,X)$ in $X$ by
$$\gamma_p(B,X):=
\left\{
\begin{array}{ll}
(1-X)\prod_{i=1}^{n/2}(1-p^{2i}X^2)(1-p^{n/2}\xi_p(B)X)^{-1} & \ {\rm if} \ n \ {\rm is \ even}, \\
(1-X)\prod_{i=1}^{(n-1)/2}(1-p^{2i}X^2) & \ {\rm if} \ n \ {\rm is \ odd}.
\end{array}
\right.$$
Then it is well known that  there exists a unique polynomial $F_p(B,X)$ in $X$ over ${\ZZ}$  with constant term $1$ such that 
$$b_p(B,s) =\gamma_p(B,p^{-s})F_p(B,p^{-s})$$ (e.g. \cite{Katsurada99}).                                   
For  $B \in \calh_n(\ZZ)_{>0}$ with $n$ even, 
 let ${\mathfrak d}_B$ be the discriminant of  ${\QQ}(\sqrt{(-1)^{n/2}\det B})/{\QQ}$,  and $\chi_B=({\frac{{\mathfrak d}_B}{*}})$  the Kronecker character corresponding to ${\QQ}(\sqrt{(-1)^{n/2}\det B})/{\QQ}$. We note that we have $\chi_B(p)=\xi_p(B)$ for any prime $p.$


We define a polynomial $F_p^*(T,X)$ for any $T \in \calh_n(\ZZ_p)$ which is not-necessarily non-degenerate as follows:
For an element $T \in \calh_n(\ZZ_p)$ of rank $m \ge 1,$ there exists an element 
$\tilde T \in \calh_m(\ZZ_p)^{{\rm nd}}$ such that $T \sim_{\ZZ_p} \tilde T \bot O_{n-m}.$ We note that  
$F_p(\tilde T,X)$ does not depend on the choice of $\tilde T.$ Then we put $F_p^\ast(T,X)=F_p(\tilde T,X).$  
For an element $T \in \calh_n(\ZZ)_{\ge 0}$ of rank $m \ge 1,$ there exists an element $\tilde T \in 
\calh_m(\ZZ)_{>0}$ such that $T \sim_{\ZZ} \tilde T \bot O_{n-m}.$
Then $\chi_{\tilde T}$ does not depend on the choice of $\widetilde T$. We write $\chi_T^{\ast}=\chi_{\tilde T}$ if $m$ is even.

\begin{proposition}\label{prop.fc-Siegel}
 Let $k\in 2\ZZ$. Assume that $k \ge (n+1)/2$ and that neither $k=(n+2)/2 \equiv 2 \text{ mod } 4$ 
 nor $k=(n+3)/2 \equiv 2 \text{ mod } 4$. Then 
for $T \in \calh_n(\ZZ)_{\ge 0}$ of rank $m,$ we have
\begin{align}
a(T,\widetilde E_{n,k})
&=2^{[(m+1)/2]}\prod_{p \mid  \det (2\widetilde T)} F_p^\ast(T,p^{k-m-1})\nonumber\\
&\times \left\{
\begin{array}{ll}
\prod_{i=m/2+1}^{[n/2]} \zeta(1+2i-2k) L(1+m/2-k,\chi_T^\ast)
 & \ {\rm if} \ m \ {\rm is \ even  }, \\
\prod_{i=(m+1)/2}^{[n/2]} \zeta(1+2i-2k) & \ {\rm if} \ m \ {\rm is \ odd}.
\end{array}
\right.
\nonumber
\end{align}
 Here we make the convention $F_p^\ast(T,p^{k-m-1})=1$ and $L(1+m/2-k,\chi_T^\ast)=\zeta(1-k)$ if $m=0.$ 
\end{proposition}
To consider the integrality of $a(T,\widetilde E_{n,k})$, we provide the following lemma.
\begin{lemma}\label{lem.precise-Kitaoka-polynomial}
Let $T \in \calh_m(\ZZ_p)^{\mathrm{nd}}$. Then, we have $F_p(p^{-[(m+1)/2]}T,X) \in \ZZ[X]$.
\end{lemma}
\begin{proof}
The assertion has been proved in \cite[Lemma 15]{Ikeda01} in the case $m$ is even, and the assertion for odd case can also be proved in the same manner.

\end{proof}
\begin{proposition} \label{prop.p-integrality-Siegel}
Let the notation and the assumption be as in Proposition \ref{prop.fc-Siegel}. Then,  we have $\widetilde E_{n,k}$ belongs to $M_k(\varGamma^{(n)})(\QQ)$. In particular, for any prime number $p >  2k$, $\widetilde E_{n,k}$ belongs to $M_k(\varGamma^{(n)})(\ZZ_{(p)})$. 
\end{proposition}
\begin{proof} The first assertion is well known. 
We prove the second assertion. Let  $T \in \calh_n(\ZZ)_{\ge 0}$ of rank $m.$ Since we have $k \ge (n+1)/2$, by Lemma \ref{lem.precise-Kitaoka-polynomial}, the product $\prod_{p |  \det (2\widetilde T)} F_p^*(T,p^{k-m-1})$ is an integer. Moreover, since we have $p >2k$, by the theorem of v. Staut-Clausen, the value $\zeta(1-k)$ and $\zeta(1+2i-2k)$ for a positive integer $i \le [n/2]$  belong to $\ZZ_{(p)}$, and by \cite[(5.1), (5.2)]{Bo84}, the value  $L(1+m/2-k,\chi^*_T)$ belongs to $\ZZ_{(p)}$ if $m \ge 2$ is even. Thus the assertion follows from Proposition \ref{prop.fc-Siegel}.

 \end{proof}

\begin{proposition}\label{prop.rationality1}
 Let the notation and the assumption be as in Proposition \ref{prop.pullback1}. Then
\[\cale_{k,l,n_1,n_2}(Z_1,Z_2) \in (M_{{\rho}_{n_1}}(\varGamma^{(n_1)}) \otimes M_{{\rho}_{n_2}}(\varGamma^{(n_2)}))(\QQ),\]
and more precisely 
\[\cale_{k,l,n_1,n_2}(Z_1,Z_2) \in (M_{{\rho}_{n_1}}(\varGamma^{(n_1)}) \otimes M_{{\rho}_{n_2}}(\varGamma^{(n_2)}))(\ZZ_{(p)})\]
for any prime number $p>2k$.
\end{proposition}
\begin{proof}
For $T_1\in \calh_{n_1}$ and $T_2\in \calh_{n_2}$, put 
\begin{align*}
\label{explicit-epsilon}
\epsilon(T_1,T_2)(U,V)&=\epsilon_{k,l,n_1,n_2}(T_1,T_2)(U,V) \tag {E} \\
&=\sum_{ R \in M_{n_1,n_2}({\ZZ})} a\Bigl(\begin{pmatrix} T_1 &R/2\\ {}^tR/2 &T_2\end{pmatrix}, \widetilde 
E_{n_1+n_2,l}\Bigr)
 \times Q_{k-l,n_1,n_2}\Bigl(\begin{pmatrix} T_1 &R/2\\ {}^tR/2 &T_2\end{pmatrix},U,V\Bigr),
\end{align*}
where $Q_{k-l,n_1,n_2}$ is the polynomial in Section 5.1.2.
Then we have 
\begin{align*} 
\cale_{k,l,n_1,n_2} (Z_1,Z_2)
&=\sum_{T_1 \in \calh_{n_1}(\ZZ)_{\ge 0},T_2 \in \calh_{n_2}(\ZZ)_{\ge 0}} \epsilon(T_1,T_2)(U,V){\bf e}(\mathrm{tr}(T_1Z_1+T_2Z_2)). 
\end{align*}
Hence the assertion follows from Proposition \ref{prop.p-integrality-Siegel}.
\end{proof}

\begin{corollary} \label{cor.rationality1}
For each  $N \in \calh_{n_2}(\ZZ)_{>0}$ let $g_{N}^{(n_1)}$ be that defined above. 
Then
\[g_{(k,l,n_1,n_2),N}^{(n_1)}(Z_1) \in ( M_{{\rho}_{n_1}}(\varGamma^{(n_1)}) \otimes V_{n_2,\lambda})(\QQ)  \]
and moreover
\[g_{(k,l,n_1,n_2),N}^{(n_1)}(Z_1) \in (M_{{\rho}_{n_1}}(\varGamma^{(n_1)})  \otimes  V_{n_2,\lambda} )(\ZZ_{(p)})\]
for any prime number $p >2k$.
\end{corollary}
\begin{proof}
$g_{(k,l,n_1,n_2),N}^{(n_1)}(Z_1)$  is expressed as
$$g_{(k,l,n_1,n_2),N}^{(n_1)}(Z_1)=\sum_{T \in \calh_{n_1}(\ZZ)}\epsilon_{k,l,n_1,n_2}(T,N)(U,V){\bf e}(\mathrm{tr}(TZ_1)).$$
Hence the assertion directly follows from the above proposition.
\end{proof}

\begin{proposition} \label{prop.rationality3}
Let the notation and the assumptions be as in Theorem \ref{th.weak-pullback}, and  let $2 \le m \le 4$.
Then for any  Hecke eigenform $f$  in $S_{{\rho}_2}(\varGamma^{(2)})(\QQ(f))$,
$[f]_{{\rho}_2}^{{\rho}_m} \in M_{{\rho}_m}(\varGamma^{(m)})(\QQ(f))$.
\end{proposition}
\begin{proof}
The assertion in the case  $k=l$ has been proved by Mizumoto \cite{Mizumoto91}, and the other case can also be proved by using the same method. \end{proof}

\begin{proposition} \label{prop.rationality2} 
Let  the notation and the assumption be as in Proposition \ref{prop.pullback1}. Let $f$ be a Hecke eigenform in $S_{k}(\varGamma^{(2)})$. Then, for any $N \in \calh_{n_2}(\ZZ)_{>0 }$ and $N_1 \in \calh_2(\ZZ)_{>0}$,
the value $\calc_{n_2+2,l}(f)\overline{a(N,[\widetilde f]_{{\rho}_2}^{{\rho}_{n_2}})(V)} a(N_1,f)$ belongs to $\QQ(f)[V]_{{\bf k}_2'}$, where $\widetilde f$ is that in $\mathrm{(IS)}$.
\end{proposition}
\begin{proof}
The value in question remains unchanged if we replace $f$ by $c f$ with $c \in \CC^\times$.
Moreover we can take $c \in \CC^\times$ so that  $cf \in S_{{\bf k}_2}(\varGamma^{(2)})(\QQ(f))$. Thus the assertion follows from Proposition \ref{prop.rationality3} remarking that $a(N_1,[\widetilde f]_{\rho_2}^{\rho_2})=a(N_1,f)(\det U)^{k-l}$.
\end{proof}

The following lemma can be proved by a careful analysis of the proof of \cite[Lemma 5.1]{Katsurada08}.
\begin{lemma} \label{lem.congruence} 
Let $F_1,\ldots,F_d$ be Hecke eigenforms in  $M_{{\rho}_{n_1}}(\varGamma^{(n_1)})$ linearly independent over $\CC$. Let  $K$ be the composite  field $\QQ(F_1)\cdots \QQ(F_d)$, $\frkO$ the ring of integers in $K$ and $\frkp$ a prime ideal of $K$.
Let  $G(Z,U,V) \in (M_{{\rho}_{n_1}}(\varGamma^{(n_1)}) \otimes  V_{n_1,\lambda})(\frkO_{(\frkp)})$ and assume the following conditions
\begin{itemize}
\item[(1)] $G$ is expressed as 
\[G(Z,U,V)=\sum_{i=1}^d c_i(V)  F_i(Z)(U)\]
with $c_i(V) \in  V_{n_1,\lambda}$. 
\item[(2)] $c_1(V) a(A_1,F_1)(U) \in (V_{n_1,\lambda} \otimes V_{n_1,\lambda})(K)$ and $\ord_{\frkp}(c_1(V)a(A_1,F_1)(U)) <0$ for some $A_1 \in \calh_{n_1}(\ZZ)$. 
\end{itemize}
Then there exists $i \not=1$ such that
\[F_i \equiv_{\ev} F_1 \pmod {\frkp}.\]
\end{lemma}

\begin{theorem} \label{th.congruence-Klingen2}
Let $k$ and $l$ be positive even integers such that $k \ge l \ge 6$ and put ${\bf k}=(k,k,l,l )$ and $\widetilde M_{\bf k}(\varGamma^{(4)})=\widetilde M_{\rho_4}(\varGamma^{(4)})$.  Let $F \in S_k(\varGamma^{(2)})$ be a Hecke eigenform, and  $\frkp$ a prime ideal of $\QQ(F)$.  Suppose that 
$\frkp$ divides $|a(A_1,F)|^2{\bf L}(l-2,F,\St)$ and does not divide 
\[\calc_{8,l}(F)a(A_1,F)\overline{a(A,[F]^{{\bf k}})}\]
 for some $A_1 \in \calh_2(\ZZ)_{>0}$ and $A \in \calh_4(\ZZ)_{>0}$, where $[F]^{\bf k}=[\widetilde F]_{\rho_2}^{\rho_4}$ as stated in Section 1.
Then there exists a Hecke eigenform $G \in\widetilde M_{{\bf k}}(\varGamma^{(4)})$ such that $G$ is not a constant multiple of $[F]^{{\bf k}}$
 and
\[G \equiv_{{\ev}}  [F]^{{\bf k}} \pmod \frkp .\]
\end{theorem}
\begin{proof} The assertion in the case  $k=l$  has been proved in \cite{Katsurada-Mizumoto12} in more general setting  and the other 
case can also be proved using the  same argument as in its  proof. But for the sake of convenience, we here 
give an outline of the proof. Suppose that $k >l$. Take a basis $\{F_j\}_{1 \le j \le d}$ of $\widetilde M_{\bf k}(\varGamma^{(4)})$ such that $F_1 = [F]^{{\bf k}}$. Then, by Corollary \ref{cor.pullback2}, for any $A \in \calh_4(\ZZ)_{>0}$ we have
\begin{align*}
g_{(k,l,4,4),A}^{(4)}(Z_1)=\sum_{j=1}^d c_j(A,V)F_j(Z_1)(U),
\end{align*}
where $c_1(A,V)=\gamma_2\calc_{8,l}(F)\overline{a(A,[F]^{{\bf k}})(V)}$ and $c_j(A,V)=\overline{a(A,\widetilde G_j)(V)}$
for some $\widetilde G_j \in M_{{\bf k}}(\varGamma^{(4)})$ for $2 \le j \le d$.
We have
\begin{align*}
&{Z(8,l) \over Z(4,l)}\calc_{8,l}(F)a(A,[F]^{{\bf k}})(U)\overline{a(A,[F]^{{\bf k}})(V)}
=\Bigl({Z(8,l) \over Z(4,l)}\Bigr)^2\\
\times &{{\bf L}(l-2,F,\St)a(A_1,F)a(A,[F]^{{\bf k}})(U) {\bf L}(l-2,F,\St)\overline{a(A_1,F)a(A,[F]^{{\bf k}})(V)} \over |a(A_1,F)|^2 {\bf L}(l-2,F,\St)}.
\end{align*}
We note that the reduced denominator of $\displaystyle {Z(8,l) \over Z(4,l)}$ is not divisible by $\frkp$ by the theorem of v. Staut-Clausen.
Hence we have 
\[\ord_{\frkp}(\gamma_2\calc_{8,l}(F)a(N,[F]^{{\bf k}})(U)\overline{a(N,[F]^{{\bf k}})(V)})<0.\] 
Hence the assertion follows from Lemma \ref{lem.congruence}.
\end{proof}

\begin{proposition} \label{prop.congruence-Klingen-Ikeda}
Let $k$ and $l$ be even integers such that $6 \le l \le k$.
Let $f$ be a primitive form in $S_{2k-2}(\SL_2(\ZZ))$. Suppose that a prime ideal $\frkp$ in $\QQ(f)$ satisfies the following conditions (1), (2), (3):
\begin{itemize}
\item[(1)] $p_{\frkp} \ge 2k-2$.
\item[(2)] $\frkp$ divides ${\bf L}(k+l-4,f)/{\bf L}(k,f)$. 
\item[(3)] $\frkp$ divides neither $\frkD_f$ nor $\zeta(3-2k)$, where $\frkD_f$ is the ideal of $\QQ(f)$ defined in Proposition \ref{prop.p-integrality}.
\end{itemize}
Then for any $N \in \calh_2(\ZZ)_{>0}$ such that $\frkp \nmid \frkd_N$, we have
\[\ord_{\frkp}(|a(N,\scri_2(f))|^2{\bf L}(l-2,\scri_2(f),{\rm St}))>0.\]
Here, $\frkd_N$ is the discriminant of $\QQ(\sqrt{-\det N})/\QQ$ as defined before.
\end{proposition}
\begin{proof}
Let $g$ be a Hecke eigenform in the Kohnen plus space $S_{k-1/2}^+(\varGamma_0(4))$ 
corresponding to $f$ under the Shimura correspondence. 
For any $N \in \calh_2(\ZZ)_{>0}$ we have $a(N,\scri_2(f))=b a(|\frkd_N|,g)$ with $b \in \ZZ$. 
Hence
\begin{align*}
{\bf L}(l-2,\scri_2(f),\St)|a(N,\scri_2(f))|^2
=b^2A_{2,k,l-2}{L(l-2,\scri_2(f),\St) \over \pi^{2k+3l-9}\langle \scri_2(f), \ \scri_2(f) \rangle}|a(|\frkd_N|,g)|^2
\end{align*}
with $A_{2,k,l-2} \in \ZZ_{(p_{\frkp})}$ \ (cf. Remark \ref{rem.integrality-gamma-factor}).
By definition, we have
\begin{align*}
L(l-2,\scri_2(f),\St)=\zeta(l-2)L(k+l-3,f)L(k+l-4,f).
\end{align*}
Moreover, by \cite{Kohnen-Skoruppa89}, we have
\begin{align*}
\langle \scri_2(f),\scri_2(f) \rangle=2^{k-2}\langle g,g \rangle \Gamma_\CC(2)\zeta(2) \Gamma_\CC(k)L(k,f),
\end{align*}
and  by \cite{Kohnen-Zagier81} we have
\begin{align*} 
{|a(|\frkd_N|,g)|^2 \over \langle g,g \rangle}={2^{k-2}|\frkd_N|^{k-3/2}
\Gamma_\CC(k-1)L(k-1,f,({\frkd_N \over  })) \over \langle f,f \rangle}.
\end{align*}
We note that $\tau(({\frkd_N \over  }))=\sqrt{-1}|\frkd_N|^{1/2}$, and
$\pi^{2-l}\zeta(l-2)$ and $\pi^{-2}\zeta(2)$ belong to $\ZZ_{(p_\frkp)}$. Hence, by a simple computation, we have

\begin{align*}
{\bf L}(l-2,\scri_2(f),\St)|a(N,\scri_2(f))|^2 & =\epsilon_{k,N} {{\bf L}(k+l-4, f
) \over   {\bf L}(k,f)} \times {\bf L}(k+l-3,k-1;f;{\bf 1},({\frkd_N \over  }))
\end{align*}
where $\epsilon_{k,N}$ is a $\frkp$-integral rational number.  Since $\frkp$ divides neither $\frkD_f\frkd_{N}$ nor $\zeta(3-2k)$, by Proposition \ref{prop.p-integrality}, the value 
${\bf L}(k+l-3,k-1;f;{\bf 1},({\frkd_N \over  }))$
is $\frkp$-integral. Thus the assertion holds.
\end{proof}
The next theorem clarifies what we need to look at 
to try to prove Conjecture \ref{conj.main-conjecture}.
\begin{theorem} \label{th.main-congruence}
Let $k$ and $l$ positive even integers such that $6 \le l \le k$, and put ${\bf k}=(k,k,l,l)$.
Let $f$ be a primitive form in $S_{2k-2}(\SL_2(\ZZ))$ and $\frkp$ a prime ideal of $\QQ(f)$ such that
\begin{itemize}
\item[(1)] $p_\frkp \ge 2k-2$.
\item[(2)] $\frkp$ divides ${\bf L}(k+l-4,f)/{\bf L}(k,f)$.
\item[(3)] $\frkp$ divides neither $\frkD_f$ nor $\zeta(3-2k)$.
\item[(4)] $\frkp$ divides neither $\calc_{8,l}(\scri_2(f))a(A_1,\scri_2(f))\overline{a(A,[\scri_2(f)]^{{\bf k}})}$ nor $\frkd_{A_1}$
for some $A_1 \in \calh_2(\ZZ)_{>0}$ and $A \in \calh_4(\ZZ)_{>0}$. 
\end{itemize}
Then there exists a Hecke eigenform $G$ in $\widetilde M_{{\bf k}}(\varGamma^{(4)})$ such that
$G$ is not a constant multiple of $[\scri_2(f)]^{{\bf k}}$ and 
\[G \equiv_{\ev} [\scri_2(f)]^{{\bf k}} \pmod{\frkp}.\]
\end{theorem}
\begin{proof} The assertion follows from Theorem \ref{th.congruence-Klingen2} and Proposition \ref{prop.congruence-Klingen-Ikeda}.
\end{proof}

\section{Fourier coefficients of Klingen-Eisenstein lift}
Let ${\bf k}=(k,k,l,l)$  with $k,l$ positive even integers such that $k \ge l$. To confirm the condition (4) in Theorem \ref{th.main-congruence}, we give a formula for computing ${\bf L}(l-2,F,\St)a(T,F)\overline{a(N,[F]^{{\bf k}}})$
for a Hecke eigenform $F$ in $S_k(\varGamma^{(2)})$, $T \in \calh_2(\ZZ)_{>0}$ and $N \in \calh_4(\ZZ)_{>0}$.
For $T \in \calh_{2}(\ZZ)_{>0}$ and $N \in \calh_{4}(\ZZ)_{>0}$, let $\epsilon_{k,l,2,4}(T,N)(U,V)$ be as in (E)  and put $g_{N}=g_{(k,l,2,4),N}^{(2)}$. Recall that $U$ and $V$ are $2 \times 2$ and $2 \times 4$ matrices, respectively, of variables. We note that $\epsilon_{k,l,2,4}(T,N)(U,V)$ can be expressed as
\[\epsilon_{k,l,2,4}(T,N)(U,V)=(\det U)^{k-l}\epsilon_{k,{\bf k}}(T,N)(V)\]
with $\epsilon_{k,{\bf k}}(T,N)=\epsilon_{k,{\bf k}}(T,N)(V) \in \CC[V]_{(k-l,k-l,0,0)}$.
Then $g_N$ is expressed as 
$$g_N(W)=\sum_{T \in \calh_2(\ZZ)}(\det U)^{k-l}{\epsilon}_{k,{\bf k}}(T,N){\bf e}(\mathrm{tr}(TW)).$$
Now, for a positive integer  $m$, let $T(m)$ be the element of ${\bf L}_2$ defined in Section 3. 
For  a positive integer $m=p_1\cdots p_r $ with $p_i$ a prime number, we define the Hecke operator $T^{(m)}=T(p_1)\cdots T(p_r)$. {We make the convention that $T^{(1)}=T(1)$.
We note that $T^{(m)}=T(m)$ if $p_1,\ldots,p_r$ are distinct, but in general it is not.
For each $m \in \ZZ_{> 0}$ and $N \in \calh_4(\ZZ)_{>0},$ write $g_N |T^{(m)}(W)$ as
\begin{align*}
g_N|T^{(m)}(W)=\sum_{T \in \calh_2(\ZZ)_{>0}} (\det U)^{k-l}\epsilon_{k,{\bf k}}(m,T,N){\bf e}(\mathrm{tr}(T W))
\end{align*}
with $\epsilon_{k,{\bf k}}(m,T,N) \in \CC[V]_{(k-l,k-l,0,0)}$.

Let $\calm_{k,l}=M_k(\varGamma^{(2)})$ or $S_{k}(\varGamma^{(2)})$ according as $k=l$ or not, and 
let $\{F_j \}_{j=1}^d$ be a basis of $\calm_{k,l}$ consisting of Hecke eigenforms. 
Furthermore write 
$$F_j|T^{(m)}(z)=\lambda_{j,m}F_j(z).$$
Then the following proposition is a consequence of applying $T(m)$ to the formula in Corollary \ref{cor.pullback2} with $n_1=2$ and $n_2=4$.
\begin{proposition} \label{prop.fc-pullback} 
Notation being as above, we have
\[\epsilon_{k,{\bf k}}(m,T,N)=\sum_{j=1}^d \lambda_{j,m} a(T,F_j)B(F_j)\]
for any $N \in \calh_4(\ZZ)_{>0}$, $T\in \calh_2(\ZZ)_{>0}$ and $m\in \ZZ_{> 0}$, where $B(F_j)$ is a certain element of $\CC[V]_{(k-l,k-l,0,0)}$, and in particular we have
\[B(F_j)=\gamma_2 \calc_{6,l}(F_j) \overline{a(N,[F_j]^{{\bf k}})}\]
if $F_j \in S_{k}(\varGamma^{(2)})$. Here, $\gamma_2$ is the rational number in Proposition \ref{prop.pullback1}.
\end{proposition}
We note that $\calc_{8,l}(F)=\zeta(9-2l)\calc_{6,l}(F)$ for a Hecke eigenform $F$ in $S_k(\varGamma^{(2)})$. Hence by the above proposition, we have the following formula:
\begin{proposition} \label{prop.fc-klingen1}
 For $N_1 \in \calh_2(\ZZ)_{>0},N \in \calh_4(\ZZ)_{>0}$  let $e_m= \epsilon_{ k,{\bf k}}(m,N_1,N)$.  
Let $F$ be a Hecke eigenform in $S_k(\varGamma^{(2)})$ and 
 $\{F_j \}_{j=1}^d$  a basis of $\calm_{k,l}$ consisting of Hecke eigenforms such that $F_1=F$.
For  positive integers  $m_1,\ldots,m_d$ put $\Delta=\Delta(m_1,\ldots,m_d)=\det (\lambda_{j,m_i})_{1 \le i,j \le d}$.
Then,
\begin{align*}
\Delta \gamma_2 \calc_{8,l}(F)a(N_1,F)\overline{a(N,[F]^{{\bf k}})}= \zeta(9-2l)\begin{vmatrix} e_1 & \lambda_{1,2} &\hdots&  \lambda_{1,d} \\
               \vdots & \vdots         &\vdots & \vdots          \\
                e_d    &  \lambda_{d,2}         &\hdots &  \lambda_{d,d} \end{vmatrix}.
\end{align*}
\end{proposition}
\begin{corollary} \label{cor.fc-klingen1}
Let the notation and the assumption as above. Let $\frkp$ be a prime ideal of $\QQ(F)$ such that $p_{\frkp} >2k$. Suppose that
$\frkp$ divides neither  $\zeta(9-2l)$ nor $\begin{vmatrix} e_1 & \lambda_{1,2} &\hdots&  \lambda_{1,d} \\
               \vdots & \vdots         &\vdots & \vdots          \\
                e_d    &  \lambda_{d,2}         &\hdots &  \lambda_{d,d} \end{vmatrix}$. Then, $\frkp$ does not divide $\calc_{8,l}(F)a(N_1,F)\overline{a(N,[F]^{{\bf k}})}$.
\end{corollary}
\begin{proof}
By Proposition \ref{prop.Hecke-integrality}, $\Delta$ is an algebraic integer, and by the assumption, $\gamma_2$ is a $\frkp$-unit. Thus the assertion holds.
\end{proof}

The following lemma will be used in the next section.

\begin{lemma} \label{lem.recusion-formula-for-FC}
Let  $N \in \calh_4(\ZZ)_{>0}.$  Then
 for any $T \in \calh_2(\ZZ)_{>0}$ and a prime number $p,$ we have the following recursion formula for $\epsilon_{k,{\bf k}}(m,T,N):$ 
$$ \epsilon_{k,{\bf k}}(1,T,N)=\epsilon_{k,{\bf k}}(T,N),$$
and for $m >1$,
\begin{align}
 \epsilon_{k,{\bf k}}(m,T,N) =& \epsilon_{ k,{\bf k}}(mp^{-1},pT,N)+p^{2k-3} \epsilon_{k,{\bf k}}(mp^{-1},T/p,N)\nonumber\\
&+p^{k-2}\sum_{D \in \GL_2({\ZZ}) U_p \GL_2({\ZZ})/\GL_2({\ZZ})} \epsilon_{k,{\bf k}}(mp^{-1},T[D]/p,N),
\nonumber
\end{align}
where $p$ is a prime factor of $m$ and $U_p=\begin{pmatrix} 1 &0 \\ 0 & p\end{pmatrix}$. 
\end{lemma}
\begin{proof} The assertion follows from \cite[Exercise 4.2.10]{Andrianov87}.
\end{proof}
 
Let  $U$ and $V$  be the matrices of variables stated above.
 \begin{theorem} \label{th.fc-klingen2}
  For $A_0 \in \calh_2(\ZZ)_{>0}, A_1 \in \calh_4(\ZZ)_{>0}$ and  $R \in M_{2,4}(\ZZ)$, put  $r(R)=r(A_0,A_,R)=\rank  \begin{pmatrix} A_0 & R/2 \\ {}^t\!R/2 & A_1 \end{pmatrix}$ and
\begin{align*}
&\calz(A_0,A_1,R,l) =
\begin{cases}
 L\Bigl({4-l,\chi_{\bigl(\begin{smallmatrix} A_0 & R/2 \\ {}^tR/2 & A_1 \end{smallmatrix}\bigr)}}\Bigr)
 & \text{ if } r(A_0,A_1;R)=6 ,\\
\zeta(7-2l) & \text{ if } \ r(A_0,A_1,R)=5 ,\\
\zeta(7-2l)L(3-l,\chi_{A_1}) & \text{ if } \ r(A_0,A_1,R)=4.
\end{cases}
\end{align*}
Moreover, put 
\begin{align*}
P\Bigl(\begin{pmatrix} A_0 & R/2 \\ {}^t\!R/2 & A_1 \end{pmatrix}\Bigr)(V) 
& =(k-l)! \sum_{a,b,c \ge 0 \atop a+2b+2c=k-l}{(-1)^b2^a \over a!b!c! }(l+c-{3 \over 2})_{a+b+c} \\
&\quad \times |R \ {}^tV|^a  (|VA_1{}^t V| |A_0|)^b \begin{vmatrix} A_0& R \ {}^tV/2 \\ V  {}^t\!R/2 & VA_1{}^tV
\end{vmatrix}^c.
\end{align*}
 Then 
\begin{align*}
\epsilon_{k,{\bf k}}(A_{0},A_{1})(V) & =\sum_{R \in M_{2,4}(\ZZ) \atop 
\left(\begin{smallmatrix} A_1 & R/2 \\ {}^t\!R/2 &A_1 \end{smallmatrix} \right) \ge 0} 2^{[r(R)+1)/2]}\calz(A_0,A_1,R,l)P\Bigl(\begin{pmatrix} A_0 & R/2 \\ {}^t\!R/2 & A_1 \end{pmatrix}\Bigr) (V)\\
& \times\prod_{p} F_p^\ast\Bigl(\left(\begin{smallmatrix} A_0 & R/2 \\ {}^t\!R/2 &A_1 \end{smallmatrix} \right),p^{l-r(R)-1}\Bigr).
\end{align*}
\end{theorem}
\begin{proof}
Let $Q_{k-l,2,4}$ be the polynomial in Section 5.1.2, Then, by 
Lemma \ref{lem.formulaQl}, we have
\[Q_{k-l,2,4}(\begin{pmatrix} A_0 & R/2 \\ {}^t\!R/2 & A_1 \end{pmatrix},U,V)=(\det U)^{k-l}P\Bigl(\begin{pmatrix} A_0 & R/2 \\ {}^t\!R/2 & A_1 \end{pmatrix}\Bigr) (V).\]
Thus by (\ref{explicit-epsilon}) in the proof of Proposition \ref{prop.rationality1}, Lemma \ref{lem.formulaQl}, and Proposition \ref{prop.fc-Siegel}, we prove the assertion.
\end{proof}
We have an explicit formula for $F_p(T,X)$ for any nondegenerate half-integral matrix $T$ over ${\ZZ}_p$ (cf. \cite{Katsurada99}), and an algorithm for computing it (cf. Chul-Hee Lee \\
\verb+ https://github.com/chlee-0/computeGK+). 
Therefore, by using Proposition \ref{prop.fc-klingen1}  and Theorem \ref{th.fc-klingen2}  we can  compute the Fourier coefficients of  the Klingen-Eisenstein series in question.  

\section{Main results}
For $l=12,16,18,22,26$ let $\phi_l$ be a unique  primitive form in $S_l(\SL_2(\ZZ))$, and for $l=24,28,30,32,34,38$ let $\phi_l^{\pm}$ be a unique  primitive form in $S_l(\SL_2(\ZZ))$ such that  $a(2,\phi^{\pm})=a \pm b\sqrt{D})$ with $a \in \QQ, b \in \QQ_{>0}$, where $D$ is the discriminant of $\QQ(\phi^{\pm})$. For each 
  $(k,j)=(10,4),(16,2), or (4,24)$ let 
$G_{k,j}$ be a Hecke eigenform in $S_{(k+j,k)}(\varGamma^{(2)})$ uniquely determined up to constant multiple, and  $\{G_{14,4}^+,G_{14,4}^-\}$ is a basis of $S_{(18,14)}(\varGamma^{(2)})$ consisting of Hecke eigenforms.

(1) Let  $f=\phi_{22}$. Put ${\bf k}=(12,12,6,6)$. Then,  by  Taibi \cite{Taibi17} and the numerical table in
\verb+ https://otaibi.perso.math.cnrs.fr/dimtrace+,  we  have
\begin{align*}
&S_{{\bf k}}(\varGamma^{(4)})= \langle \scra_4^{(I)}(G_{10,4}) \rangle_{\CC},\\
&S_{(12,12,6)}(\varGamma^{(3)})=\{0\},\\
&S_{12}(\varGamma^{(2)})=\langle \scri_2(\phi_{22})\rangle_{\CC}.
\end{align*}
Hence we have $\widetilde M_{{\bf k}}(\varGamma^{(4)})=\langle [\scri_2(\phi_{22})]^{\bf k}, \ \scra_4^{(I)}(G_{10,4})\rangle_\CC$.
Then  $41$ is the only prime number which satisfies the assumptions in Conjecture \ref{conj.modified-Harder}, and it divides ${\bf L}(14,f)/{\bf L}(12,f)$. 
Let $N=\left(\begin{matrix} 
1    &1/2&0&1/2 \\
1/2& 1 & 0 & 0 \\
0 & 0 & 1& 1/2 \\
1/2 &0 & 1/2 & 1
\end{matrix} \right)$ and $N_1=\left(\begin{matrix} 1 & 0 \\ 0  & 1 \end{matrix}\right)$. Then, substituting $V$ for  $\begin{pmatrix} 1 & 0 & 1 &0 \\1 & 1 & 0 & 3 \end{pmatrix}$ in Theorem \ref{th.fc-klingen2},  by a computation with Mathematica
\begin{align*}
\varepsilon_{12,{\bf k}}(N_1,N)\Bigl(\begin{pmatrix} 1 & 0 & 1 &0 \\1 & 1 & 0 & 3 \end{pmatrix}\Bigr)
&={-20874555 \over 28} \equiv 11 \pmod{41}.  
\end{align*}
Hence, applying Corollary \ref{cor.fc-klingen1} to $d=1$ and $\lambda_{1,1}=1$, we see that $\frkp$ does not divide
$\calc_{8,6}(\scri_2(\phi_{22}))\overline{a(N,[\scri_2(\phi_{22})]^{\bf k})}a(N_1,\scri_2(\phi_{22}))$, and 
by Theorem \ref{th.main-congruence} we prove the following theorem.
\begin{theorem} \label{th.main-result1} 
 There exists a Hecke eigenform $G$ in $S_{(14,10)}(\varGamma^{(2)})$ such that  
\[\scra_{4}^{(I)}(G)
 \equiv_{\ev} [\scri_2(f)]^{{\bf k}} \pmod {41} .\]
\end{theorem}
\begin{corollary} \label{cor.main-result1}
Conjecture \ref{conj.modified-Harder} holds for $(k,j)=(10,4)$.
\end{corollary}
 We note that Harder's conjecture for $(k,j)=(10,4)$ has been already proved by Chenevier and Lannes \cite{Chenevier-Lannes19}.

 (2) Let $f=\phi_{30}^+$ and $f'=\phi_{30}^-$ and 
put 
\[\alpha=4320+96\sqrt{51349} \text{ and } \alpha'=4320-96\sqrt{51349}.\] 
Then 
\[a(2,f)=\alpha, \ a(3,f)=-552\alpha-99180 \]
 and 
\[ a(2,f')=\alpha', \ a(3,f)=-552\alpha'-99180.\]
We also have 
\[\lambda_{\scri_2(f)}(T(2))=\alpha+49152, \ \lambda_{\scri_2(f')}(T(2))=\alpha'+49152, \]
 and 
\[\lambda_{\scri_2(f)}(T(3))=-552\alpha+19032696, \ \lambda_{\scri_2(f')}(T(3))=-552\alpha'+19032696.\]

(2.1) Let $(k,j)=(14,4)$. Then the prime number $4289$ divides
$N_{\QQ(f)/\QQ}({\bf L}(18,f)/{\bf L}(16,f))$ and it splits  in $\QQ(f)$.
Hence there exists  prime ideals $\frkq, \frkq'$ of $\QQ(f)$ such that $(4289)=\frkq \frkq'$ and
$\frkq$ divides ${\bf L}(18,f)/{\bf L}(16,f)$. The prime ideal $\frkq$ is the only prime ideal which satisfies the assumptions in Conjecture \ref{conj.modified-Harder}.
Put ${\bf k}=(16,16,6,6)$. 
Let $N$ and $N_1$ be as those in (1). For   $V=\begin{pmatrix} 1 & 0 & 1 & 0 \\ 1 & 1 & 0 & 3 \end{pmatrix}$,  put
$\varepsilon_{16,{\bf k}}(N_1,N)=\varepsilon_{16,{\bf k}}(N_1,N)(V) ,  \varepsilon_{16,{\bf k}}(2N_1,N)=\varepsilon_{16,{\bf k}}(2N_1,N)(V)$, 
and $e_i=\varepsilon_{16,{\bf k}}(i,N_1,N) \ (i=1,2)$. Then
\begin{align*}
&e_1= \varepsilon_{16,{\bf k}}(N_1,N), \ e_2=\varepsilon_{16,{\bf k}}(2N_1,N)+2^{14} \varepsilon_{16,{\bf k}}(N_1,N)
\end{align*}
and 
\[\alpha_{16,{\bf k}}(N_1,N)= \begin{vmatrix}  e_1 & 1 \\
e_2 & \lambda_{\scri_2(f')}(T(2))
\end{vmatrix}.\]
By a computation with Mathematica, we have
\[\varepsilon_{16,{\bf k}}(N_1,N)=1744286277555/28672, \quad \varepsilon_{16,{\bf k}}(2N_1,N)=309108562779375/112
,\]
and hence
\begin{align*}
\alpha_{16,{\bf k}}(N_1,N)=405 (-1114174584071 + 12920639093 \sqrt{51349})/896.
\end{align*}
Using Theorem \ref{th.fc-klingen2}, by a computation with Mathematica we have
\[N_{\QQ(f)/\QQ}(\alpha_{16,{\bf k}}(N_1,N)) \equiv 2206 \pmod{4289}.\]
Hence, by Corollary \ref{cor.fc-klingen1}, $\frkp$ does not divide $\calc_{8,6}(\scri_2(f))\overline{a(N,[\scri_2(f)]^{\bf k})}a(N_1,\scri_2(f))$.
Moreover, by  Taibi \cite{Taibi17} and the numerical table in
\\
 \verb+https://otaibi.perso.math.cnrs.fr/dimtrace+,  we have 
\begin{align*}
&S_{{\bf k}}(\varGamma^{(4)})=\langle \scra_4^{(I)}(G_{14,4}^+),\scra_4^{(I)}(G_{14,4}^-) \rangle_{\CC},\\
&S_{(16,16,6)}(\varGamma^{(3)})=\{0\},\\
&S_{16}(\varGamma^{(2)})=\langle \scri_2(\phi_{30}^+),\scri_2(\phi_{30}^-)\rangle_{\CC}.
\end{align*}
The prime ideal $\frkq$ does not divide $\frkD_f$ and $[\scri_2(\phi_{30}^+)]^{\bf k} \not \equiv [\scri_2(\phi_{30}^-)]^{{\bf k}} \text{ mod } \frkq$. 
Hence, by Theorem \ref{th.main-congruence} we prove the following theorem. 
\begin{theorem} \label{th.main-result2} 
 There exists a Hecke eigenform $G$ in $S_{(18,14)}(\varGamma^{(2)})$ such that  
\[\scra_4^{(I)}(G)
 \equiv_{\ev} [\scri_2(f)]^{{\bf k}} \pmod \frkq .\]
\end{theorem}
\begin{corollary} \label{cor.main-result2}
Conjecture \ref{conj.modified-Harder} holds for $(k,j)=(14,4)$.
\end{corollary}

(2.2) Let $(k,j)=(4,24)$ and put ${\bf k}=(16,16,16,16)$. Then the prime number $97$ divides
$N_{\QQ(f)/\QQ}({\bf L}(28,f)/{\bf L}(16,f))$ and it splits  in $\QQ(f)$. 
 Hence there exist prime ideals $\frkq, \frkq'$ of $\QQ(f)$ such that $(97)=\frkq \frkq'$ and
$\frkq$ divides ${\bf L}(28,f)/{\bf L}(16,f)$. The prime ideal $\frkq$ is the only prime ideal which satisfies the assumptions in Conjecture \ref{conj.modified-Harder}.
We have  
\[M_{16}(\varGamma^{(2)})=\bigl\langle E_{2,16}, [\phi_{16}]^{(16,16)}, \scri_2(f),\scri_2(f') \bigr\rangle_{\CC}.\]
Let $N$ be as that in (1), $N_1=\begin{pmatrix}1 & 1/2 \\ 1/2 & 1\end{pmatrix}$, and  $N_2=\begin{pmatrix} 1 & 0 \\ 0 & 3 \end{pmatrix}$. Put $e_i=e_{16,{\bf k}}(i,N_1,N) \ (i=1,2,3,4)$. Then
\begin{align*}
&e_1=\varepsilon_{16,{\bf k}}(N_1,N)\\
&e_2=\varepsilon_{16,{\bf k}}(2N_1,N),\\
&e_3=\varepsilon_{16,{\bf k}}(3N_1,N)+3^{14}\varepsilon_{16,{\bf k}}(N_1,N)\\
&e_4=\varepsilon_{16,{\bf k}}(4N_1,N)+2^{29}\varepsilon_{16,{\bf k}}(N_1,N)+3\cdot 2^{14}\varepsilon_{16,{\bf k}}(N_2,N)
\end{align*}
and  
\[\alpha_{16,{\bf k}}(N_1,N)=\begin{vmatrix}  e_1 & 1 & 1 & 1\\
e_2 & \lambda_{\scri_2(f')}(T(2)) & \lambda_{[\phi_{16}]^{(16,16)}}(T(2) ) & \lambda_{E_{2,16}}(T(2)) \\
e_3 & \lambda_{\scri_2(f')}(T(3)) & \lambda_{[\phi_{16}]^{(16,16)}}(T(3)) & \lambda_{E_{2,16}}(T(3))\\
e_4 & (\lambda_{\scri_2(f')}(T(2)))^2 & (\lambda_{[\phi_{16}]^{(16,16)}}(T(2)))^2 & (\lambda_{E_{2,16}}(T(2)))^2 \end{vmatrix}.\]
Then, by Lemma \ref{lem.recusion-formula-for-FC}, we have
$e_i=\varepsilon(i,N_1,N)$ for $i=1,2,3,4$.
Using Theorem \ref{th.fc-klingen2}, by a computation with Mathematica, we have
\begin{align*}
\varepsilon_{16,{\bf k}}(N_1,N)
&=38740804007974226508744800778240/6232699579062017,
\\
\varepsilon_{16,{\bf k}}(2N_1,N)
&=8035873503466715618094093067152998400/6232699579062017,
\\
\varepsilon_{16,{\bf k}}(3N_1,N)
&=-29430266109700665036971047394543222568960/6232699579062017,
\\
\varepsilon_{16,{\bf k}}(4N_1,N)&=7060754204175435666580204417230810153615360/6232699579062017
\\
\text{ and } \\
\varepsilon_{16,{\bf k}}(N_2,N)&=
337608542664093039037162829831689850880/6232699579062017
.\end{align*} 
We also have
\begin{align*}
&\lambda_{[\phi_{16}]^{(16,16)}}(T(2))=a(2,\phi_{16})(1+2^{14})=216(1+2^{14}),\\
&\lambda_{[\phi_{16}]^{(16,16)}}(T(3))=a(3,\phi_{16})(1+3^{14})=-3348 (1+3^{14}),\\
& \lambda_{E_{2,16}}(T(q))=(1+q^{14})(1+q^{15}) \quad \text{ for } q=2,3.
\end{align*}
Hence by a simple computation we have 
\[N_{\QQ(f)/\QQ}(\alpha_{k,k}(N_1,N)) \not\equiv 0 \pmod{97}.\]
Hence by  Corollary \ref{cor.fc-klingen1}, $\frkp$ does not divide $\calc_{8,16}(\scri_2(f)) a(N_1,\scri_2(f))\overline{
a(N,[\scri_2(f)]^{{\bf k}})}$.
The prime ideal $\frkq$ does not divide $\frkD_f$.
Hence, by Theorem \ref{th.main-congruence},
there exists a Hecke eigenform $F$ in $M_{16}(\varGamma^{(4)})$  such that
\[F \equiv_{\ev} [\scri_2(f)]^{\bf k} \pmod \frkp.\]
To show that $F$ is a lift of type $\scra^{(I)}$, we classify  the Hecke eigenforms in $M_{16}(\varGamma^{(4)})$ following  \cite{Poor-Ryan-Yuen09} and \cite{Ib-Kat-P-Y14}.
\begin{proposition} \label{prop.numerical-data}
We have $\dim M_{16}(\varGamma^{(4)})=14$ and $\dim S_{16}(\varGamma^{(4)})=7$, and we have the following 
\begin{itemize}
\item[(1)] We can take a basis 
$\{h_i \}_{i=1}^7$ of $S_{16}(\varGamma^{(4)})$ such that
\[h_i=\begin{cases}
\scri_4(\phi_{28}^+) & i=1, \\
\scri_4(\phi_{28}^-)  &  i=2, \\
\scra_4^{(II)}(\phi_{30}^+,G_{16,2})  & i=3,\\
\scra_4^{(II)}(\phi_{30}^-,G_{16,2})  & i=4,\\
\scrm^{(I)}(\phi_{26}, \scri_2(\phi_{30}^+)) \quad  & i=5, \\
\scrm^{(I)}(\phi_{26},\scri_2(\phi_{30}^-))  & i=6, \\
\scra_4^{(I)}(G_{4,24}) & i=7.
\end{cases}\]
Moreover we have
\[\lambda_{h_i}(T(2))=
\begin{cases} 12960(67989+443\sqrt{18209}) & i=1, \\
12960(67989-443\sqrt{18209}) & i=2, \\
-230400(1703+9\sqrt{18209}) & i=3, \\
-230400(1703-9\sqrt{18209})  & i=4, \\
1175040(557+\sqrt{51349})     & i=5, \\
 1175040(557-\sqrt{51349})   & i=6, \\
230400000 & i=7.
\end{cases}
\]
\item[(2)] We can take a basis $\{h_i \}_{i=8}^{14}$ of $S_{16}(\varGamma^{(4)})^{\perp}$ such that
\[h_i=\begin{cases}
E_{4,16} & i=8, \\
[\phi_{16}]^{\bf k} & i=9, \\
[\scrm^{(I)}(\phi_{16},\phi_{28}^+)]^{\bf k} & i=10,\\
[\scrm^{(I)}(\phi_{16},\phi_{28}^-)]^{\bf k} & i=11, \\
[H_{16}^{(3)}]^{\bf k} & i=12, \\
[\scri_2(\phi_{30}^-)]^{\bf k} & i=13, \\
[\scri_2(\phi_{30}^+)]^{\bf k}& i=14,
\end{cases}\]
where $H_{16}^{(3)}$ is a unique tempered Hecke eigenform, up to  constant multiple, in $S_{16}(\varGamma^{(3)})$. For the definition of tempered forms, see Appendix A.

Moreover, we have
\[\lambda_{h_i}(T(2))=
\begin{cases} 
18022646021156865 & i=8, \\
  118797996294360                          & i=9, \\
4097(4414176+23328\sqrt{18209}) & i=10, \\
4097(4414176-23328\sqrt{18209}) & i=11, \\
-471974400 & i=12, \\
  33566721(53472-\sqrt{51349})& i=13, \\
 33566721(53472+\sqrt{51349}) & i=14.
\end{cases}\]
\end{itemize}
\end{proposition}
\begin{remark}
We have 
$\scrm^{(I)}(\phi_{26}, \scri_2(\phi_{30}^{\pm}))=\scrm^{(II)}(\phi_{30}^{\pm},\scri_2(\phi_{26}))$.
\end{remark}

By Proposition \ref{prop.numerical-data}, we have
\[\lambda_{h_i}(T(2)) \not\equiv \lambda_{[\scri_2(f)]^{\bf k}}(T(2)) \pmod \frkp \]
for any $1 \le i \le 13$ such that $i\not=7$.
Hence $F$ coincides with  $h_7$ up to constant multiple.
Hence we have the following theorem.
\begin{theorem} \label{th.main-result3} 
 There exists a Hecke eigenform $G$ in $S_{(28,4)}(\varGamma^{(2)})$ such that  
\[\scra_4^{(I)}(G)^{{\bf k}}
 \equiv_{\ev} [\scri_2(f)]^{{\bf k}} \pmod \frkq .\]
\end{theorem}
\begin{corollary}
\label{cor.main-result3}
Conjecture \ref{conj.modified-Harder} holds for $(k,j)=(4,24)$.
\end{corollary}

\appendix


\section{Proofs of Theorems \ref{th.atobe1} and \ref{th.atobe2}}
In this appendix, we will give proofs of Theorems \ref{th.atobe1} and \ref{th.atobe2}. 
These theorems are a simple application of Arthur's endoscopic classification \cite{Ar} to Siegel modular forms
as in the book of Chenevier--Lannes \cite[\S 8.5.1]
{Chenevier-Lannes19}. 
\par

First we recall the explicit multiplicity formula for $\Sp_n(\ZZ)$. 
The following theorem is just a reformulation of \cite[Theorem 8.5.2, Corollary 8.5.4]{Chenevier-Lannes19}. 

\begin{theorem}[Explicit multiplicity formula for $S_{{\bf k}}(\Sp_n(\ZZ))$]
\label{MF}
Let ${\bf k} = (k_1, \dots, k_n)$ be a sequence of positive integers such that 
$k_1 \geq \dots \geq k_n \geq n+1$. 
\begin{enumerate}
\item
We can associate each Hecke eigenform $G$ in $S_{{\bf k}}(\Sp_n(\ZZ))$
with its $A$-parameter $\psi_G$ which is a formal sum
\[
\psi_G = \boxplus_{i=1}^t \pi_i[d_i],
\]
where $\pi$ and $d_i$ satisfy the following conditions (a) to (f):
 
\begin{enumerate}
\item
$\pi_i = \otimes_v \pi_{i,v}$ 
is an irreducible unitary cuspidal automorphic self-dual representation of $\mathrm{PGL}_{n_i}(\mathbb{A}_\QQ)$; 
\item
$d_i$ is a positive integer such that $\sum_{i=1}^t n_id_i = 2n+1$; 
\item
$\pi_{i,p}$ is unramified for any prime $p < \infty$; 
\item
if we denote the infinitesimal character of $\pi_{i,\infty}$ by $c_{i,\infty}$ 
which is a semisimple conjugacy class of $\mathfrak{sl}_{n_i}(\RR)$, 
and if we set 
\[
e_d = \mathrm{diag}\left( \frac{d-1}{2}, \frac{d-3}{2}, \dots, -\frac{d-1}{2} \right)
\in \mathfrak{sl}_{d}(\RR), 
\]
then the eigenvalues of the semisimple conjugacy class
\[
\bigoplus_{i=1}^t c_{i,\infty} \otimes e_{d_i}
\] 
in $\mathfrak{sl}_{2n+1}(\RR)$ are the distinct integers
\[
k_1-1 > \dots > k_n-n > 0 > -(k_n-n) > \dots > -(k_1-1);
\]
\item
there exists $1 \leq i_0 \leq t$ such that 
$d_{i_0} = 1$, $n_{i_0} \equiv 1 \pmod 2$, and
$n_id_i \equiv 0 \pmod 4$ for any $i \ne i_0$; 
\item
the sign condition 
\[
\prod_{\substack{ 1\leq j \leq t \\ j \not=i}} \varepsilon(\pi_i \times \pi_j)^{\min\{d_i,d_j\}}
= \left\{
\begin{aligned}
&(-1)^{\frac{n_id_i}{4}} &&\text{if $d_i \equiv 0 \pmod 2$}, \\
&(-1)^{|K_i|} &&\text{if $d_i \equiv 1 \pmod 2$}
\end{aligned}
\right. 
\]
holds for any $i \ne i_0$, 
where $K_i$ is the set of odd indices $1 \leq j \leq n$ such that $k_j-j$ is an eigenvalue of $c_{i,\infty}$.
\end{enumerate}
The $A$-parameter $\psi_G$ is characterized by
\[
L(s,G, \St) = \prod_{i=1}^t \prod_{d = 1}^{d_i} L^\infty \left(s+\frac{d_i+1}{2}-d, \pi_i \right), 
\]
where the right-hand side is a product of the finite parts of the Godement--Jacquet $L$-functions. 

\item
Conversely, for any formal sum $\psi = \boxplus_{i=1}^k \pi_i[d_i]$ satisfying (a)--(f) above, 
there exists a Hecke eigenform $G$ such that $\psi = \psi_G$. 

\item
For two Hecke eigenforms $G_1, G_2$ in $S_{{\bf k}}(\Sp_n(\ZZ))$, 
the following are equivalent: 
\begin{itemize}
\item
$G_1$ and $G_2$ are constant multiples of one another; 
\item
$L(s,G_1,\St) = L(s,G_2,\St)$; 
\item
$\psi_{G_1} = \psi_{G_2}$. 
\end{itemize}
\end{enumerate}
\end{theorem}

Here, a formal sum means that it is an equivalence class defined so that 
$\boxplus_{i=1}^t \tau_i[d_i] \sim \boxplus_{i=1}^{t'} \tau'_i[d'_i]$
if $t=t'$ and if there exists a permutation $\sigma \in \mathfrak{S}_t$ such that 
$\pi'_i \cong \pi_{\sigma(i)}$ and $d'_i = d_{\sigma(i)}$ for any $1 \leq i \leq t$.

\begin{remark} \label{remarkMF}
\begin{enumerate}
\item
In \cite[Theorem 8.5.2]{Chenevier-Lannes19}, 
Chenevier and Lannes assumed \cite[Conjecture 8.4.22]{Chenevier-Lannes19}.
As is written in the postface in that book, 
this conjecture has been proved by Arancibia, M{\oe}glin and Renard \cite{AMR}.

\item
Theorem \ref{MF} (3) is a multiplicity one theorem. 
In \cite[Corollary 8.5.4]{Chenevier-Lannes19}, 
not to use \cite[Conjecture 8.4.22]{Chenevier-Lannes19}, 
the stronger assumption $k_1 > \dots > k_n > n+1$ is assumed. 
Using \cite{AMR}, the same proof is available even when $k_1 \geq \dots \geq k_n \geq n+1$. 

\item
By condition (d), 
$\pi_i$ is an irreducible regular algebraic cuspidal self-dual automorphic representation of $\mathrm{PGL}_{n_i}(\mathbb{A}_\QQ)$
(\cite[Definition 8.2.7]{Chenevier-Lannes19}). 
As explained in \cite[Theorem 8.2.17]{Chenevier-Lannes19},
thanks to numerous mathematicians, 
one can prove that such a representation satisfies the Ramanujan conjecture. 
Namely, for any $p$, 
all the eigenvalues of the Satake parameter of $\pi_{i,p}$ have absolute value $1$. 
In particular, a Hecke eigenform $G$ in $S_{{\bf k}}(\Sp_n(\ZZ))$ satisfies the Ramanujan conjecture 
if and only if its $A$-parameter is of the form $\psi_G = \boxplus_{i=1}^t \pi_i[1]$. 
In this case we call $G$ tempered. 

\item
By the purity lemma of Clozel \cite[Proposition 8.2.13]{Chenevier-Lannes19}, 
the Langlands parameter of $\pi_{i,\infty}$ is completely determined by 
the eigenvalues of the infinitesimal character $c_{i,\infty}$. 
In particular, one can compute the Rankin-Selberg root number $\varepsilon(\pi_i \times \pi_j)$ explicitly
in terms of the eigenvalues of $c_{i,\infty}$ and $c_{j,\infty}$ by
\begin{align*}
\varepsilon(\pi_i \times \pi_j) 
&= \prod_{w_i > 0} \prod_{w_j > 0} (-1)^{1+\max\{2w_i, 2w_j\}}
\times \left\{
\begin{aligned}
&1 &&\text{if $i,j \ne i_0$}, \\
&\prod_{w_i>0} (-1)^{\frac{1+2w_i}{2}} &&\text{if $i \ne j = i_0$}, 
\end{aligned}
\right.
\end{align*}
where $w_i$ (resp.~$w_j$) runs over the positive eigenvalues of $c_{i,\infty}$ (resp.~$c_{j,\infty}$). 
Note that $w_i$ and $w_j$ are in $(1/2)\ZZ$ and 
that $2w_i \equiv d_i-1 \pmod 2$ for any (positive) eigenvalue $w_i$ of $c_{i,\infty}$. 

\item
By \cite[Theorem 1.5.3]{Ar}, we know that 
$\varepsilon(\pi_i \times \pi_j) = 1$ if $d_i \equiv d_j \pmod 2$. 
This is easily checked when $i,j \not= i_0$.

\item
Theorem \ref{MF} is an existence theorem. 
To construct a modular form $G$ from a parameter $\psi$ is a different problem. 

\item
If we were not to assume that $k_n > n$, 
the statement of theorem would be much more difficult. 
At least when the scalar weight case, i.e., when $k_1 = \dots = k_n = k$ with $k \leq n$, 
a similar theorem, in particular a multiplicity one theorem, 
would follow from a result of M{\oe}glin--Renard \cite{MR}. 

\end{enumerate}
\end{remark}

To obtain several lifting theorems, 
we need the following proposition which comes from accidental isomorphisms. 

\begin{proposition} \label{accidental}
\begin{enumerate}
\item
Suppose that $k>0$ is even.
For any primitive form $f$ in $S_{k}(\SL_2(\ZZ))$, 
there exists an irreducible unitary cuspidal automorphic self-dual representation $\pi_f$ of $\mathrm{PGL}_{2}(\mathbb{A}_\QQ)$
such that
\[
L\left(s+\frac{k-1}{2},f \right) = L^\infty(s, \pi_f)
\]
and such that the eigenvalues of the infinitesimal character of $\pi_{f,\infty}$ are $\pm(k-1)/2$. 

\item
Suppose that $j>0$ is even and that $k \geq 4$. 
For any Hecke eigenform $F$ in $S_{(k+j,k)}(\Sp_2(\ZZ))$, 
there exists an irreducible unitary cuspidal automorphic self-dual representation $\pi_F$ of $\mathrm{PGL}_{4}(\mathbb{A}_\QQ)$
such that
\[
L\left(s+\frac{2k+j-3}{2},F, \Sp \right) = L^\infty(s, \pi_F)
\]
and such that the eigenvalues of the infinitesimal character of $\pi_{F,\infty}$ are 
$\pm(j+2k-3)/2, \pm(j+1)/2$. 

\item
Suppose that $k_1 \geq k_2 >0$ are even. 
For any primitive forms $f_1 \in S_{k_1}(\SL_2(\ZZ))$ and $f_2 \in S_{k_2}(\SL_2(\ZZ))$, 
there exists an irreducible unitary cuspidal automorphic self-dual representation $\pi_{f_1,f_2}$ of $\mathrm{PGL}_{4}(\mathbb{A}_\QQ)$
such that
\[
L\left(s+\frac{k_1+k_2-2}{2},f_1 \otimes f_2 \right) = L^\infty(s, \pi_{f_1,f_2})
\]
and such that the eigenvalues of the infinitesimal character of $\pi_{f_1,f_2,\infty}$ are 
$\pm(k_1+k_2-2)/2$ and $\pm(k_1-k_2)/2$. 

\end{enumerate}
\end{proposition}
\begin{proof}
(1) is well-known (see, e.g., \cite[Proposition 9.1.5]{Chenevier-Lannes19}). 
(2) is \cite[Proposition 9.1.4]{Chenevier-Lannes19}. 
(3) can be proved in a way similar to \cite[Proposition 9.1.4]{Chenevier-Lannes19}. 
\end{proof}

Now we explain Theorems \ref{th.atobe1} and \ref{th.atobe2}. 
\begin{proof}[Proof of Theorem \ref{th.atobe1} (1)]
Let $G$ be a Hecke eigenform in $S_{{\bf k}}(\Sp_n(\ZZ))$, 
and $\psi_G = \boxplus_{i=1}^t\pi_i[d_i]$ be its $A$-parameter. 
Here we make the convention that $\psi_G = \mathbf{1}_{\GL_1(A_\QQ)}[1]$ if $n=0$. 
Let $F \in S_{(k+j,k)}(\Sp_2(\ZZ))$ be a Hecke eigenform with $k \geq 4$ and $j >0$ even, 
and $\pi_F$ be the irreducible unitary cuspidal automorphic self-dual representation of $\mathrm{PGL}_4(\mathbb{A}_\QQ)$
associated by $F$ by Proposition \ref{accidental} (2). 
It suffices to show that 
under the assumptions of Theorem \ref{th.atobe1} (1), 
the parameter 
\[
\psi' = \psi_G \boxplus \pi_F[2d]
\]
satisfies the conditions (a)--(f) of Theorem \ref{MF} (1) with respect to 
${\bf k}'=(k_1',\ldots,k_{n+4d}') \in \ZZ^{n+4d}$ defined in Theorem \ref{th.atobe1} (1).
The conditions (a), (b), (c) and (e) are obvious. 
The condition (d) follows from the definition of ${\bf k}'$. 
To check the sign condition (f), we will compute $\varepsilon(\pi_i \times \pi_F)$. 
By Remark \ref{remarkMF} (5), we have $\varepsilon(\pi_i \times \pi_F) =1$ if $d_i$ is even.
When $d_i$ is odd, 
any positive eigenvalue $w_i$ of $c_{i,\infty}$ belongs to $\{k_1-1, \dots, k_n-n\}$ so that $j+1 < 2w_i < j+2k-3$. 
Hence when $d_i$ is odd and $i \not= i_0$ so that $n_i \equiv 0 \pmod 4$, we have 
\begin{align*}
\varepsilon(\pi_i \times \pi_F)^{\min\{d_i,2d\}}
&=
\left(\prod_{w_i>0}(-1)^{\max\{2w_i, j+2k-3\}+\max\{2w_i, j+1\}}\right)^{\min\{d_i,2d\}}
\\&=
\left(
\prod_{w_i>0}(-1)^{j+2k-3+2w_i}
\right)^{\min\{d_i,2d\}}
=
(-1)^{\frac{n_i}{2}(j+2k-3)\min\{d_i,2d\}} = 1.
\end{align*}
Since the cardinality of the set $K_i$ for $\psi'$ is the same as the one for $\psi_G$, 
we obtain the sign condition for $\pi_i$. 
Also, 
\begin{align*}
\prod_{j=1}^t \varepsilon(\pi_F \times \pi_j)^{\min\{2d,d_j\}}
&= 
\varepsilon(\pi_F \times \pi_{i_0})^{\min\{2d,d_{i_0}\}}
\\&=
(-1)^{\frac{n_{i_0}-1}{2}(j+2k-3)+\frac{1+(j+2k-3)}{2}+\frac{1+(j+1)}{2}}
=(-1)^{n+k}
\end{align*}
since $n_{i_0} \equiv 2n+1 \pmod 4$ and $j \equiv 0 \pmod 2$.
Hence the sign condition for $\pi_F$ is equivalent to $k \equiv n \pmod 2$. 
\end{proof}

The proofs of Theorem \ref{th.atobe1} (2) and Theorem \ref{th.atobe2} are similar. 
Let $G$ and $f$ be as in the statements, $\psi_G = \boxplus_{i=1}^t\pi_i[d_i]$ be the $A$-parameter of $G$. 
and $\pi_f$ be the irreducible unitary cuspidal automorphic self-dual representation of $\mathrm{PGL}_2(\mathbb{A}_\QQ)$
associated to $f$. 
We have to check the conditions (a)--(f) in Theorem \ref{MF} for 
\[
\psi' = \psi_G \oplus \pi_f[2d]
\]
with $2d = n-2$ in the proof of Theorem \ref{th.atobe1} (2). 
Only the condition (f) is non-trivial. 
\par

When $G$ is in Theorem \ref{th.atobe1} (2), 
we claim that 
$\psi_G$ is of the form $\pi_G[1]$ 
for some irreducible unitary cuspidal automorphic self-dual representation $\pi_G$ of $\mathrm{PGL}_5(\mathbb{A}_\QQ)$.
If not, by the condition (e) and by $n > 2$, 
we would have $\psi_G = \pi_1[1] \boxplus \mathbf{1}_{\GL_1(\mathbb{A}_\QQ)}[1]$.  
In this case, the sign condition fails since $|K_1| = 1$. 
Since 
\begin{align*}
\varepsilon(\pi_f \times \pi_G) 
&= 
(-1)^{(1+\max\{2k-n-1,2(k-1)\}) + (1+\max\{2k-n-1,2(k-n)\}) + \frac{1+(2k-n-1)}{2}}
\\&=
(-1)^{2(k-1)+(2k-n-1)+k-\frac{n}{2}} 
= (-1)^{k-\frac{n-2}{2}} = (-1)^{k+\frac{2d}{2}}, 
\end{align*}
the sign condition for $\psi'$ holds if and only if $k$ is even.
\par

When $G$ is in Theorem \ref{th.atobe2}, 
since 
\begin{align*}
\varepsilon(\pi_f \times \pi_j)^{\min\{2d,d_j\}}
&= 
\left\{
\begin{aligned}
&1 &&\text{if $j \ne i_0$}, \\
&(-1)^k &&\text{if $j=i_0$ in case (1)}, \\
&(-1)^{k+n} &&\text{if $j=i_0$ in case (2)}
\end{aligned}
\right.
\end{align*}
and the right-hand side of (f) for $\pi_f[2d]$ is $(-1)^d$, 
we can check the sign conditions.
This completes the proofs of Theorems \ref{th.atobe1} and \ref{th.atobe2}.

\section{An explicit pullback formula}
Suppose that $k >n_1+n_2+1$. Then for any $r \le \min(n_1,n_2)$ and a Hecke eigenform $f \in S_{\rho_r}(\varGamma^{(r)})$, the Klingen-Eisenstein series $[f]_{\rho_r}^{\rho_{n_1}}(Z,U)$ and $[f]_{\rho_r}^{\rho_{n_2}}(W,V)$ become holomorphic modular forms, and we obtain  more explicit results.
The proof of the following theorem is independent of B\"{o}cherer's argument. 
The proof here is a brute force but we still believe this way of calculation would be useful in some cases. For more conceptual description of a complete general
exact pullback formula, see \cite{Ibukiyamalaplace}.

\begin{theorem}
\label{th.explicit-pullback} 
Notation being as in Theorem \ref{th.pullback}, 
suppose that $l>0$. Then we have
\begin{align*}
(\DD_{\lambda,n_1,n_2}E_{n_1+n_2,k})\begin{pmatrix} Z & 0 \\ 0 & W \end{pmatrix}
&=
\sum_{r=2}^{\min(n_1,n_2)}c_r\sum_{j=1}^{d(r)}{D(k,f_{r,j})\over ( f_{r,j}, \ f_{r,j} )}[f_{r,j}]_{\rho_r}^{\rho_{n_1}}(Z)(U)
 [\theta f_{r,j}]_{\rho_r}^{\rho_{n_2}}(W)(V),
\end{align*}
where $c_r$ is a certain constant depending on $\DD_{\lambda,n_1,n_2}$.
\end{theorem}
This has been already proved by Kozima \cite{Kozima08} in more general setting. However he  does not give explicit values of $c_r$ in general. 
He gave in \cite{Kozima08} one strategy how to calculate the constant $c_{r}$ in the formula.
Actually this calculation is difficult to execute in general, but in the next section, 
we explain how to do this when $\DD_{\lambda}=Q_l(\p_Z,U,V)$, where $Q_l$ is 
defined by \eqref{defdiff} before Lemma \ref{lem.formulaQl}.
Though this calculation is not necessary for proving our main results, 
it is  interesting in its own right, and 
may have an application for computing exact standard $L$-values for $f \in S_{\rho_r}(\varGamma^{(r)})$.   
\subsubsection{Calculation of the constant}
To calculate the constant $c_r$, we follow Kozima's formulation in \cite{Kozima08}.
We fix $g=\begin{pmatrix} A & B \\ C & D \end{pmatrix}\in \Sp_{n_1+n_2}(\RR)$
and put $\delta=\det(CZ+D)$. 
To obtain $c_r$ for our differential operator $\bbD$ we need,  
it is a key to calculate 
$\bbD(\delta^{-k})$ explicitly.
For $Z\in \HH_{n_1+n_2}$, we write 
\[
Z=\begin{pmatrix} Z_{11} & Z_{12} \\
^{t}Z_{12} & Z_{22}  \end{pmatrix}
\qquad Z_{11}\in \HH_{n_1}, Z_{22}\in \HH_{n_2}, Z_{12}\in M_{n_1,n_2}(\CC).
\]
For an $m\times n_1$ matrix $U$ and $m\times n_2$ matrix $V$ of variables, 
define $\Bbb U$ as before and we write 
\[
\p_{U,Z}=(\p_{ij}^{U})=\Bbb U  Z ^{t}\Bbb U=\begin{pmatrix} UZ_{11}\, ^{t}U & UZ_{12}\, ^{t}V \\
V\,^{t}Z_{12}\, ^{t}U & VZ_{22}\,^{t}V \end{pmatrix}.
\]
We put $\Delta=(CZ+D)^{-1}C$ and 
\[
\Delta^{\Bbb U}=(\Delta_{ij})=\Bbb U (CZ+D)^{-1}C \,^{t}\Bbb U.
\]
It is well-known and easy to see that this is a symmetric matrix.
We write blocks of $\Delta$ by 
\[
\Delta^{\Bbb U}=\begin{pmatrix} \Delta^{(11)} & \Delta^{(12)} \\
^{t}\Delta^{(12)} & \Delta^{(22)}
\end{pmatrix},
\]
where $\Delta^{(ij)}$ are $m\times m$ matrices.
We consider a differential operator $\bbD=P(\p_Z)$ which satisfies Condition \ref{condition}.
The following Proposition is the same as Proposition 4.1 in p 241 in \cite{Kozima08} except for 
the point that the realization of the representation is slightly different. 
\begin{proposition}\label{kozimalemma}
There exists a polynomial $Q(X)$ in the components of $m\times m$ matrix $X$ such that 
\[
\bbD(\delta^{-k})=\delta^{-k}Q(\Delta^{(12)}).
\]
\end{proposition}
The key point is that the polynomial does not contain components of $\Delta^{(11)}$ and $\Delta^{(22)}$.

Since the realization of our polynomials and Kozima's are different, we explain the relation.
We fix $\lambda=(l,\ldots,l,0,\ldots,0)$ with $\depth(\lambda)=m$.
We denote by $u_i$ and $v_j$ the $i$-th row vector of $U$ and the $j$-th row vector of $V$ 
respectively and write $U=(u_i)$, $V=(v_j)$. 
Our polynomial is a polynomial in components of $T$, $U$ and $V$.
We write this as $P(T,U,V)$ 
to emphasize its dependence on $U$ and $V$.
We define polarization of $P$ for each rows $u_i$, $v_j$ in the usual way as follows.
We prepare $ml$ row vectors $\xi_{i\nu}$ ($1\leq i\leq m,1\leq \nu\leq l$) and 
other $ml$ vectors $\eta_{j\mu}$ ($1\leq j\leq m,1\leq \mu\leq l$) of variables. 
We write $\xi=(\xi_{i\nu})$ and $\eta=(\eta_{j\mu})$ for short.
For $P$ we define $P^*$ by 
\begin{multline*}
P^*(T,\xi,\eta) =l^{-2ml}\sum_{i,j=1}^{l}
\frac{\p^l}{\p c_{i1}\cdots \p c_{il}}\frac{\p^l}{\p d_{j1}\cdots \p d_{jl}}
P\biggl(T,\bigl(\sum_{\nu=1}^{l}c_{i\nu}\xi_{i\nu}\bigr),\bigl(\sum_{\mu=1}^{l}d_{jl}\eta_{jl}\bigr)\biggr)
\biggl|_{c_{i\nu}=d_{j\mu}=0}.
\end{multline*}
In other words, replacing $u_i$ and $v_j$ by $u_{i}=c_{i1}\xi_{i1}+\cdots+c_{il}\xi_{il}$ and 
$v_j=d_{j1}\eta_{i1}+\cdots+d_{jl}\eta_{jl}$ respectively in $P(T,U,V)$ and 
take coefficients of 
\[
\prod_{i,j=1}^{m}\prod_{\nu,\mu=1}^{l}c_{i\nu}d_{j\mu}.
\]
Here by definition, the polynomial $P^*$ is a multilinear polynomial in 
$\xi_{i\nu}T_{11}\,^{t}\xi_{j\mu}$, $\xi_{i\nu}T_{12}\,^{t}\eta_{j\mu}$, $\eta_{i\nu}T_{22}\,^{t}\eta_{j\mu}$
and it is homogeneous in the sense of Kozima.  
Since the polarization $P\rightarrow P^*$ commutes with 
$\Delta_{ij}(X)$ and $\Delta_{ij}(Y)$, the polynomial $P^*$ is also pluri-harmonic.
The action of $A_1\in \GL(n_1,\CC)$ and $A_2 \in \GL(n_2,\CC)$ is the same 
as $P$ since we have $u_iA_1=\sum_{\nu=1}^{l}c_{i\nu}x_{i\nu}A_1$ 
and $v_jA_2=\sum_{\mu=1}^{l}d_{j\mu}\eta_{j\mu}A_2$. 
So if we use $P^*$ instead of $P$, our formulation becomes Kozima's formulation.
So we can use Kozima's Proposition 4.1 in \cite{Kozima08}, and 
the interpretation in our case is given in Proposition \ref{kozimalemma} above.

Our next problem is to obtain $Q$ in Proposition \ref{kozimalemma}.

For any row vector $x$, $y$ of length $n$, we write 
\[
\p[x,y]=x\p_{Z}\,^{t}y=\sum_{1\leq i,j \leq n}\frac{1+\delta_{ij}}{2}x_iy_j \frac{\p}{\p z_{ij}}.
\]
The following formulas are given in Kozima \cite{Kozima08}.
(See also \cite{ibubook} for a precise proof.)
For any row vectors $u_1$, $u_2$, $u_3$, $u_4$ of length $n$ and any functions $f$, $g$ of $Z$, we have 
\[
\begin{array}{lll}
(D1) & \p[u_1,u_2](fg) = (\p[u_1,u_2]f)g+f(\p[u_1,u_2]g), \\ [1ex]
(D2) & \p[u_1,u_2](\delta^{-k})=-k\delta^{-k}u_1\Delta\, ^{t}u_2, \\ [1ex]
(D3) & \p[u_1,u_2](u_3\Delta\,^{t}u_4)=-\dfrac{1}{2}\bigl(u_1\Delta \,^{t}u_3)(u_2\Delta\,^{t}u_4)
+(u_1\Delta \,^{t}u_4)(u_2\Delta \,^{t}u_3)\bigr).
\end{array}
\]
By iterate use of these formulas, we can calculate 
the action of $P(\Bbb U\p_Z\,^{t}\Bbb U)$ for any polynomial $P$. But 
 actual calculation is a bit complicated. 
For our case, we have a following result. The rest of this section is 
devoted to prove this theorem.

\begin{theorem}\label{actiondelta}
We assume that $m=2$ and we define differential operator $\bbD_{l}
=Q_l(\p_Z,U,V)$ 
from $det^k$ to $\det^k\rho_{n_1,\lambda}\otimes \det^k\rho_{n_2,\lambda}$ 
for $\lambda=(l,l,0,\ldots,0)$ by \eqref{defdiff} before Lemma \ref{lem.formulaQl}. 
We write $\Bbb U(CZ+D)^{-1}C^{t}\Bbb U=(\Delta_{ij})_{1\leq i,j\leq 4}$. Then we have 
\[
\bbD_l(\delta^{-k})=\frac{1}{2^{2l}l!}(2k-3)_l(2k-1)_{2l}(\Delta_{13}\Delta_{24}-\Delta_{14}\Delta_{23})^{2l}.
\]
\end{theorem}
In the notation before, we have $\det(\Delta^{(12)})=\Delta_{13}\Delta_{24}-\Delta_{14}\Delta_{23}$.

A simple theoretical proof of Theorem \ref{actiondelta} is given in \cite{Ibukiyamalaplace}, but  
here we give an original proof prepared for  the present paper. 
This proof consists of complicated combinatorial brute force calculations, and 
we believe such alternative proof is not useless.

To make it readable, we first explain rough idea of calculation, and then 
we give actual calculation.
Define $F_i(T,U,V)$ as in Lemma \ref{lem.formulaQl}. We also define
\[
F_4(T,U,V)=\det(UT_{11}\,^{t}U), \quad F_5(T,U,V)=\det(VT_{22}\,^{t}V),
\]
where $T_{ij}$ are defined by \eqref{Tblocks} in section 5.
Then of course we have $F_2=F_4F_5$. 
We write 
\[
\calf_i=F_1(\p_Z,U,V), \qquad i=1,2,3,4,5. 
\]
Here $\calf_1$, $\calf_4$, $\calf_5$ are differential operators of order $2$ and 
$\calf_3$ of order $4$. 
We put
\begin{align*}
& \Delta^{\Bbb U} =\Bbb U (CZ+D)^{-1}C\,^{t}\Bbb U=(\Delta_{ij}),
\qquad C_1 = \Delta_{13}\Delta_{24}-\Delta_{14}\Delta_{23},\\
& C_4  = \Delta_{11}\Delta_{22}-\Delta_{12}^2 , \qquad 
C_5  = \Delta_{33}\Delta_{44}-\Delta_{34}^2, \qquad C_2 = C_4C_5, \qquad
C_3  =\det(\Delta^{\Bbb U}).
\end{align*}

Now our strategy of calculation is as follows.

\begin{enumerate}
\item For any $a$, $b$, $c$, 
we see easily that $\calf_1^a\calf_2^b\calf_3^c(\delta^{-k})$ 
is written as a product of $\delta^{-k}$ and a polynomial in $\Delta_{ij}$
by virtue of the formula (D1), (D2), (D3).  
But in fact, more strongly, we can show that it is a product of $\delta^{-k}$ 
and a polynomial in $C_1$, $C_2$, $C_3$ that is a weighted homogeneous polynomial of total 
degree $a+2b+2c$ if we put $\deg(C_1)=1$, $\deg(C_2)=2$, $\deg(C_3)=2$. 

\item Here we can show that $C_1$, $C_2$ and $C_3$ are algebraically independent for generic $g$ 
and $\Delta^{\Bbb U}$, so by virtue of Lemma \ref{kozimalemma}, we need only 
the coefficient of $C_1^{a+2b+2c}$ in $\calf_1^a\calf_2^b\calf_3(\delta^{-k})$
to describe $\bbD_l(\delta^{-k})$. So we calculate these coefficients 
for all $(a, b, c)$ and 
sum them up according to the explicit definition of $\bbD_l$. 
\end{enumerate}

Now we execute these calculations.
\begin{lemma}\label{F3}
For any non-negative integer, we have 
\[
\calf_3(\delta^{-k}C_3^r)=\frac{(k+r-1)(k+r)(2k+2r-3)(2k+2r-1)}{4}\delta^{-k}C_3^{r+1}.
\]
\end{lemma}

\begin{proof}
The operator $\calf_3$ is a differential operator of order $4$. 
There are many ways to prove Lemma \ref{F3}. One way is to use computer directly. Actually, 
by (D1), (D2), (D3), it is clear that 
\[
\calf_3(\delta^{-k}C_3^r)=\delta^{-k}C_3^{r-4}\times P(r,\Delta_{ij})
\]
where $P(r,\Delta_{ij})$ is a polynomial in $r$ of degree $4$ 
whose coefficients are polynomials in $\Delta_{ij}$ that do not 
depend on $r$. So the calculation for $r=0,\ldots, 4$ is enough
and executing these we obtain the above formula.
An alternative way is to specialize $\Delta$ to the case $n_1=n_2=2$ and $U=1_2$, $V=1_2$, $C=1_4$, $D=0_4$.
Then we have $\Delta^{\Bbb U}=Z^{-1}$ for $Z\in \HH_4$. 
Then under this specialization on $\Delta^{U}$ and $\calf_3$, we have 
\begin{align*}
\calf_3(\delta^{-k}C_3^r)& =\det(\p_Z)(\det(Z)^{-k-r})
\\ & =
(k+r)(k+r+\frac{1}{2})(k+r+1)(k+r+\frac{3}{2})\det(Z)^{-k-r-1}.
\\ &= \frac{(k+r)(2k+2r+1)(k+r+1)(2k+2r+3)}{4}\delta^{-k}C_3^{r+1}.
\end{align*}
The second equality is nothing but the Cayley type identity for symmetric matrices (\cite{cayley}).
Since the calculation is formally the same, we get Lemma \ref{F3},
\end{proof}

Next we consider $\calf_1$ and $\calf_2=\calf_4\calf_5$. 
Since $\calf_1$, $\calf_4$, $\calf_5$ are differential operators of order $2$, 
the operation of these on products of functions can be calculated if we have 
several fundamental operations on factors. To explain this, we assume that 
$\calf$ is a differential operator of homogeneous order $2$
and define a bracket $\{A,B\}_{\calf}$ by 
\[
\calf(AB)=\calf[A]B+A\calf[B]+\{A,B\}_{\calf}.
\]
We have $\{B,A\}_{\calf}=\{A,B\}_{\calf}$. 
For the operator $\p_1\p_2$ 
where $\p_1$ and $\p_2$ are differential operators of the first order,
we have 
\[
\{A,B\}_{\p_1\p_2}=(\p_1A)(\p_2B)+(\p_2A)(\p_1B),
\]
so for general $\calf$ of order $2$ and functions $A$, $B$, $C$, we have 
\[
\{A,BC\}_{\calf}=\{A,B\}_{\calf}C+\{A,C\}_{\calf}B.
\]
So for example, the operation of $\calf$ on a product $A_1\cdots A_{\nu}$ 
of functions $A_i$ can 
be calculated if we have $\calf(A_i)$ and $\{A_i,A_j\}_{\calf}$.
More generally, for $\delta^{-k}$, any functions $A$, $B$, $C$, and 
non-negative integers $p$, $q$, $r$ and for a differential operator $\calf$ 
of second order, 
we can give the following general formula by repeating the above consideration.
\begin{lemma}\label{product} 
\begin{align*}
& F(\delta^{-k}A^pB^qC^r) =F(\delta^{-k})A^pB^qC^r
 +pA^{p-1}B^qC^r\{\delta^{-k},A\}_{\calf}+qA^pB^{q-1}C^r\{\delta^{-k},B\}_{\calf}
\\ & \quad +rA^pB^qC^{r-1}\{\delta^{-k},C\}_{\calf}
 +\delta^{-k}\biggl(pA^{p-1}B^qC^rF(A)+qA^pB^{q-1}C^rF(B)
 +rA^pB^qC^{r-1}F(C)
\\ & \quad + pqA^{p-1}B^{q-1}C^r\{A,B\}_{\calf}+qrA^{p}B^{q-1}C^{r-1}\{B,C\}_{\calf}
+pqA^{p-1}B^qC^{r-1}\{A,C\}_{\calf}
\\ & \quad + 
\frac{p(p-1)}{2}A^{p-2}B^qC^r\{A,A\}_{\calf}
\frac{q(q-1)}{2}A^pB^{q-2}C^r\{B,B\}_{\calf}
+\frac{r(r-1)}{2}A^pB^qC^{r-2}\{C,C\}_{\calf}\biggr).
\end{align*}
\end{lemma}

So for $\delta^{-k}$, $C_1$, $C_2$, $C_3$ (and $C_4$, $C_5$ when necessary), 
and $\calf=\calf_1$, $\calf_4$, $\calf_5$, 
we list up all these fundamental data below.
For $i=1$, $4$, $5$, we write $\{*,*\}_{\calf_i}=\{*,*\}_i$.

Next we consider $\calf_1$ and $\calf_2$. 
We give {\it fundamental formulas} below.
\begin{align*}
& \calf_1(\delta^{-k}) =\frac{k(2k-1)}{2}C_1\delta^{-k}, 
\quad \calf_1(C_1)  = \frac{1}{2}C_2, \quad
\calf_1(C_2)  =3C_1C_2, \quad 
\calf_1(C_3)  =\frac{1}{2}C_1C_3, \\
&\{\delta^{-k},C_1\}_1 = \frac{k}{2}(3C_1^2+C_2-C_3)\delta^{-k},\quad
\{\delta^{-k},C_2\}_1  = 4kC_1C_2\delta^{-k}, \quad
\{\delta^{-k},C_3\}_1  = 2kC_1C_3 \delta^{-k}, \\
& 
\{C_1,C_1\}_1  =\frac{1}{2}(2C_1^2+2C_2-C_3), \quad 
\{C_1,C_2\}_1  = C_2(3C_1^2+C_2-C_3) , \\
& \{C_1,C_3\}_1  = \frac{1}{2}C_3(3C_1^2+C_2-C_3), \quad
\{C_2,C_2\}_1  = 8C_1C_2^2, \qquad
\{C_2,C_3\}_1  = 4C_1C_2C_3, \\
& 
\{C_3,C_3\}_1  = 2C_1C_3^2, \quad 
 \calf_4(\delta^{-k})  = \frac{k(2k-1)}{2}C_4\delta^{-k}, \quad
\calf_4(C_1)  = \frac{1}{2}C_4C_1,
\\ & \calf_4(C_2)  = \frac{1}{2}C_4(4C_1^2+2C_2-C_3), \quad
\calf_4(C_3)  = \frac{1}{2}C_4C_3, \quad 
\{\delta^{-k},C_1\}_4  = 2kC_4C_1\delta^{-k}, \\
& \{\delta^{-k},C_2\}_4  = kC_4(C_1^2+3C_2-C_3)\delta^{-k}, \quad
\{\delta^{-k},C_3\}_4  = 2kC_4C_3\delta^{-k}, \quad
\{C_1,C_1\}_4  = 2C_4C_1^2, \\
&\{C_1,C_2\}_4  = C_4C_1(C_1^2+3C_2-C_3), \quad
\{C_1,C_3\}_4  = 2C_4C_1C_3, \\
& 
\{C_2,C_2\}_4  = 2C_4C_2(2C_1^2+2C_2-C_3), \\
& 
\{C_2,C_3\}_4  = C_4C_3(C_1^2+3C_2-C_3), \quad
\{C_3,C_3\}_4  = 2C_4C_3^2 \quad 
\calf_5(\delta^{-k})  = \frac{k(2k-1)}{2}C_5,\\
& \calf_5(C_1)  =\frac{1}{2}C_5C_1, \quad 
\calf_5(C_2)  = \frac{1}{2}C_5(4C_1^2+2C_2-C_3), \quad 
\calf_5(C_3)  = \frac{1}{2}C_5C_3,\\
& \calf_5(C_4)  = \frac{1}{2}(2C_1^2-C_2+C_3),\quad 
\{\delta^{-k},C_1\}_5  = 2kC_5C_1\delta^{-k},\\
& \{\delta^{-k},C_2\}_5  = kC_5(C_1^2+3C_2-C_3)\delta^{-k},\quad 
\{\delta^{-k},C_3\}_5  = 2kC_5C_3\delta^{-k},\\
& 
\{\delta^{-k},C_4\}_5  = k(C_1^2+C_2-C_3)\delta^{-k}, \quad 
\{C_1,C_1\}_5  = 2C_5C_1^2,\quad 
\{C_1,C_2\}_5  = C_5C_1(C_1^2+3C_2-C_3),\\
& 
\{C_1,C_3\}_5  = 2C_5C_1C_3, \quad 
\{C_2,C_2\}_5  = 2C_5C_2(2C_1^2+2C_2-C_3),\\
& 
\{C_2,C_3\}_5  = C_5C_3(C_1^2+3C_2-C_3), \quad 
\{C_3,C_3\}_5  = 2C_5C_3^2, \\
& 
\{C_1,C_4\}_5  = C_1(C_1^2+C_2-C_3), \quad 
\{C_2,C_4\}_5  = C_2(3C_1^2+C_2-C_3),
\\ & 
\{C_3,C_4\}_5  = C_3(C_1^2+C_2-C_3).
\end{align*}

\begin{lemma}
For a generic $g$, $U$, $V$ such that $\Delta_{ij}$ are algebraically 
independent, three variables $C_1$, $C_2$ and $C_3$ are 
algebraically independent.
\end{lemma}

\begin{proof}
Assume that 
\[
\sum_{p,q,r}C(p,q,r)C_1^pC_2^qC_3^q=0
\]
for some constants $C(p,q,r)$ where the degree of $C_1$ is the
smallest among such relations．If we put $\Delta_{14}=\Delta_{24}=0$, 
then we have $C_1=0$, and $C_3=\Delta_{13}^2\Delta_{22}\Delta_{44}
+\cdots$. Then, since $C_2$ does not contain any $\Delta_{13}$, 
this means that $C(0,q,r)=0$ for any $q$, $r$. 
So we may assume that 
\[
C_1\sum_{p\geq 1,q,r}C(p,q,r)C_1^{p-1}C_2^qC_3^r=0.
\]
Since the polynomial ring in $\Delta_{ij}$ is UFD, 
we have $\sum_{p\geq 1,q,r}C(p,q,r)C_1^{p-1}C_2^qC_3^r=0$.
This contradicts to the assumption.
\end{proof}

\begin{lemma}
For any integers $a$, $b$, $c$, $p$, $q$, $r\geq 0$, 
there exists a polynomial $P(C_1,C_2,C_3)$ 
such that 
\[
\calf_1^a\calf_2^b\calf_3^c(\delta^{-k})=\delta^{-k}P(C_1,C_2,C_3).
\]
\end{lemma}

\begin{proof}
By Lemma \ref{F3}, 
it is enough to assume that LHS is 
$\calf_1^a\calf_2^b(\delta^{-k}C_3^{r+1})$.
Here we have $\calf_2=\calf_5\calf_4$.
By the fundamental formulas, we see that for $i$, $j$ 
with $1\leq i<j\leq 3$, any of  
$\calf_4(\delta^{-k})$, $\calf_4(C_i)$, $\{\delta^{-k},C_i\}_4$, 
$\{C_i,C_j\}_4$ 
are $\delta^{-k}C_4$ times a polynomial in $C_1$, $C_2$, $C_3$, so 
by Lemma \ref{product}, we see that 
\[
\calf_4(\delta^{-k}C_1^pC_2^qC_3^r)
=
\delta^{-k}C_4P_1(C_1,C_2,C_3)
\]
for some polynomial $P_1(x,y,z)$ in three variables.
We have 
\begin{align*}
\calf_5(\delta^{-k}C_4P_1(C_1,C_2,C_3))
& = 
C_4\calf_5(\delta^{-k}P_1(C_1,C_2,C_3))
+\calf_5(C_4)\delta^{-k}P_1(C_1,C_2,C_3)
\\
& \quad +\{C_4,\delta^{-k}P_1(C_1,C_2,C_3)\}_5
\end{align*}
Here by the same reason as before, the first term is equal to 
\[
C_4C_5\delta^{-k}P_2(C_1,C_2,C_3)=\delta^{-k}C_2P_2(C_1,C_2,C_3)
\]
for some polynomial $P_2(x,y,z)$. 
Since $\delta^{k}\{C_4,\delta^{-k}\}_5$, 
$\{C_4,C_i\}_5$ for $i=1$, $2$, $3$ 
and $\calf_5(C_4)$ are polynomials in $C_1$, $C_2$, $C_3$, 
we see that 
$\calf_2(\delta^{-k}C_1^pC_2^qC_3^r)$ is $\delta^{-k}$ times 
a polynomial in $C_1$, $C_2$, $C_3$. We can show inductively 
that the same is true for 
$\calf_2^{b}$ and $\calf_1^a\calf_2^b$, so we prove Lemma.
\end{proof}

By Lemma \ref{kozimalemma}, we need only the power of $C_1$ part 
in the polynomial in 
$C_1$, $C_2$, $C_3$, so we will study that.

We denote by $\frkC_3=C_3\CC[C_1,C_2,C_3]$
and $\frkC_{23}=(C_2,C_3)\CC[C_1,C_2,C_3]$ 
the ideals of $\CC[C_1,C_2,C_3]$ generated by 
$C_3$, and by $C_2$ and $C_3$, respectively.
We have the following result.

\begin{proposition}\label{C1part}
(i) If $c\geq 1$, then we have
\[
\delta^{k}
\calf_1^a\calf_2^b\calf_3^c(\delta^{-k})\in \frkC_3.
\]
(ii) When $c=0$ in the above, we have  
\[
\delta^{k}
\calf_1^a\calf_2^b(\delta^{-k}) \equiv 
\frac{1}{2^a}b!(k)_{2b}\bigl(k-\frac{1}{2}\bigr)_b(k+2b)_a(2k+2b-1)_aC_1^{a+2b}\bmod \frkC_{23}.
\]
\end{proposition}

To prove Proposition \ref{C1part}, we prepare several lemmas.

\begin{lemma}\label{F2part}
(i) For any integer $r\geq 1$, we have 
\[
\delta^k\calf_2(\delta^{-k}C_1^pC_2^qC_3^r) \in \frkC_3.
\]
(ii) For any integer $b\geq 0$, we have 
\[
\delta^{k}\calf_2^b(\delta^{-k})\equiv \sum_{p=0}^{b}p!\binom{b}{p}^2
(k)_{2b}\bigl(k-\frac{1}{2}\bigr)_b\bigl(k+p-\frac{1}{2}\bigr)_{b-p}C_1^{2p}C_2^{b-p}
\bmod \frkC_3.
\]
\end{lemma}

\begin{proof}
We have 
\begin{equation*}
\calf_4(\delta^{-k}C_1^pC_2^qC_3^r) =\delta^{-k}C_1^pC_2^q\calf_4(C_3^r)
+C_3^r\calf_4(\delta^{-k}C_1^pC_2^q)
+\{\delta^{-k}C_1^pC_2^q,C_3^r\}_4.
\end{equation*}
Since $\calf_4(C_3)\in C_4\frkC_3$, $\{C_i,C_3\}_4\in C_4\frkC_3$ for $i=1$, $2$, $3$, and 
$\{\delta^{-k},C_3\}_4 \in \delta^{-k}C_4\frkC_3$, 
we see that $\calf_4(\delta^{-k}C_1^pC_2^qC_3^r)\in \delta^{-k}C_4\frkC_3$. 
In the same way, we can show that 
$\calf_5(\delta^{-k}C_4\frkC_3)\subset \frkC_3$.
So we prove (i). 
Next we prove (ii) by induction. 
By direct calculation, we have 
\begin{equation*}
\delta^{k}\calf_2(\delta^{-k})=  
\frac{k(2k-1)(k+1)}{2}C_1^2+\frac{k(2k-1)^2(k+1)}{4}C_2-\frac{k(2k-1)^2}{4}C_3.
\end{equation*}
This is nothing but the case $b=1$ of Lemma \ref{F2part} (2).
Next we calculate $F_2(\delta^{-k}C_1^pC_2^q)$ in order 
to calculate $\calf_2^{b}(\delta^{-k})$ 
inductively. 
We have 
\begin{align*}
 F_4(\delta^{-k}C_1^pC_2^q)
 & =  F_4(\delta^{-k})C_1^pC_2^q+\{\delta^{-k},C_1^pC_2^q\}_4+\delta^{-k}F_4(C_1^pC_2^q)
\\ & =   F_4(\delta^{-k})C_1^pC_2^q+pC_1^{p-1}C_2^q\{\delta^{-k},C_1\}_4+qC_1^pC_2^{q-1}\{\delta^{-k},C_2\}_4
\\ & \quad + \delta^{-k}(F_4(C_1^p)C_2^q+\{C_1^p,C_2^q\}_4+C_1^pF_4(C_2^q))
\\ &  =  F_4(\delta^{-k})C_1^pC_2^q+pC_1^{p-1}C_2^q\{\delta^{-k},C_1\}_4+qC_1^pC_2^{q-1}\{\delta^{-k},C_2\}_4
\\ & \quad + \delta^{-k}\biggl(pF_4(C_1)C_1^{p-1}C_2^q+\frac{p(p-1)}{2}\{C_1,C_1\}_4C_1^{p-2}C_2^q
\\ & \quad +pqC_1^{p-1}C_2^{q-1}\{C_1,C_2\}_4+qC_1^pC_2^{q-1}F_4(C_2)
 +\frac{q(q-1)}{2}C_1^pC_2^{q-2}\{C_2,C_2\}_4 \biggr).
\end{align*}
So by the fundamental formulas, we have 

\begin{align}\label{F4part}
F_4(\delta^{-k}C_1^pC_2^q)  & =  
\delta^{-k}C_4C_1^pC_2^{q-1}
 \times \biggl( q(k+p+2q)C_1^{2}
+\frac{(k+p+2q)(2k+2p+2q-1)}{2}C_2 
\\ &  \qquad -\frac{1}{2}q(2k+2p+2q-1)C_3\biggr). \notag
\end{align}
In the same way, we have 
\begin{align}\label{F5part}
F_5(\delta^{-k}C_4C_1^pC_2^q)
& = \delta^{-k}C_1^pC_2^q
\times \biggl(
(q+1)(k+p+2q+1)C_1^2
\\ & \quad +\frac{(k+p+2q+1)(2k+2p+2q-1)}{2}C_2  
-\frac{(q+1)(2k+2p+2q-1)}{2}C_3\biggr). \notag
\end{align}
By applying \eqref{F4part}, we have 
\begin{align}
& \delta^{k}F_2(\delta^{-k}C_1^pC_2^q) \equiv  
q(k+p+2q)F_5(\delta^{-k}C_4C_1^{p+2}C_2^{q-1}) \\
& \qquad + \frac{(k+p+2q)(2k+2p+2q-1)}{2}F_5(\delta^{-k}C_4C_1^pC_2^q)
\bmod \frkC_3, 
\end{align}
and by \eqref{F5part}, we have 
\begin{align}\label{F2C1C2}
& \delta^kF_2(\delta^{-k}C_1^pC_2^q) = 
q^2(k+p+2q)(k+p+2q+1)C_1^{p+4}C_2^{q-1}
\\ & +(k+p+2q)(k+p+2q+1)\bigl((4p+2)+(2p^2+pq+p+q-\frac{1}{2})\bigr)C_1^{p+2}C_2^q
\notag \\ &
+\frac{1}{4}(k+p+2q)(k+p+2q+1)(2k+2p+2q-1)^2C_1^pC_2^{q+1}. \notag
\end{align}

Now assume that the claim (ii) holds for $b-1$. Then applying \eqref{F2C1C2}, 
we have (ii). This is of course a straight forward calculation, but this is a bit complicated so 
we give a precise proof. 
For the sake of simplicity, we put 
\[
x_{p,b}=p!\binom{b}{p}^2(k)_{2b}\bigl(k-\frac{1}{2}\bigr)_b\bigl(k+p-\frac{1}{2}\bigr)_{b-p}. 
\]
To see the coefficient of $C_1^{2p}C_2^{b-p}$ in $\calf_2(\calf_2^{b-1}(\delta^{-k}))$, we should see 
the linear combination of $\delta^{-k}C_1^{2p-4}C_2^{b-p-1}$, $\delta^{-k}C_1^{2p-2}C_2^{b-p}$ and 
$\delta^{-k}C_1^{2p}C_2^{b-1-p}$ in $\calf_2^{b-1}(\delta^{-k})$ and apply $\calf_2$ on it and see the 
coefficient at $C_1^{2p}C_2^{b-p}$. We compare each term with $x_{p,b}$. 
First the coefficient at $C_1^{2p}C_2^{b-p}$ of $\delta^{k}\calf_2(\delta^{-k}C_1^{2p-4}C_2^{b-p+1})$
is given by 
\[
(k+2b-2)(k+2b-1)(b-p+1)^2.
\]
We must multiply $x_{b-1,p-2}$ to this. The product is given by 
\begin{align*}
& \frac{(b-1)!^2}{(p-2)!(b-p)!^2}
(k)_{2b}(k-\frac{1}{2})_{b-1}(k+p-\frac{5}{2})(k+p-\frac{3}{2})\cdots(k+b-\frac{5}{2}) 
\\ & 
= e_1x_{p,b}
\end{align*}
where we put 
\[
e_1=\frac{p(p-1)(k+p-\frac{5}{2})(k+p-\frac{3}{2})}{b^2(k+b-\frac{3}{2})^2}.
\]
Secondly, the coefficient at $C_1^{2p}C_2^{b-p}$ of $\delta^{k}\calf_2(\delta^{-k}C_1^{2p-2}C_2^{b-p})$ is given by 
\[
\frac{1}{2}(k+2b-2)(k+2b-1)(2(2b-2p+1)k+(4b^2-6b-4p^2+10p-5)).
\]
We must multiply $x_{b-1,p-1}$ to this. The result is $e_2x_{b,p}$ where we put 
\[
e_2=\frac{p(k+p-\frac{3}{2})((2b-2p+1)k+(2b^2-3b-2p^2+5p-\frac{5}{2}))}
{b^2(k+b-\frac{3}{2})^2}.
\]
Finally, the coefficient at $C_1^{2p}C_2^{b-p}$ of $\delta^k\calf_2(\delta^{-k}C_1^{2p}C_2^{b-1-p})$
is given by 
\[
\frac{1}{4}(k+2b-2)(k+2b-1)(2k+2b+2p-3)^2. 
\]
Multiplying $x_{b-1,p}$ to this, we have $e_3x_{b,p}$ where we put 
\[
e_3=\frac{(b-p)^2(k+b+p-\frac{3}{2})^2}{b^2(k+b-\frac{3}{2})^2}. 
\]
Since we easily see $c_1+c_2+c_3=1$, we prove (ii).  
\end{proof}

\begin{lemma}\label{F1part}
(i) For any integer $r\geq 1$, we have 
\[
\delta^{k}\calf_1(\delta^{-k}C_1^pC_2^qC_3^r) \in \frkC_3.
\]
(ii) 
For any integer $q\geq 1$, we have 
\[
\delta^k\calf_1(C_1^pC_2^q)\in \frkC_{23}.
\]
(iii) For any non-negative integers $a$, $p$, we have 
\begin{align*}
\delta^{k}F_1(\delta^{-k}C_1^p) & = 
\frac{(k+p)(2k+p-1)}{2}C_1^{p+1}
+\frac{p(k+p)}{2}C_1^{p-1}C_2-\frac{p(2k+p-1)}{4}C_1^{p-1}C_3, \\ 
\delta^k\calf_1^a(\delta^{-k}C_1^p) & \equiv 
\frac{(k+p)_a(2k+p-1)_a}{2^a}C_1^{a+p}  \bmod \frkC_{23}. 
\end{align*}
\end{lemma}

\begin{proof}
Since $\calf_1(C_3)$, $\delta^{k}\{\delta^{-k},C_3\}_1$, $\{C_i,C_3\}_1$ are in $\frkC_{3}$ for any $i=1$, $2$, $3$, 
the assertion (i) is clear.
Since $\calf_1(C_2)$, $\delta^{k}\{\delta^{-k},C_2\}_1$, $\{C_i,C_2\}_1$ are in $\frkC_{23}$ for any $i=1$, $2$, 
the assertion (ii) is clear. For (iii), the first assertion is obtained by direct calculation.
The second assertion is shown by induction by using (i) and (ii).
\end{proof}

\begin{proof}[Proof of Proposition \ref{C1part}]
The assertion (i) is clear from Lemma \ref{F3} and 
Lemma \ref{F2part} (i), \ref{F1part} (i). The assertion (ii) is obvious by Lemma \ref{F2part} (ii) and 
\ref{F1part} (iii). So Proposition \ref{C1part} is proved.
\end{proof}

In order to prove Theorem \ref{actiondelta}, we fix a non-negative integer $l$. 
In order to give $\bbD_l(\delta^{-k})$, we must sum up each contribution of 
$\calf_1^a\calf_2^b\calf_3^c(\delta^{-k})$ such that $a+2b+2c=l$. 
By Proposition \ref{kozimalemma} and Proposition \ref{C1part} (i), the term 
with $c\geq 1$ 
does not contribute to the final sum. So we assume $c=0$ and $a+2b=l$. 
We put 
\[
q_{a,b}=\frac{(-1)^b}{a!}(k)_{l}(2k+2b-1)_a(k-\frac{1}{2})_b(k-\frac{3}{2})_{a+b}.
\]
By Proposition \ref{C1part}, 
we see that this is the contribution from $\calf_1^a\calf_2^b(\delta^{-k})$ times the coefficient of $F_1^aF_2^b$  
in the definition of $\bbD_l$, noting that 
\[
(k)_{2b}(k+2b)_a=(k)_l.
\]
What we want to calculate is $q_0=\sum_{a+2b=l}q_{a,b}$. 
Denote by $[l/2]$ the maximum integer with does not exceed $l/2$. 
To calculate $Q$ inductively, for any integer $b$ such that 
$0\leq b\leq [l/2]$, we put 
\[
q_{b}=\sum_{0\leq b\leq b_0\leq [l/2]}q_{l-2b_0,b_0}.
\]

\begin{lemma}\label{finalrecursion}
The notation being as above, for any $b$ with $0\leq b\leq [l/2]$, we have 
\begin{equation}\label{Qb}
q_b=\frac{2(k+b-1)(2k+2l-2b-3)}{(2k+l-2)(2k+l-3)}q_{a,b}.
\end{equation}
\end{lemma}

\begin{proof}
We prove this by induction from $b=[l/2]$ to $b=0$. 
First we show this for $b=[l/2]$. By definition, we have 
$q_b=q_{l-2b,b}$, so the problem is if the coefficient 
of RHS of \eqref{Qb} is $1$. 
For even $l$, we have $[l/2]=l/2$ and for odd $l$ we have $[l/2]=(l-1)/2$, 
and in both cases we have 
\[
2(k+[l/2]-1)(2k+2l-2[l/2]-3)=(2k+l-2)(2k+l-3).
\]
So the assertion is clear for $b=[l/2]$.
Now assume that the claim holds for some $b\leq [l/2]$
and we calculate $q_{b-1}$. 
Calculating the ratio $q_{a,b}/q_{a+2,b-1}$, we have 
\begin{align*}
q_{b-1} & =q_b+q_{a+2,b-1}
 =
q_{a+2,b-1}\left(1-\frac{(l-2b+1)(l-2b+2)}{(2k+l-2)(2k+l-3)}\right)
\\ & = 
\frac{2(k+b-2)(2k+2l-2b-1)}{(2k+l-2)(2k+l-3)}q_{l-2b+2,b-1}.
\end{align*}
So the claim holds also for $b-1$. 
\end{proof}

\begin{proof}[Proof of Theorem \ref{actiondelta}]
By Lemma \ref{finalrecursion}, 
we have 
\[
q_0=\frac{2(k-1)(2k+2l-3)}{(2k+l-2)(2k+l-3)}
\times \frac{(k)_l(2k-1)_l(k-\frac{3}{2})_l}{l!}.
\]
Here we have 
\[
(k)_{l}(k-\frac{3}{2})_l(2k+2l-3)=2^{-2l}(2k-3)(2k-1)_{2l},
\]
and 
\[
2(k-1)\frac{(2k-1)_l}{(2k+l-2)(2k+l-3)}=
(2k-2)(2k-1)_{l-2},
\]
where we define $(x)_{-1}=1/(x-1)$ and $(x)_{-2}=1/(x-1)(x-2)$. 
So we have 
\[
q_0=\frac{1}{2^{2l}l!}(2k-3)_l(2k-1)_{2l}.
\]
So we prove Theorem \ref{actiondelta}.
\end{proof}

\subsubsection{Explicit pullback formula}
Based on the results in the last subsection, 
we can write down the pullback formula for the differential 
operator $\bbD_l$ in general for any $n=n_1+n_2$ with 
$2\leq \min(n_1,n_2)$ and $\lambda=(l,l,0,\ldots,0)$.
We use Kozima's formula in \cite{Kozima08} p.247.
We define $Q(X)$ as in Proposition \ref{kozimalemma} and Theorem 
3.2. Then by Theorem \ref{actiondelta}, we have
\[
Q(\Delta^{(12)})=q_0C_1^{2l}.
\] 
For a $n\times n$ symmetric matrix, we write a block decomposition as 
\[
T=\begin{pmatrix} T_{11} & T_{12} \\ ^{t}T_{12} & T_{22} \end{pmatrix}
\]
where $T_{11}$ is $n_1\times n_1$ and $T_{22}$ is $n_2\times n_2$. 
we define a polynomial $\frkQ(T)$ in $t_{ij}$ by  
\[
\frkQ(T)=Q(UT_{12}\,^{t}V).
\]
Then we have $\frkQ((CZ+D)^{-1}C)=Q(\Delta^{(12)})$.
For any $r\leq \min (n_1,n_2)$, we define $n_1\times n_2$ 
matrix by 
$\begin{pmatrix} 1_r & 0 \\ 0 & 0 \end{pmatrix}$, 
and by abuse of language, we denote this also by $1_r$. 
As in Kozima, we are allowed to write $\frkQ(T)$ for $T_{12}=1_r$ as  
\[
\frkQ\begin{pmatrix} * & 1_r \\ * & * \end{pmatrix},
\]
not specifying $*$, since this does not depend on $*$ part by definition.
Now we put 
\[
R_r=\sum_{1\leq i<j\leq r}
(u_{1i}u_{2j}-u_{1j}u_{2i})(v_{1i}v_{2j}-v_{1j}v_{2i}).
\]
Then for $\lambda=(l,l,0,\ldots,0)$, we have 
\[
\frkQ\begin{pmatrix}* & 1_r \\
* & * \end{pmatrix}=q_0\times R_r^l.
\]
We consider two isomorphic realizations of the representation 
$\det^k \rho_{r,\lambda}$, one is on the space generated by bideterminants 
in $u_{ij}$ with $i=1$, $2$, $j\leq r$ and 
the other is on the space generated by bideterminants in $v_{ij}$ 
with $i=1$, $2$, $j\leq r$.
We denote the former representation space by $V_*^r$ and 
the latter by $V_{*r}$.
We identify these representation spaces of $\GL_r(\CC)$ on 
$U$ variables and $V$ variables 
by mapping $u_{ij}$ to $v_{ij}$. 
For $v_*\in V_{*r}$, we denote by $v^*$ the 
corresponding element in $V_r^*$.
Now we define 
\[
S_r=\{S\in M_r(\CC) \mid S=\,^{t}S, 1_r-S\overline{S}>0\} ,
\]
where $*>0$ means that the matrix is positive definite.
We define a linear map from $V_{r*}$ to $V_{r*}$ by 
\[
\psi(v_*) =
\int_{S_r}\langle 
\rho_{r}(1_r-\overline{S}S)v^*,\frkQ\begin{pmatrix} * & 1_r \\ * & * 
\end{pmatrix}\rangle \det(1_r-\overline{S}S)^{-r-1}dS,
\]
where $\rho_r=\det{}^{k}\rho_{r,\lambda}$ and 
$dS=\prod_{h\leq j}dx_{hj}dy_{hj}$ for 
$S=X+iY$ with $X=(x_{hj})$, $Y=(y_{hj}) \in M_r(\RR)$.
Then we know that $\psi(v_*)=\varphi v_*$ for some 
constant $\varphi$. 
Then by the result of \cite{Kozima08} in p.247 , we have 
\[
c_r=2^{r(r+1)+1-(rk+2l)}i^{rk+2l}\cdot \varphi.
\]
Here $\varphi$ obviously depends on the inner product.
We can explicitly calculate the inner product $\langle *, * \rangle_0$ 
defined before for the necessary quantity.
\begin{lemma}
We have   
\[
\langle v^*, R_r^l\rangle_{0}
=
(l+1)!l! v_*.
\]
\end{lemma}

\begin{proof}
For the proof, we quote \cite[Theorem 2.16]{cayley}.
Let $X=(x_{ij})$ be an $m\times r$ matrix of variables with $m\leq r$, and let 
$\p=(\frac{\p}{\p x_{ij}})$. Let $A$ and $B$ be $r\times m$ constant matrices. 
Then that theorem gives the following general formula. 
\[
\det(\p\times B)(\det(XA))^{s}=\det(^{t}AB)(s)_m \det(XA)^{s-1}.
\]
So if we put $m=2$ and assume $X$ and $^{t}A$ to be $2\times r$ 
matrices consisting of the first $r$ columns of $U$ and $V$ 
respectively, and $B=\begin{pmatrix} 1_2 \\ 0_{r-2,2}\end{pmatrix}$, then  the above formula means 
\[
\frac{\p^2R_r^l}{\p u_{11}\p_{22}}-\frac{\p^2R_r^l}{\p u_{12}\p u_{21}}=l(l+1)(v_{11}v_{22}-v_{12}v_{21})R_r^{l-1}.
\]
In the same way we have 
\[
\frac{\p R_r^l}{\p u_{1p}\p u_{2q}}-\frac{\p R_r^l}{\p u_{1q}\p u_{2p}}=l(l+1)(v_{1p}v_{2q}-v_{1q}v_{2p})R_r^{l-1}.
\]
So iterating these operations $l$ times, we have the assertion.

\end{proof}

By the above results, it is natural to use here the inner product 
\begin{equation}\label{innerl}
\langle *,*\rangle_l=\langle *,* \rangle_0/(l+1)!l!.
\end{equation}
The remaining part is the following integral
\[
I_r=\int_{S_r}\rho_r(1_r-\overline{S}S)\det(1_r-\overline{S}S)^{-r-1}dS, 
\]
where $\rho_r=\det^k \rho_{r,\lambda}$, to which 
a dominant integral weight
$(k+l,k+l,k,\ldots,k)$ corresponds.
By Kozima \cite[Lemma 2]{Kozima02} (and also by \cite{Hua63}, 
\cite{Bo85a}, \cite{BSY92}, \cite{takayanagi})
we have 
\[
I_r=\frac{2^r\pi^{r(r+1)/2}}
{\prod_{\nu=2}^{4}(2k+2l-\nu)\prod_{\mu=1}^{2}\prod_{\nu=3}^{r}
(2k+l-\mu-\nu)\prod_{3\leq \mu\leq \nu\leq r}
(2k-\mu-\nu)}.
\]

\begin{theorem}\label{th.constant}
We assume that $k$ is even with $k>n+1$. 
Assumption being the same as above, taking the inner product $\langle *,* \rangle_l$ as 
in \eqref{innerl}, the constants $c_r$ for 
$2\leq r\leq \min(n_1,n_2)$ in Theorem \ref{th.pullback} are given by 
\[
c_r=\frac{2^{(r+1)^2-(rk+2l)}(-1)^{rk/2+l}\pi^{r(r+1)/2}(2k-3)_l(2k-1)_{2l}}
{2^{2l}l!\prod_{\nu=2}^{4}(2k+2l-\nu)\prod_{\mu=1}^{2}
\prod_{\nu=3}^{r}
(2k+l-\mu-\nu)\prod_{3\leq \mu\leq \nu\leq r}
(2k-\mu-\nu)}.
\]
In particular, we have 
\[
c_2=\frac{2^{9-2(k+2l)}(-1)^{k+l}\pi^3(2k-3)_l(2k-1)_{2l-3}}
{l!}.
\]
\end{theorem}

We note here that since we assumed that $k$ is even, 
the number $rk/2$ is an integer.

\end{document}